\documentclass[11pt,a4paper,twoside]{article}

\usepackage{amssymb,amscd,amsfonts,amsbsy}
\usepackage{enumerate}
\usepackage{amsmath,amsthm}
\usepackage{mathrsfs}
\usepackage{epsf,epsfig,esint}
\usepackage{pdfsync}
\usepackage[all]{xy}

\usepackage[colorlinks=true,citecolor=magenta,linkcolor=cyan,backref]{hyperref}
\AtBeginDocument{}

\numberwithin{equation}{section}

\newtheorem{proposition}{Proposition}[section]
\newtheorem{theorem}[proposition]{Theorem}
\newtheorem{lemma}[proposition]{Lemma}
\newtheorem{corollary}[proposition]{Corollary}

\theoremstyle{definition}
\newtheorem{definition}[proposition]{Definition}
\newtheorem{remark}[proposition]{Remark}

\usepackage{color}

\definecolor{gr}{rgb}   {0., 0.8, 0. } 
\definecolor{bl}{rgb}   {0., 0.5, 1. } 
\definecolor{mg}{rgb}   {0.7, 0., 0.7}

\DeclareMathOperator*{\esssup}{ess\,sup}

\title{On existence and uniqueness for non-autonomous parabolic Cauchy 
problems with rough coefficients}
\author{Pascal Auscher\,\thanks{Laboratoire de Math\'ematiques d'Orsay, 
Univ. Paris-Sud, CNRS, Universit\'e Paris-Saclay, 91405 Orsay, France.}\,
\and Sylvie Monniaux\,\thanks{Aix-Marseille Universit\'e, CNRS, Centrale Marseille, 
I2M, UMR 7373, 13453 Marseille, France.}\, 
\and Pierre Portal\,\thanks{Australian National 
University, Canberra, Australia, and Universit\'e Lille 1, France.}\, 
}

\date{\today}

\begin{document}

\maketitle

\begin{abstract}
We consider existence and uniqueness issues for the initial value problem 
of parabolic equations $\partial_{t} u = {\rm div} A \nabla u$ on the upper half 
space, with initial data in $L^p$ spaces. The coefficient matrix $A$ is assumed 
to be uniformly elliptic, but merely bounded measurable in space and time.
For real coefficients and a single equation, this is an old topic for which a 
comprehensive theory is available, culminating in the work of Aronson.
Much less is understood for complex coefficients or systems of equations 
except for the work of Lions, mainly because of the failure of maximum principles. 
In this paper, we come back to this topic with new methods that do not rely on 
maximum principles. This allows us to treat systems in this generality when 
$p\ge 2$, or under  certain   assumptions such as bounded variation in the time 
variable (a much weaker assumption that the usual H\"older continuity assumption) 
when $p< 2$. We reobtain results for real coefficients, and also 
complement them. For instance, we obtain uniqueness for arbitrary $L^p$ data, 
$1\le p \le \infty$, in the class $L^\infty(0,T; L^p({\mathbb{R}}^n))$.
Our approach to the existence problem relies on a careful construction of 
propagators for an appropriate energy space, encompassing previous 
constructions. Our approach to the uniqueness problem, the most novel aspect 
here, relies on a parabolic version of the Kenig-Pipher maximal function, used in 
the context of elliptic equations on non-smooth domains.  
We also prove comparison estimates involving conical square functions of Lusin 
type and prove some Fatou type results about non-tangential convergence of 
solutions. Recent results on maximal regularity operators in tent spaces that do 
not require pointwise heat kernel bounds are key tools in this study. 
\end{abstract}

\tableofcontents

\section{Introduction} 
\label{sec1}

We consider the problem
\begin{equation}
\label{eq1}
\partial_t u(t,x)={\rm div}\,(A(t,x)\nabla u(t,x)),\quad t>0, x\in {\mathbb{R}}^n
\end{equation}
where 
$A\in L^\infty((0,\infty)\times {\mathbb{R}}^n,{\mathscr{M}}_n({\mathbb{C}})))$
satisfies uniform ellipticity estimates:
\begin{equation}
\label{ell}
\begin{array}{l}
\exists\,\Lambda>0 \mbox{ such that } \forall\,\xi,\eta\in{\mathbb{C}}^n,
|\langle A(t,x)\xi,\eta\rangle|\le \Lambda|\xi||\eta| \mbox{ for a.e. $t>0$ and
$x\in{\mathbb{R}}^n$};
\\[8pt]
\exists\,\lambda>0 \mbox{ such that } \forall\,\xi\in{\mathbb{C}}^n,
\Re e(\langle A(t,x)\xi,\xi\rangle)\ge \lambda|\xi|^2 \mbox{ for a.e. $t>0$ and
$x\in{\mathbb{R}}^n$}.
\end{array}
\end{equation}
The divergence and gradient are taken with respect to the $x$ variables only. 
We mention right away that our results extend to systems of parabolic equations
with ellipticity \eqref{ell} replaced by a G\aa rding inequality on ${\mathbb{R}}^n$ 
uniformly with respect to $t$. For the sake of simplicity, we only consider one 
equation, but complex valued coefficients. We also restrict to $t>0$ since we 
are interested in the initial value problem with data at $t=0$. More precisely, 
we shall study three problems. 

1) Construct weak solutions for general $L^p$ initial data and prove sharp 
estimates.

2)  Show when a weak solution has a trace at $t=0$ and is uniquely determined 
by it. 

3) Establish well-posedness as a consequence. 
 
These problems have been studied in \cite{Li57, Ar68}; see  also \cite{russe} 
and the references therein. Here, we obtain striking results for systems and 
$L^p$ data, as well as new results (e.g., concerning well-posedness classes) 
for the case of a real equation. For example, we prove uniqueness results for 
arbitrary $L^p$ data, an issue left unresolved by Aronson.  
Furthermore, even in the case of a real equation, our methods are technically 
innovative: they have to be so to circumvent the loss of maximum principles.
In particular, we do not rely on the local regularity theory for solutions which 
culminated in \cite{Na, Mo}, and do not require a priori knowledge of boundedness 
or regularity properties of solutions in our approach.   

Recall the meaning of a weak solution. 

\begin{definition} 
Let $0\le a <b \le \infty$, $\Omega$ be an open subset of ${\mathbb{R}}^n$ and 
$Q=(a,b)\times \Omega$. A weak solution of \eqref{eq1} on $Q$ is 
a (complex-valued) function $u\in L^2_{\rm loc}(a,b; H^1_{\rm loc}(\Omega))$ 
such that  
\begin{equation}
\label{eq:weak}
\iint_Q u(t,x)\overline{\partial_t\varphi(t,x)}
\,{\rm d}x\,{\rm d}t
=\iint_Q A(t,x)\nabla u(t,x)\cdot \overline{\nabla\varphi(t,x)}
\,{\rm d}x\,{\rm d}t
\end{equation}
for all $\varphi\in{\mathscr{C}}_c^\infty(Q)$. For $0\le a<b<\infty$ and 
$\Omega={\mathbb{R}}^n$, we say that $u$ is a local (in time) solution on $(a,b)$, 
and when $Q={\mathbb{R}}^{n+1}_{+}:= (0,\infty)\times {\mathbb{R}}^n$ we say 
that $u$ is a global weak solution.   
\end{definition}

Recall that well-posedness for the Cauchy problem consists in proving 
existence and uniqueness for global (or local) weak solutions of \eqref{eq1} 
$u$ in some solution space $X$, converging, as $t$ tends to $0$, to an initial data 
$f$ in a space of initial data $Y$, in some appropriate sense. In this case, $X$ 
is said to be a well-posedness class for \eqref{eq1} for $Y$ data.  

This problem is well-understood for global solutions of the heat equation when 
$Y= L^p({\mathbb{R}}^n)$ and $X = L^\infty((0,\infty);L^p({\mathbb{R}}^n))$, 
for $p \in [1,\infty]$.  

First, the heat extension of $f\in Y$ is easily seen to belong to $X$. 
Conversely, use of the maximum principle and  form methods allow one to 
prove that all weak solutions (which are, in fact, classical solutions) in $X$ 
have a trace in $Y$ and are given by the semigroup.
The most efficient arguments seem to be the ones designed for Riemannian 
manifolds, because they do not rely on any explicit formula for the heat kernel. 
Strichartz, in \cite{Str83}, proves this result for $1<p<\infty$, even for global 
solutions with $ \|u(t,\cdot)\|_{L^p({\mathbb{R}}^n)}$ possibly growing as 
$t\to \infty$ (but not faster than exponentially). For $p=1$, 
we refer to \cite{Li84} for a neat proof, and another argument for $1<p<\infty$. 
For $p=\infty$, see \cite{Dod83} for a uniqueness result under a continuity 
assumption.

\medskip

Back to the Euclidean case for the non-autonomous problem \eqref{eq1}, 
it was Aronson \cite{Ar68} who obtained the most complete results for real 
equations in divergence form. He considers the energy space 
$L^\infty(0,T; L^2({\mathbb{R}}^n))\cap L^2(0,T; H^1({\mathbb{R}}^n))$.   
He proves that all solutions $u$ in this space have a trace $u_0$ in 
$L^2({\mathbb{R}}^n)$, and are uniquely determined by this trace.  
It follows that this class is a uniqueness class. Aronson also obtains existence 
given an $L^2$ initial data, hence defines a propagator $\Gamma$
such that $u(t, \cdot)=\Gamma(t,0)u_{0}$ for $t> 0$. The same strategy, 
with a slightly different energy space, was employed by Lions \cite{Li57} 
earlier for complex equations, and it yields the same solution. For real equations,  
however, Aronson also proved pointwise Gaussian decay of the propagator in
\cite{Ar67}. This allows one to define weak solutions by the integral representation
$$
u(t,x)=\int_{{\mathbb{R}}^n}k(t,0,x,y)u_0(y)\,{\rm d}y
$$
for $u_0$ in various spaces of initial conditions. 
For solutions satisfying an integral condition
$$
\|u\|_{\mathcal{E}}^2:=\int_{0}^T\int_{{\mathbb{R}}^n} e^{-a |x|^2} u(t,x)^2
\,{\rm d}t\,{\rm d}x<\infty
$$
for some $a>0$, Aronson proves uniqueness in this class, and existence given  
$u_0\in L^2(e^{-\gamma |x|^2}{\rm d}x)$ (with an assumption linking
$\gamma>0$, $T>0$ and $a$). This covers $u_0\in L^p({\rm d}x)$ with 
$2\le p\le \infty$, but note that $\|u\|_{\mathcal{E}}$ is not comparable to 
$\|u_0\|_{L^p}$. Aronson's class may thus not be optimal for uniqueness 
(one could look for a larger one). We are not mentioning here the results 
for non-negative solutions as they are clearly outside the scope of the present 
article, since we want to address complex equations. 

\medskip

Let us come back to the heat equation and consider solutions given by 
$u(t,x)=e^{t\Delta}f(x)$ for, say, $f \in L^2$. In harmonic analysis, there are 
other well-known estimates for such solutions given in terms of the 
non-tangential maximal function and the Lusin area functionals:
\begin{equation}
\label{FS}
\|u^*\|_{L^p} \sim \| \nabla u \|_{T^{p,2}},  \quad 1\le p <\infty.
\end{equation}
Here $T^{p,2}$ denotes the tent space of Coifman-Meyer-Stein. See 
Section~\ref{subsec:functions} for its definition. We denote by  
$u^{*}$ the non-tangential maximal function 
$x\mapsto \displaystyle{\sup_{|y-x|<\sqrt t}|u(t,y)|}$.
A key feature of these estimates is that they hold also for some $p\le1$, 
and play a fundamental role in Hardy space theory.
  
For example, $\|u^*\|_{L^p}<\infty$ characterises the real Hardy space 
$H^p$ as shown in \cite{FeSt72}.   
When $1<p<\infty$, an implicit argument (it is done for harmonic functions 
but the same idea applies to caloric functions) in \cite{FeSt72}, using Fatou type 
results (based on the maximum principle), shows that all weak solutions of the 
heat equation satisfying $\|u^*\|_{L^p}<\infty$ are given by the semigroup, 
and thus are uniquely determined by their traces in $L^p$ at $t=0$. 
As we have comparability of norms, uniqueness in such a class is an optimal 
result for $L^p$ data. It is not known to us whether the condition 
$\|\nabla u\|_{T^{p,2}}<\infty$ with $u$ vanishing at $\infty$ yields uniqueness 
(recall that $\nabla$ is only with respect to $x$) except when $p=2$.

\medskip

Our approach to \eqref{eq1} starts as in Aronson (\cite{Ar68}) 
or Lions (\cite{Li57}), by considering energy solutions. If $u$ is either one's 
solution (it turns out that they are the same) for a data 
$u_0\in L^2({\mathbb{R}}^n)$, one obtains the energy equality 
$$
\|u_0\|_{L^2}^2 = 2\Re e \int_0^T\int_{{\mathbb{R}}^n}  A(s,x)\nabla u(s,x) 
\cdot \overline{\nabla u(s,x)} \,{\rm d}s\,{\rm d}x  + \|u(T)\|_{L^2}^2.
$$
By taking the limit as $T\to \infty$, provided that $u$ is a global weak solution 
and that $\|u(T)\|_{L^2}\to 0$, one obtains 
$$
\|u_0\|_{L^2}^2 = 2\Re e \int_0^\infty\int_{{\mathbb{R}}^n} A(s,x)\nabla u(s,x) 
\cdot \overline{\nabla u(s,x)} \,{\rm d}s\,{\rm d}x.
$$   
This equality suggests that it should be possible to work directly in the
largest possible energy space to begin with, consisting of global weak solutions 
with $\nabla u \in L^2({\mathbb{R}}^{n+1}_{+})=L^2(L^2)$. We prove that this is 
indeed the case, and establish well-posedness of the Cauchy 
problem for $L^2({\mathbb{R}}^n)$ data in this energy space. We also show 
that such solutions are continuous from $[0,\infty)$ into $L^2({\mathbb{R}}^n)$,  
norm decreasing in time with limit 0 at $\infty$, and satisfying this energy equality, 
of course, together with  
$$
\|u_0\|_{L^2}= \|u\|_{L^{\infty}(L{^2})}\sim \|\nabla u \|_{L^2(L^2)}.
$$
This is to be expected but note that the lack of a priori control on the 
$L^2({\mathbb{R}}^n)$ norm in our energy space is a difficulty which we 
overcome thanks to a structural lemma for this space. As a consequence, 
we recover, by restriction to finite intervals $(0,T)$, the Aronson/Lions solutions. 
This gives rise to a propagator $\Gamma(t,s)$ that sends a data at time 
$s$ to the solution at time $t$. The only available estimates for this propagator 
in full generality are Gaffney type estimates, which are localized $L^2$ Gaussian 
bounds. The same holds for the backward in time adjoint equation. Using 
properties of this adjoint propagator to create test functions for \eqref{eq1}, 
our main result towards uniqueness is an interior reproducing 
formula for local weak solutions under a certain control.

\begin{theorem}
\label{thm:intro}
Let $u$ be a local weak solution of \eqref{eq1} on $(a,b)\times {\mathbb{R}}^n$. 
Assume
$$
M:=\int_{{\mathbb{R}}^{n}} \Bigl(\int_a^b \int_{B(x,\sqrt{b})} |u(t,y)|^{2}
\,{\rm d}y\,{\rm d}t\Bigr)^{\frac{1}{2}}e^{-\gamma |x|^2}\,{\rm d}x<\infty
$$
for some $0<\gamma<\gamma(a,b,\lambda,\Lambda)$. Then 
$u(t,\cdot) = \Gamma(t,s)u(s,\cdot)$ for every $a<s\leq t<b$, in the following sense:
$$
\int_{{\mathbb{R}}^n} u(s,x)\,\overline{\Gamma(t,s)^*h(x)}\,{\rm d}x 
= \int_{{\mathbb{R}}^n} u(t,x)\, \overline{h(x)}\,{\rm d}x 
\quad \forall h \in {\mathscr{C}}_c({\mathbb{R}}^n).
$$
\end{theorem}

Note that the control is in terms of local $L^2$ estimates on $u$. This is the only 
available information. Also the presence of the square root in the control turns out 
to be very useful.   

Once this is proved the matter reduces to controlling $u$ near the boundary 
$t=0$ to be able to take a limit as $s$ tends to $0$ in  
$u(t,\cdot) = \Gamma(t,s)u(s,\cdot)$. 

We thus need to work with solution spaces for which the hypothesis of this 
result can be checked. A natural choice is to use a modification of the maximal 
function $u^*$, adapting the one introduced by Kenig-Pipher \cite{KePi93} in the 
context of elliptic equations: 
$$
\tilde N(F)(x):=\sup_{\delta>0}\Bigl(\fint_{\frac{\delta}{2}}^\delta
\fint_{B(x,\sqrt{\delta})}|F(t,y)|^2\,{\rm d}y\,{\rm d}t\Bigr)^{\frac{1}{2}}.
$$
However, note that the space of all measurable functions 
with $\|\tilde N(F)\|_{L^p}<\infty$ does not seem to have a trace space at $t=0$, 
even allowing limits in the weakest possible sense. Hence, finding the initial 
value relies on the equation as well, using the interior representation above. 
When $2<p\le \infty$, we prove well-posedness of global weak solutions of 
\eqref{eq1} in the class $X^p = \{u \in L^2_{\rm loc}({\mathbb{R}}_{+}^{n+1}) 
\;;\; \|\tilde N(u)\|_{L^p}<\infty\}$ with arbitrary data in $L^p$. In particular, 
when $p=\infty$, we establish the conservation property 
$$
\Gamma(t,0)1\!{\rm l}=1\!{\rm l}
$$
in $L^2_{\rm loc}({\mathbb{R}}^n)$, for all $t>0$. This seems to be new under 
the sole ellipticity assumption.  

For $p=2$, we also establish, via a different argument, that both 
$L^\infty(L^2)$ and $X^2$ are well-posedness classes for $L^2$ data. 
The corresponding solutions agree with the energy solutions. In particular, 
for any given global weak solution, we have the a priori equivalences
$$ 
\|u\|_{L^\infty(L^2)}\sim \|\nabla u \|_{L^2(L^2)} \sim \|\tilde N(u)\|_{L^2}.
$$

The above results can also be considered for local solutions or for global 
solutions with growth when $t\to \infty$. Combining this with the interior
representation, we obtain a representation for classes containing 
global weak solutions having arbitrary growth as $t\to \infty$ (but still 
controlled as $|x|\to \infty$). This is quite new as well. 

Imposing more properties on the propagators, such as uniform $L^p$ 
boundedness in some range of $p$, allows us to consider the classes 
$L^\infty(L^p)$ as above when $p\neq 2$. This is true for small perturbations 
of autonomous equations (coefficients independent of $t$) or when the
coefficients are of bounded variation in time. This is far less demanding than 
the usual H\"older regularity assumption. We expect that this will give 
substantial improvements to maximal regularity results for the associated
inhomogeneous non-autonomous problem. 

Another consequence is that a pointwise upper Gaussian bound condition 
on the propagator kernel (as obtained by Aronson for real equations) yields 
unique determination of weak solutions from their traces at $t=0$ in the classes
$L^\infty(L^p)$, when $1<p\le \infty$. Note that this pointwise upper Gaussian bound condition has been 
characterized in \cite{HK04} in terms of local $L^2-L^\infty$ 
estimates of weak solutions of \eqref{eq1} and of the dual backward equation. 
For $p=1$, we obtain, under this assumption, two criteria to decide whether
or not a weak solution in $L^\infty(L^1)$ is determined by its trace in $L^1$ or 
in the space of Radon measures. This requires some further regularity on the 
propagators.  

Our work also includes a non-autonomous analog of the Fefferman-Stein 
equivalence \eqref{FS}. Namely we prove that, for all weak solutions of 
\eqref{eq1} of the form $u(t,\cdot)=\Gamma(t,0)f$ with $f\in L^2 \cap L^p$, 
we have the a priori comparison 
\begin{equation}
\label{FS1}
\|\tilde{N}(u)\|_{L^p} \sim \| \nabla u \|_{T^{p,2}},  \quad 1 \le p < \infty,
\end{equation}
In fact, the control of $\| \nabla u \|_{T^{p,2}}$ by $\|\tilde{N}(u)\|_{L^p}$ is 
valid for any global weak solution and $0<p\le \infty$ and it is only for the 
converse that we use the form of the solution.   

Finally, we prove Fatou type result on non-tangential almost everywhere 
convergence at the boundary. To do so, since solutions may not be locally 
bounded, we replace pointwise values by averages on Whitney regions. 

\medskip

The paper is organised as follows.
In Section~\ref{sec2}, we recall the definitions of various function spaces and 
operators used in this article. We also recall results from \cite{Au07,AKMP12} 
that play a key role here.

In Section~\ref{sec3}, we develop a new approach to the $L^2$ theory, 
including well-posedness in the largest possible energy space and, as a 
consequence, the existence of a contraction operator $L^2({\mathbb{R}}^n)$, 
called propagator, that maps the data $h$ to our solution $u$ at time $t$.   
By restriction, this propagator gives both Lions' energy solution and 
Aronson's energy solution.   

In Section~\ref{sec4}, we prove the fundamental a priori estimates for 
weak solutions (either general weak solutions or energy solutions given by 
the propagator), including reverse H\"older estimates, and appropriate 
integrated off-diagonal bounds. The latter are a replacement for the pointwise 
heat kernel bounds available in the case of real coefficients.

In Section~\ref{sec5}, we prove our existence and uniqueness results. 
This includes the key interior representation result, Theorem \ref{thm:intro},  
well-posedness in $L^\infty(L^2)$ and in  $X^p$ for $p > 2$, and the 
conservation property. Under an additional assumption on the $L^p$ behaviour 
of the propagators, we prove well posedness in $L^\infty(L^p)$ for all 
$p \in (1,\infty]$. 

In Section~\ref{sec6}, we show that this additional assumption is satisfied 
for a range of values of $p$ in two important situations: when $A$ is a small 
$L^\infty$ perturbation of a $t$-independent matrix, and when $A$ is of 
bounded variation in time. We also show a local result when the dependency with respect to $t$ is continuous.

In Section~\ref{sec7}, we complete the picture by showing an $L^p$ analogue, 
for $p\in (1,\infty)$, of the norm estimates available for energy solutions when 
$p=2$. This is an analogue of Fefferman-Stein's equivalence of maximal function 
norms and square function norms in Hardy space theory. 

In Section~\ref{sec8}, we show non-tangential convergence results to the initial 
data for our weak solutions.

In Section~\ref{sec9}, we focus on $p=1$, assuming that our propagators have 
pointwise kernel bounds (as in the case of real coefficients). We then get a 
complete theory for Radon mea\-sures as data and solutions in 
$L^\infty(L^1)$, or $L^1$ data and solutions in a subspace of $L^\infty(L^1)$.

Finally, in Section~\ref{sec10}, we mention an easy extension of our results: 
similar well-posedness results hold for global weak solutions $u$ such that 
norms (in the corresponding solution space) of 
$(t,x)\mapsto 1\!{\rm l}_{(0,T)}(t)u(t,x)$ can grow as $T$ tends to $\infty$. 
A posteriori, we show that this growth is bounded.  
  
\subsection{Acknowledgments}

This work was supported by the Australian Research Council through
Discovery Project  DP120103692, and Portal's Future Fellowship FT130100607.
A long term visit of Monniaux at the Australian National University played a key 
role in the completion of the project. This visit was possible thanks to the 
``Laboratoire International Associ\'e en Analyse et G\'eom\'etrie" agreement 
between the French Conseil National de la Recherche Scientifique and the
ANU. Portal would also like to thank the Laboratoire de 
Math\'ematiques de Besan\c con, for hosting him as part of its 2014 trimester 
in ``Geometric and Noncommutative Methods in Functional Analysis", during 
which part of this work was conducted.  Auscher and Monniaux were also 
partially supported by the ANR project ``Harmonic analysis at its boundaries'' 
ANR-12-BS01-0013.

\section{Preliminaries}
\label{sec2}

\subsection{Function spaces}
\label{subsec:functions}

\subsubsection*{Vector valued spaces}

When dealing with function spaces over $(a,b) \times {\mathbb{R}}^n$, we 
write $L^p(X)$ for the Bochner space of  $X({\mathbb{R}}^n)$ valued 
$L^p$ functions  $L^p(a,b;X({\mathbb{R}}^n))$ or
$L^p(a,b;X({\mathbb{R}}^n;{\mathbb{C}}^n))$ (as long as no confusion can occur).

We denote by ${\mathscr{D}}$ the space
${\mathscr{C}}_c^\infty((0,\infty)\times{\mathbb{R}}^n)$ and by
${\mathscr{D}}'$ the  space of distributions on $(0,\infty)\times{\mathbb{R}}^n$.
We denote by ${\mathscr{C}}_0(L^p)$ the space of 
$L^p({\mathbb{R}}^n)$-valued continuous functions on $[0,\infty)$ that tend 
to $0$ at infinity. 

\subsubsection*{The homogeneous Sobolev spaces $\dot H^1({\mathbb{R}}^n)$} 

There are many ways to define the homogeneous space 
$\dot H^1({\mathbb{R}}^n)$.  We depart a little bit from tradition of having 
this space as a space of distributions modulo constants, as this simplifies 
its use in \eqref{eq1}. 

We denote by $H^s({\mathbb{R}}^n)$ the standard inhomogeneous Sobolev 
space for $s \in {\mathbb{R}}$, and we equip $L^2({\mathbb{R}}^n;{\mathbb{C}}^k)$ 
with the standard complex inner product, which we denote by 
$\langle\cdot,\cdot\rangle$ or ${}_{L^2}\langle \cdot,\cdot\rangle_{L^2}$.

We set $\dot{H}^1({\mathbb{R}}^n) 
= \bigl\{u \in {\mathscr{D}}'({\mathbb{R}}^n) \;;\; 
\nabla u \in L^2({\mathbb{R}}^n;{\mathbb{C}}^n)\bigr\}$, 
and equip this space with the seminorm $u\mapsto \|\nabla u\|_{L^2}$. 
With this definition, the following properties hold:
\begin{enumerate}
\item
$H^1({\mathbb{R}}^n) \subset \dot{H}^1({\mathbb{R}}^n)
\subset L^2_{\rm loc}({\mathbb{R}}^n)$ (set inclusions).
\item
${\mathscr{D}}({\mathbb{R}}^n)$ is dense in $\dot{H}^1({\mathbb{R}}^n)$: 
for all $u \in \dot{H}^1({\mathbb{R}}^n)$ there exists a sequence 
$(\phi_j)_{j \in {\mathbb{N}}}$ of functions in ${\mathscr{D}}({\mathbb{R}}^n)$ 
such that $\|\nabla \phi_j- \nabla u\|_2 \xrightarrow[j \to \infty]{} 0$.
\item
$\dot{H}^1({\mathbb{R}}^n)/ {\mathbb{C}}$ is a Banach space equipped with its 
quotient norm.
\item
$\dot{H}^1({\mathbb{R}}^n) \subset {\mathscr{S}}'({\mathbb{R}}^n)$ (set inclusion).
\item
The dual of $\dot{H}^1({\mathbb{R}}^n)$ can be identified with the dual of 
$\dot{H}^1({\mathbb{R}}^n)/{\mathbb{C}}$, and with
$\dot{H}^{-1}({\mathbb{R}}^n) = \bigl\{{\rm div}\,g \;;\; 
g \in L^2({\mathbb{R}}^n;{\mathbb{C}}^n)\bigr\}$
equipped with the norm 
$f \mapsto \|f\|_{\dot{H}^{-1}}= \inf\bigl\{\|g\|_{L^2} \;;\; f={\rm div}\,g\bigr\}$. 
Moreover, for all $u \in \dot{H}^1({\mathbb{R}}^n)$, all 
$g \in L^2({\mathbb{R}}^n,{\mathbb{C}}^n)$, and $f={\rm div}\, g$, 
we have that 
$$
{}_{\dot{H}^{-1}}\langle f,u\rangle_{\dot{H}^1}
= - {}_{L^2}\langle g,\nabla u\rangle_{L^2}
= {}_{\dot{H}^{-1}}\langle f,[u]\rangle_{\dot{H}^1/ \mathbb{C}}.
$$
In particular, 
$\dot{H}^{-1}({\mathbb{R}}^n) \subset H^{-1}({\mathbb{R}}^n) 
\subset {\mathscr{S}}'({\mathbb{R}}^n)$ 
(embeddings), and, if $u \in \dot{H}^1({\mathbb{R}}^n)\cap L^2({\mathbb{R}}^n) 
= H^1({\mathbb{R}}^n)$ and 
$f \in \dot{H}^{-1}({\mathbb{R}}^n)\cap L^2({\mathbb{R}}^n)$ then
$$
{}_{\dot{H}^{-1}}\langle f,u\rangle_{\dot{H}^1}
= \int_{{\mathbb{R}}^n}f(x)\overline{u(x)}\,{\rm d}x = 
{}_{L^2}\langle f, u\rangle_{L^2}.
$$
\end{enumerate}
These properties are well known. We shall often write
${}_{\dot{H}^1}\langle u,f\rangle_{\dot{H}^{-1}}$ to mean
${}_{\dot{H}^{-1}}\overline{\langle f,u\rangle}_{\dot{H}^1}$.
Having this in hand, we have that, for $A$ satisfying \eqref{ell} and almost 
every $t>0$, 
$$
L(t) = -{\rm div}\,A(t,\cdot) \nabla
$$
defines a bounded operator from $\dot{H}^1({\mathbb{R}}^n)$ to 
$\dot{H}^{-1}({\mathbb{R}}^n)$, 
which is onto and has ${\mathbb{C}}$ as its null space (if one uses 
$\dot{H}^1({\mathbb{R}}^n)/{\mathbb{C}}$, we thus have an isomorphism).
More precisely, for all $u,v \in \dot{H}^1({\mathbb{R}}^n)$,
${}_{\dot{H}^1}\langle L(t)u,v\rangle_{\dot{H}^{-1}}=
{}_{L^2}\langle A(t,.)\nabla u,\nabla v\rangle_{L^2}$, and
$$
\lambda \|u\|_{\dot{H}^1} \le \|L(t)u\|_{\dot{H}^{-1}} 
\le \Lambda \|u\|_{\dot{H}^1}.
$$
Now assume that $A$ is constant in $t$, and set $L=-{\rm div}\, A \nabla$ and 
$D(L) = \bigl\{u \in H^1 \;;\; Lu \in L^2\bigr\}$. Then $L$ is the maximal 
accretive operator on $L^2({\mathbb{R}}^n)$ associated with the form 
$(u,v) \mapsto {}_{L^2}\langle A\nabla u, \nabla v \rangle_{L^2}$ on 
$H^1({\mathbb{R}}^n)$. In particular, it is sectorial and $-L$ generates an 
analytic semigroup of contractions $(e^{-tL})_{t>0}$. Also, the solution of 
Kato's square root conjecture in \cite{AHLMcIT02} implies that
\begin{align*}
\sup_{t>0}\|\nabla e^{-tL}u\|_{L^2} &\lesssim 
\sup_{t>0}\|L^{\frac{1}{2}} e^{-tL}u\|_{L^2} 
\lesssim \|L^{\frac{1}{2}}u\|_{L^2} \lesssim
\|\nabla u\|_{L^2} \quad \forall u \in H^1({\mathbb{R}}^n).
\end{align*}
Therefore, as  $e^{-tL}1\!{\rm l}=1\!{\rm l}$ in $L^2_{\rm loc}$ 
(see \cite[\S2.5]{Au07}), we have that $\{e^{-tL} \;;\; t>0\}$ extends to a 
uniformly bounded family of bounded operators on $\dot{H}^1({\mathbb{R}}^n)$. 
Finally, we use the space $L^2(a,b;\dot{H}^1({\mathbb{R}}^n))$ for 
$-\infty\le a <b \le +\infty$, endowed with the seminorm
$u \mapsto \bigl(\int_a^b \|\nabla u(t,.)\|_{L^2}^2\,{\rm d}t\bigr)^{\frac{1}{2}}$.
It follows from the above discussion that 
${\mathscr{C}}_c^\infty((a,b)\times {\mathbb{R}}^n)$ 
is dense in $L^2(a,b;\dot{H}^1({\mathbb{R}}^n))$, that 
$L^2(a,b;\dot{H}^1({\mathbb{R}}^n)) 
\subset L^2(a,b;L^2_{\rm loc}({\mathbb{R}}^n)\cap {\mathscr{S}}'({\mathbb{R}}^n))$,
and that its dual can be identified with $L^2(a,b;\dot{H}^{-1}({\mathbb{R}}^n))$ 
through the pairing
\begin{align*}
\label{}
 {}_{L^2(a,b;\dot{H}^{-1})} \langle f,u
\rangle_{L^2(a,b;\dot{H}^1)}   &  = \int_a^b 
{}_{\dot{H}^{-1}}\langle f(t,.),u(t,.)\rangle_{\dot{H}^1}\,{\rm d}t
\\
& = -\int_a^b {}_{L^{2}}\langle \psi(t,.),\nabla u(t,.) \rangle_{L^{2}}\,{\rm d}t,
\end{align*}
for any $\psi \in L^2(a,b;L^2({\mathbb{R}}^n))$ such that $f={\rm div}\,\psi$, 
and $u \in L^2(a,b;\dot{H}^1({\mathbb{R}}^n))$.

\subsubsection*{Homogeneous Lions spaces $\dot W(0,\infty)$}

We define the following spaces that are variants of the solution spaces used by 
Lions in \cite[spaces ${\mathscr{A}}(\Omega)$ and ${\mathscr{B}}(\Omega)$ 
p.\,147]{Li57} (see also \cite[Chap.\,XVIII]{DL88}).
\begin{align*}
\dot W(0,\infty)&:=\bigl\{u\in{\mathscr{D}}' ;
u\in L^2(\dot H^1) \mbox{ and }\partial_tu\in L^2(\dot H^{-1})\bigr\}
\\
\mbox{and}\qquad&
\\
W(0,\infty)&:=\dot W(0,\infty)\cap{\mathscr{C}}_0(L^2),
\end{align*}
and the corresponding spaces on a time interval $(a,b)$, $0\le a<b<\infty$
$$
\dot{W}(a,b):=\bigl\{u\in({\mathscr{C}}_c^\infty((a,b)\times{\mathbb{R}}^n))' ;
u\in L^2(a,b;\dot H^1) \mbox{ and }\partial_tu\in L^2(a,b;\dot H^{-1})\bigr\},
$$
and $W(a,b)= \dot{W}(a,b) \cap {\mathscr{C}}([a,b]; L^2)$. 
An important result of Lions \cite[Proposition~3.1]{Li57} states that 
inhomogeneous versions of these spaces (replacing $\dot{H}^1$ and 
$\dot{H}^{-1}$ by $H^1$ and $H^{-1}$) embed into 
${\mathscr{C}}([a,b];L^2)$, (see also \cite[Chap.\,XVIII]{DL88}), that is, 
into $W(a,b)$. 
With quite a different proof, we prove, in Section~\ref{subsec:struct}, a version 
of this result for $\dot W(0,\infty)$.

\subsubsection*{Tent spaces $T^{p,2}$}

The tent spaces introduced by Coifman, Meyer, and Stein in \cite{CMS85} play a 
key role in our work. For $p\in(0,\infty]$, the (parabolic) tent space $T^{p,2}$ is 
the set of measurable  functions $u$ on ${\mathbb{R}}^{n+1}_+$ such that
\begin{align*}
x\mapsto \Bigl(\int_0^\infty \fint_{B(x,\sqrt{t})} |u(t,y)|^2
\,{\rm d}y\,{\rm d}t\Bigr)^{\frac{1}{2}} \in L^p({\mathbb{R}}^n),
\; \text{if} \; p<\infty,\\
x\mapsto \sup_{B\ni x}\Bigl(\int_0^{r_B^2} 
\fint_B |u(t,y)|^2
\,{\rm d}y\,{\rm d}t\Bigr)^{\frac{1}{2}} \in L^\infty({\mathbb{R}}^n),
\; \text{if} \; p=\infty,
\end{align*}
where we denote by $r_B$ the radius of a ball $B$.
Note that $T^{p,2}$ is contained in  $L^2_{\rm loc}({\mathbb{R}}^{n+1}_+)$.
As shown in \cite{CMS85}, these spaces are Banach spaces when 
$1\le p\le \infty$, reflexive when $p\in(1,\infty)$, and the dual of $T^{p,2}$ 
is $T^{p',2}$ for the duality given by 
$\int_{{\mathbb{R}}^{n+1}_+} f(t,y) {g(t,y)}\,{\rm d}y\,{\rm d}t$.  
Their importance for us has two origins. One is elliptic boundary value 
problems including the Laplace equation, where tent spaces, along with closely 
related objects such as Hardy spaces and Carleson measures, are already used 
extensively. Since we consider equation \eqref{eq1} weakly in space and time, 
it is natural to use such norms rather than the $L^\infty(L^p)$ norms which would 
correspond to treating \eqref{eq1} as an (non-autonomous) evolution equation 
in $L^p$.  
The other reason why tent spaces are so important in our work comes from 
the recent extension of Calder\'on-Zygmund theory to rough settings, i.e.
the application of Calder\'on-Zygmund ideas to operators such as 
$e^{-tL}$ with $L= - {\rm div}\, A \nabla$, 
$A \in L^\infty({\mathbb{R}}^n;{\mathscr{M}}_n({\mathbb{C}}))$ 
satifying \eqref{ell}, that do not, in general, have Calder\'on-Zygmund kernels 
(see \cite{Au07} and the references therein). In such a setting, integral 
operators such as 
$$
f\mapsto \Bigl[ (t,x) \mapsto \int_0^t \nabla e^{-(t-s)L }
{\rm div} f(s,\cdot)(x) {\rm d}s \Bigr]
$$
are often unbounded on Bochner spaces $L^p(L^q)$ but bounded on 
$T^{p,2}$. This is the subject of our paper \cite{AKMP12}. The results we 
use here are recalled in Section~\ref{subsec:maxreg}. Keeping in mind that 
$T^{2,2} = L^2(L^2)$, we then use the condition $\nabla u \in T^{p,2}$ 
(here, we mean that each component of $\nabla u$ is in $T^{p,2}$; in general, 
we shall not distinguish the notation as this will be clear from the context) as 
a replacement for the condition $\nabla u \in L^2(L^2)$ to attack $L^p$ theory. 
For uniqueness, however, maximal function estimates  on solutions are more 
suitable than square function estimates.

\subsubsection*{Kenig-Pipher modified $T^{p,\infty}$ space $X^p$}

Coifman-Meyer-Stein's tent space theory also includes maximal function 
estimates via the tent spaces $T^{p,\infty}$ defined as spaces of continuous 
functions on $(0,\infty)\times{\mathbb{R}}^n$ with $u^*\in L^p$ and with 
non-tangential limit, where $u^*$ is the non-tangential maximal function defined by 
$$
u^*:x\mapsto \sup_{\underset{|x-y|< \sqrt{t}}{(t,y)\in(0,\infty)
\times {\mathbb{R}}^n}} |u(t,y)|.
$$ 
This maximal function, however, is not appropriate for us because of the lack 
of pointwise bounds on our solutions. We thus use a modified version of the 
non-tangential maximal function, introduced by Kenig and Pipher for elliptic 
equations in \cite{KePi93}, and used extensively in \cite{AuAx11} (see also 
\cite{HM09,AMcIR08} and further development in the theory of Hardy 
spaces associated with operators without Gaussian bounds).

\begin{definition}
\label{def:tildeN}
For $F\in L^2_{\rm loc}({\mathbb{R}}^{n+1}_+)$, we define the
following maximal function $\tilde N(F)$ by 
$$
\tilde N(F)(x):=\sup_{\delta>0}\Bigl(\fint_{\frac{\delta}{2}}^\delta
\fint_{B(x,\sqrt{\delta})}|F(t,y)|^2\,{\rm d}y\,{\rm d}t\Bigr)^{\frac{1}{2}},  
\quad \forall x \in {\mathbb{R}}^n.
$$
\end{definition}

The corresponding modification of $T^{p,\infty}$ is defined as follows.

\begin{definition}
\label{def:Xp}
Let $0< p \le \infty$.
The space $X^p$ is the subspace of functions
$F\in L^2_{\rm loc}({\mathbb{R}}^{n+1}_+)$ such that
$$
\|F\|_{X^p}:=\|\tilde N(F)\|_p<\infty.
$$
\end{definition}

This space has been defined in \cite{KePi93}.  For $1\le p \le \infty$, it is a Banach space. 
Duality and interpolation is studied in \cite{HyRo12, Huang}.  

Note that, given a parameter $\beta > 1$, the maximal function 
$\tilde{N}(F)$ in the definitions above can be replaced by
$$
{\mathcal{N}}_\beta(F)(x) = \sup_{\delta>0}
\Bigl(\fint_{\delta^2}^{\beta^2\delta^2} \fint_{B(x,\beta\delta)}|F(t,y)|^2 
\,{\rm d}y\,{\rm d}t\Bigr)^{\frac{1}{2}}, \quad \forall x \in {\mathbb{R}}^n,
$$
since a  simple covering argument yields 
$\|\mathcal{N}_{\beta}(F)\|_p \sim \|F\|_{X^p}$.

A difficulty with this norm compared to the one with $u^*$ is the lack of stability 
by translation: one can check that if $\tau_s F(t,x)=F(t+s,x)$, then there is 
neither pointwise nor $L^p$ control of $\tilde N(F_s)$ by $\tilde N(F)$ for any 
$p$. The same difficulty appears with the tent spaces $T^{p,2}$ above except 
when $p=2$. 
 
\subsubsection*{Slice spaces $E_\delta^p$}

While integral operators such as 
$$
f\mapsto \Bigr[ (t,x) \mapsto \int_0^t \nabla e^{-(t-s)L} 
{\rm div}\, f(s,\cdot)(x)\,{\rm d}s \Bigr]
$$
act on $T^{p,2}$ (see Section \ref{subsec:maxreg}), their (operator-valued) 
kernels $\nabla e^{-t L} {\rm div}\,$ (for a fixed $t>0$) do not act, 
in general, on $L^p({\mathbb{R}}^n)$. Appropriate substitutes for 
$L^p({\mathbb{R}}^n)$ are  the following spaces.

\begin{definition}
\label{def:slice}
Let $p\in [1,\infty]$ and $\delta>0$. The (parabolic) slice space $E_\delta^p$ 
is the subspace of functions $g\in L^2_{\rm loc}({\mathbb{R}}^n)$
such that
$$
\|g\|_{E_\delta^p}:=\Bigl(\int_{{\mathbb{R}}^n}\Bigl(\fint_{B(x,\sqrt{\delta})}
|g(y)|^2\,{\rm d}y\Bigr)^{\frac{p}{2}}\,{\rm d}x\Bigr)^{\frac{1}{p}} <\infty.
$$
\end{definition}

This space can also be seen as one of Wiener amalgam spaces, which have been 
studied for a long time. However, \cite[\S 3]{AM14}  points out that these spaces
are retracts of tent spaces, and thus inherit many of their key properties: 
$\bigl(E_\delta^p\bigr)^*=E_\delta^{p'}$ 
with $\|\ell\|_{\bigl(E_\delta^p\bigr)^*}\sim \|\ell\|_{E_\delta^{p'}}$, with 
implicit constants uniform in $\delta >0$,  for all $p\in [1,\infty)$ under the 
duality pairing $\int_{{\mathbb{R}}^n}f(x)\overline{g(x)}\,{\rm d}x$. In particular, 
slice spaces are reflexive Banach spaces when $p \in (1,\infty)$.
The following result is \cite[Lemma~3.5]{AM14} and compares
the norms in $E_\delta^p$ and in $E_{\delta'}^p$ for $\delta'\neq\delta$.

\begin{lemma}
\label{lem:4.7}
Let $p\in[1,\infty]$ and $\delta,\delta'>0$. For all $f\in E_{\delta'}^p$, one
has $f\in E_\delta^p$ and
$$
\min\Bigl\{1,\bigl(\textstyle{\frac{\delta'}{\delta}}\bigr)^{\frac{1}{2}
(\frac{n}{2}-\frac{n}{p})}\Bigr\}
\|f\|_{E_{\delta'}^p} \lesssim \|f\|_{E_{\delta}^p} \lesssim
\max\Bigl\{1,\bigl(\textstyle{\frac{\delta'}{\delta}}\bigr)^{\frac{1}{2}
(\frac{n}{2}-\frac{n}{p})}\Bigr\}
\|f\|_{E_{\delta'}^p}.
$$
\end{lemma}

\subsection{Maximal regularity operators}
\label{subsec:maxreg}

Given $A\in L^\infty({\mathbb{R}}^n;{\mathscr{M}}_n({\mathbb{C}}))$ satisfying 
\eqref{ell}, recall that  $L=-{\rm div}\,A\nabla$ denotes the maximal accretive
operator with domain $D(L) = \{u \in H^1({\mathbb{R}}^n) \;;\; 
A\nabla u \in D({\rm div})\}$. 
Recall also that $\langle Lu,v \rangle = \langle A\nabla u, \nabla v \rangle$ for all 
$u \in D(L)$ and $v \in H^1({\mathbb{R}}^n)$.
See \cite{Au07, Ou05}, for more background on the operator theory of divergence 
form elliptic operators. 

We consider the associated maximal regularity operator ${\mathcal{M}}_L$ 
initially defined as a bounded operator from $L^1(D(L))$ to 
$L^\infty_{\rm loc}(L^2)$ by
\begin{equation}
\label{eq:maxreg}
{\mathcal{M}}_Lf(t,x)=\int_0^t Le^{-(t-s)L}f(s,\cdot)(x)\,{\rm d}s
\end{equation}
for almost every $(t,x)\in (0,\infty)\times{\mathbb{R}}^n$ and all
$f\in L^1(D(L))$. A classical result by De Simon \cite{deS64} states that
${\mathcal{M}}_L$ extends to a bounded operator on $L^2(L^2)$. 

De Simon's result can be extended in several directions, including
$L^p(L^p)$ boundedness, $L^p({\mathbb{R}}^n;L^2(0,\infty))$ boundedness,
and $T^{p,2}$ boundedness.

The $L^p(L^p)$ extension is the most well-known. Lutz Weis proved
in \cite{W01} that the maximal regularity operator
${\mathcal{M}}_L$ belongs to ${\mathscr{L}}(L^p(L^p))$ if and
only if $(e^{-tL})_{t\ge0}$ is $R$-analytic in $L^p({\mathbb{R}}^n)$. This
holds in a range $(p_-(L),p_+(L))$ around $2$ as shown in 
\cite[Theorem~5.1]{Au07} (com\-bined with \cite[Theorem~5.3]{KW01}). 
Note that, for $p$ outside of $[p_-(L),p_+(L)]$, $-L$ does not generate a 
${\mathscr{C}}_0$-semigroup on $L^p$.

The $L^p({\mathbb{R}}^n;L^2(0,\infty))$ extension has recently been 
considered in \cite{vNVW15}. Again 
${\mathcal{M}}_L\in{\mathscr{L}}(L^p({\mathbb{R}}^n;L^2(0,\infty)))$ when
$p\in(p_-(L),p_+(L))$ by a combination of \cite[Theorem~5.1]{Au07} and
\cite[Theorem~3.3]{vNVW15}.

The $T^{p,2}$ extension is the subject of our work \cite{AKMP12}. In
\cite[Proposition~1.6]{AKMP12}, we prove that 
${\mathcal{M}}_L\in{\mathscr{L}}(T^{p,2})$ for a range of values of $p$ that
can be strictly larger than $(p_-(L),p_+(L))$ (recall that, for all 
$\varepsilon>0$, there exists $-L$ that does not generate a semigroup on
$L^p({\mathbb{R}}^n)$ for $p<\frac{2n}{n+2}-\varepsilon$).

In this paper, however, we need to use a variant $\tilde{\mathcal{M}}_L$  of 
${\mathcal{M}}_L$ for which the $T^{p,2}$ boundedness has still a large 
range of exponents while the $L^p(L^p)$ theory would hold on an even 
smaller range  than for ${\mathcal{M}}_L$. 

\begin{proposition}
The integral  
\begin{equation}
\label{eq:modmaxreg}
\tilde{\mathcal{M}}_Lf(t,\cdot)=\int_0^t \nabla e^{-(t-s)L}{\rm div}\,f(s,\cdot)\,{\rm d}s
\end{equation}
defines a bounded operator from $L^1(H^2)$, where 
$H^2=H^2({\mathbb{R}}^n;{\mathbb{C}}^n))$, to $L^\infty_{\rm loc}(L^2)$. 
This operator extends to a bounded operator on $L^2(L^2)$.
\end{proposition}

\begin{proof}
To see that $\tilde{\mathcal{M}}_L$ is well defined, remark that, for all 
$\tau> 0$ and all $g\in H^2$,
$$
\|\nabla e^{-\tau L}{\rm div}\,g\|_{L^2} 
\lesssim \|\nabla{\rm div}\,g\|_{L^2}\lesssim \|g\|_{H^2}.
$$
Next, we turn to the extension. Remark that for such $g$, 
$h=L^{-\frac{1}{2}} {\rm div}\,g \in D(L)$. First, 
$g \in L^2({\mathbb{R}}^n; {\mathbb{C}}^n)$ and by the solution of the Kato 
square root problem \cite{AHLMcIT02}, $h\in L^2({\mathbb{R}}^n)$. Secondly, 
$Lh= L^{\frac{1}{2}} {\rm div}\,g \in L^2({\mathbb{R}}^n)$ as 
${\rm div}\,g \in H^1= D(L^{\frac{1}{2}})$ by \cite{AHLMcIT02}. Using 
$L^2$ boundedness of $\nabla L^{-\frac{1}{2}}$, \cite{AHLMcIT02}, we have 
the  equality in $L^2$
$$
\nabla e^{-\tau L}{\rm div}\,g= \nabla L^{-\frac{1}{2}} Le^{-\tau L} L^{-\frac 1 2} 
{\rm div}\,g
$$
for all such $g$ and  all $\tau>0$. It follows  
$$
\tilde{\mathcal{M}}_Lf=\nabla L^{-\frac{1}{2}}{\mathcal{M}}_L
L^{-\frac{1}{2}}{\rm div}\, f
$$
for all $f \in L^1(H^2)$ and that  $\tilde{\mathcal{M}}_L$  extends by density 
to a bounded operator on $L^2(L^2)$. 
\end{proof}

The adjoint $\tilde{\mathcal{M}}_L^*\in {\mathscr{L}}(L^2(L^2))$ is given as follows.

\begin{lemma}
\label{lem:tildeM_L^*}
For all $f,g\in {\mathscr{D}}$,
$$
\int_{\mathbb{R}}\langle \tilde{\mathcal{M}}_Lf(t,\cdot),g(t,\cdot) \rangle\,{\rm d}t
=\int_{\mathbb{R}} \Bigl\langle f(t,\cdot),
\int_0^\infty \nabla (e^{-sL})^*{\rm div}\,g(t+s,\cdot)\,{\rm d}s\Bigr\rangle
\,{\rm d}t.
$$
\end{lemma}

\begin{proof}
Let $f,g\in {\mathscr{D}}$. We have that
\begin{align*}
\int_{\mathbb{R}}\langle \tilde{\mathcal{M}}_Lf(t,\cdot),g(t,\cdot)\rangle\,{\rm d}t
=&\int_{\mathbb{R}}\int_{\mathbb{R}} 1\!{\rm l}_{(0,\infty)}(t-s)
\langle\nabla e^{-(t-s)L}{\rm div}\,f(s,\cdot),g(t,\cdot)\rangle\,{\rm d}s\,{\rm d}t
\\
=&\int_{\mathbb{R}}\int_{\mathbb{R}} 1\!{\rm l}_{(0,\infty)}(\sigma)
\langle f(s,\cdot),\nabla (e^{-\sigma L})^*{\rm div}\,g(\sigma+s,\cdot)\rangle\,
{\rm d}\sigma\,{\rm d}s
\end{align*}
where we have made the change of variables $s=s$ and $\sigma=t-s$
on ${\mathbb{R}}\times{\mathbb{R}}$.
\end{proof}

\begin{remark}
\label{rem:tildeM_L^*}
The operator, initially defined for 
$g\in {\mathscr{D}}$ by
$$
\tilde{\mathcal{M}}_L^*g(s,x)=\int_0^\infty
\nabla (e^{-\sigma L})^*{\rm div}\,g(\sigma+s,\cdot)(x)\,{\rm d}\sigma,
\quad (s,x)\in(0,\infty)\times{\mathbb{R}}^n
$$ 
thus extends to a bounded linear operator on $L^2(L^2)$.
\end{remark}

\begin{proposition}
\label{prop:AKMP}
Let $q \in [1,2)$ be such that 
$\sup_{t>0}\|\sqrt{t}\nabla e^{-tL^*}\|_{{\mathscr{L}}(L^s)}<\infty$
for all $s\in[2,q')$. Then $\tilde{\mathcal{M}}_L$ extends to a bounded operator 
on $T^{p,2}$ for all $p\in (p_c,\infty]$ where $p_c=\max\bigl\{\frac {nq}{n+q},
\frac{2n}{n+q'}\bigr\}$. 
\end{proposition}

\begin{proof} 
We first recall, from \cite[Section 3.4]{Au07}, that
there exist an exponent $q$ as above (denoted by $q_+(L^{*})'$ in \cite{Au07}), 
and another one $p_{-}(L) \ge 1$ with $p_{-}(L) \le\max\bigl\{1,\frac {nq}{n+q}\bigr\}$ 
such that
\begin{equation}
\label{eq:p-}
\sup_{t\ge0}\|e^{-tL}\|_{{\mathscr{L}}(L^r)}<\infty,\quad\forall\,r\in(p_{-}(L),2].
\end{equation}
To prove the result for $p\le 2$, we apply \cite[Theorem~3.1]{AKMP12} 
with $m=2, \beta=0$. To do so, we only have to show that 
$$
\sup_{t>0}\bigl\|t^{1+\frac{n}{2}(\frac{1}{\tilde q}-\frac{1}{2})}
\nabla e^{-tL}{\rm div}\,\bigr\|_{{\mathscr{L}}(L^{\tilde q},L^2)}<\infty
$$
for all $\tilde q\in(q,2]$ and compute the exponents. Indeed, this estimate and 
$L^2-L^2$ off diagonal estimates imply the $L^r-L^2$ decay with $r\in(\tilde q,2)$. 
See for example \cite[Proposition 3.2]{Au07}. Write
$$
t^{1+\frac{n}{2}(\frac{1}{\tilde q}-\frac{1}{2})}
\nabla e^{-tL}{\rm div}= A_tB_tC_t
$$
with 
$A_t= t^{\frac{1}{2}}\nabla e^{-\frac{t}{3} L}$, 
$B_t= t^{\frac{n}{2}(\frac{1}{\tilde q}-\frac{1}{2})}e^{-\frac{t}{3}L}$ and 
$C_t= t^{\frac{1}{2}}  e^{-\frac{t}{3} L}{\rm div}$. Observe that  $C_t$ is 
uniformly bounded on $L^{\tilde q}$ using $\tilde q > q$ and duality. 
Next, $B_t$ is uniformly bounded from $L^{\tilde q}$ to $L^2$ by 
\cite[Proposition~3.9]{Au07} and \eqref{eq:p-}. Finally $A_{t}$ is uniformly 
bounded on $L^2$. For $p\ge 2$, we apply \cite[Proposition~4.2]{AKMP12} 
with $m=2$, $\beta=0$ and $q=2$.
\end{proof}

\begin{remark} 
If we were to use maximal regularity results in $L^p(L^p)$ or 
$L^p({\mathbb{R}}^n;L^2(0,\infty))$ as in \cite{vNVW15} instead of this result, 
we would need the family $\{\nabla e^{-tL}{\rm div} \;;\; t>0\}$ 
to be $R$-bounded on $L^p({\mathbb{R}}^n)$. As shown in \cite{Au07}, this is 
false for $p<q$, and $q$ can be arbitrarily close to $2$. In the above 
proposition, however, we allow, at least, $p \in [\frac{2n}{n+2},\infty]$ 
(see \cite{Au07, AHM12}). 
\end{remark}

\begin{remark} 
\label{rem:huang}  
If  $q'>n$, then $p_c=\frac {nq}{n+q}<1$. When $q'\le n$, $p_c=\frac{2n}{n+q'}$. 
Actually, we have learned from Yi Huang (personal communication) that in this 
case, the exponent $p_c$ can be taken to be the smaller value $\frac {nq}{n+q}$,
using an improved version of \cite[Theorem 3.1]{AKMP12}. This value is in 
agreement with the number $p_{-}(L)$ above.  
\end{remark}
 
\begin{remark}
\label{rem:ell} 
We remark that given the ellipticity constants $\lambda,\Lambda$, 
there is $\varepsilon(\lambda,\Lambda)\in (0,\infty]$ such that 
$q_+(L^*)\ge 2+\varepsilon(\lambda,\Lambda)$ whenever $A$ satisfies
\eqref{ell}. This implies that 
$p_{-}(L) \le\max\bigl\{1,\frac {2n}{n+2}-\varepsilon'(\lambda,\Lambda,n)\bigr\}$ 
for such $L$. See again \cite[Section 3.4]{Au07}. 
\end{remark}   

We also consider the integral operator ${\mathcal{R}}_L$ 
initially defined as a bounded operator from $L^1(H^1)$, with 
$H^1=H^1({\mathbb{R}}^n; {\mathbb{C}}^n)$, to $L^\infty_{\rm loc}(L^2)$ by
\begin{equation}
\label{eq:semimaxreg}
{\mathcal{R}}_L f(t,x)=\int_0^t e^{-(t-s)L}{\rm div}\, f(s,\cdot)(x)\,{\rm d}s.
\end{equation}
Note that ${\mathscr{C}}_{c}({\mathbb{R}}^{n+1}_+;{\mathbb{C}}^n)$, 
the space of compactly supported continuous functions on 
${\mathbb{R}}^{n+1}_+$ into ${\mathbb{C}}^n$ is contained in 
$L^1(H^1)$ and is dense in $T^{p,2}$ (of ${\mathbb{C}}^n$-valued functions) 
for all $p\in (0,\infty)$. 

{\bf Remark (erratum):} 
The proof of Proposition \ref{prop:subMR} below has a gap. However it, was recently proved in \cite[Theorem 5.1]{ah}
that the result nonetheless holds in an interval $p_{-}(L)<p\leq \infty$. See \cite[Remark 5.4]{ah}, and note that, in the present paper, Proposition \ref{prop:subMR} is only used (in Lemma \ref{lem:babyDuhamel}, Theorem \ref{thm:pert} , Theorem \ref{thm:cont}, Proposition \ref{prop:NbyS}, and Theorem \ref{thm:fatou81})
for $L=-\Delta$ (in which case $p_{-}(-\Delta)= \frac{n}{n+1}$) when $1<p$, or $p=2$ for general $L$.
Hence all the results are valid. \\
It is also worth noting at this stage that \cite{ah} improves some related exponents. 
In particular, it confirms the validity of the above Remark \ref{rem:huang}. This, in turn, improves the ranges of Lemma \ref{lem:Aescalier}, Theorem \ref{thm:pert}, and Theorem \ref{thm:cont} to $p\in (\max(1, \frac{qn}{n+q}), 2)$.

\begin{proposition}
\label{prop:subMR}
Let $p \in (0,\infty]$.
The operator ${\mathcal{R}}_L$ extends to a bounded operator 
from $T^{p,2}$ to $X^{p}$.
\end{proposition}

In the proof below, and throughout the paper, we use dyadic annuli defined 
as follows. For $x\in {\mathbb{R}}^n$, $r>0$, set $S_1(x,r)=B(x,2r)$, and 
$S_j(x,r) = B(x,2^{j+1}r)\setminus B(x,2^j r)$ for $j\ge 2$.  

\begin{proof}
Let $f \in {\mathscr{C}}_{c}({\mathbb{R}}^{n+1}_+;{\mathbb{C}}^n)$. 
We have that, for almost every $(t,x) \in {\mathbb{R}}^{n+1}_+$,
$$
{\mathcal{R}}_L f(t,x) = \sum_{k=0}^{\infty} e^{-(1-2^{-k})tL}K_L f(2^{-k}t,x),
$$
where $K_L f(t,x)=\int_{\frac{t}{2}} ^t e^{-(t-s)L}{\rm div}\, f(s,\cdot)(x)\,{\rm d}s$.
Fix $x\in {\mathbb{R}}^n$ and $k \in {\mathbb{N}}\setminus \{0\}$.  
Since $\{e^{-tL} \;;\; t\geq 0\}$ satisfies 
Gaffney-Davies estimates (see \cite[\S2.3]{Au07}), we have that for any $\delta >0$, 
\begin{align*}
\Bigl(\fint_{\frac{\delta}{2}}^\delta &
\fint_{B(x,\sqrt{\delta})}|e^{-(1-2^{-k})tL} K_{L} f(2^{-k}t,y)|^2
\,{\rm d}y\,{\rm d}t\Bigr)^{\frac{1}{2}} 
\\ 
& \le
\sum_{j=1}^{\infty}
\Bigl(\fint_{\frac{\delta}{2}}^\delta
\fint_{B(x,\sqrt{\delta})}|e^{-(1-2^{-k})tL}(1\!{\rm l}_{S_{j}(x,\sqrt{\delta})}
K_{L} f(2^{-k}t,\cdot))(y)|^2\,{\rm d}y\,{\rm d}t\Bigr)^{\frac{1}{2}}
\\
& \lesssim 
\sum_{j=1}^{\infty} 2^{j\frac{n}{2}} e^{-c4^j}
\Bigl(\fint_{\frac{\delta}{2}}^\delta
\fint_{B(x,2^{j+1}\sqrt{\delta})}|K_{L} f(2^{-k}t,y)|^2
\,{\rm d}y\,{\rm d}t\Bigr)^{\frac{1}{2}}
\\ 
& \le
\sum_{j=1}^{\infty} 2^{j\frac{n}{2}} e^{-c4^j} \sup_{\delta'>0}
\Bigl(\fint_{\frac{2^{k}\delta'}{2}}^{2^{k}\delta'}
\fint_{B(x,2^{j+1+\frac{k}{2}}\sqrt{\delta'})}|K_{L} f(2^{-k}t,y)|^2
\,{\rm d}y\,{\rm d}t\Bigr)^{\frac{1}{2}} 
\\
& = \sum_{j=1}^{\infty} 2^{j\frac{n}{2}} e^{-c4^j} \sup_{\delta'>0}
\Bigl(\fint_{\frac{\delta'}{2}}^{\delta'}
\fint_{B(x,2^{j+1+\frac{k}{2}}\sqrt{\delta'})}|K_{L} f(t,y)|^2
\,{\rm d}y\,{\rm d}t\Bigr)^{\frac{1}{2}}.
\end{align*}
Note that this estimate also holds for $k=0$.
Now with $\delta'>0$,  and $j\ge 1$, we have that
\begin{align*}
\Bigl(\fint_{\frac{\delta'}{2}}^{\delta'} &
\fint_{B(x,2^{j+1+\frac{k}{2}}\sqrt{\delta'})}|K_{L} f(t,y)|^2
\,{\rm d}y\,{\rm d}t\Bigr)^{\frac{1}{2}}
\\
& \leq  \sum_{\ell=1}^{\infty} \Bigl(\fint_{\frac{\delta'}{2}}^{\delta'}
\fint_{B(x,2^{j+1+\frac{k}{2}}\sqrt{\delta'})}\Bigl|
\int_{\frac{t}{2}} ^t e^{-(t-s)L}{\rm div}\,
(1\!{\rm l}_{S_{\ell}(x,2^{j+1+\frac{k}{2}}\sqrt{\delta'})}f(s,\cdot))(y)\,{\rm d}s\Bigr|^2
\,{\rm d}y\,{\rm d}t\Bigr)^{\frac{1}{2}}.
\end{align*}
For $\ell=1$, and $t \in (\frac{\delta'}{2},\delta')$, we have that
\begin{align*}
\Bigl\|\int_{\frac{t}{2}}^t  e^{-(t-s)L}{\rm div}\,
(1\!{\rm l}_{S_{1}(x,2^{j+1+\frac{k}{2}}\sqrt{\delta'})}f(s,\cdot))(y)\,{\rm d}s\Bigr\|_2
\leq \int_{\frac{t}{2}}^{t} \frac{1}{\sqrt{t-s}} 
\bigl\|1\!{\rm l}_{B(x,2^{j+2+\frac{k}{2}}\sqrt{\delta'})}f(s,\cdot)\bigr\|_2 {\rm d}s,
\end{align*}
and thus
\begin{align*}
\Bigl(\fint_{\frac{\delta'}{2}}^{\delta'} &
\fint_{B(x,2^{j+1+\frac{k}{2}}\sqrt{\delta'})}\Bigl|\int_{\frac{t}{2}}^t 
e^{-(t-s)L}{\rm div}\,(1\!{\rm l}_{S_{1}(x,2^{j+1+\frac{k}{2}}\sqrt{\delta'})}
f(s,\cdot))(y)\,{\rm d}s\Bigr|^2\,{\rm d}y\,{\rm d}t\Bigr)^{\frac{1}{2}}
\\ 
& \lesssim
\Bigl(\int_{\frac{\delta'}{2}}^{\delta'} 
\Bigl(\int_{\frac{t}{2}}^{t} \frac{1}{\sqrt{\delta'}}\frac{1}{\sqrt{t-s}} 
\bigl\|(2^{j+1+\frac{k}{2}}\sqrt{\delta'})^{-\frac{n}{2}} 
1\!{\rm l}_{B(x,2^{j+2+\frac{k}{2}}\sqrt{\delta'})}f(s,\cdot)\bigr\|_2 {\rm d}s\Bigr)^2
\,{\rm d}t\Bigr)^{\frac{1}{2}}.
\end{align*}
By Schur's Lemma, we thus have that
\begin{align*}
\Bigl(\fint_{\frac{\delta'}{2}}^{\delta'} &
\fint_{B(x,2^{j+1+\frac{k}{2}}\sqrt{\delta'})}\Bigl|\int_{\frac{t}{2}}^t 
e^{-(t-s)L}{\rm div}\,(1\!{\rm l}_{S_{1}(x,2^{j+1+\frac{k}{2}}\sqrt{\delta'})}
f(s,\cdot))(y)\,{\rm d}s\Bigr|^2\,{\rm d}y\,{\rm d}t\Bigr)^{\frac{1}{2}}
\\ & \lesssim
\Bigl(\int_{\frac{\delta'}{2}}^{\delta'}
\fint_{B(x,2^{j+2+\frac{k}{2}}\sqrt{\delta'})}|f(t,y)|^2
\,{\rm d}y\,{\rm d}t\Bigr)^{\frac{1}{2}}
\le
\Bigl(\int_0^\infty
\fint_{B(x,2^{j+3+\frac{k}{2}}\sqrt{t})}|f(t,y)|^2\,{\rm d}y\,{\rm d}t\Bigr)^{\frac{1}{2}}.
\end{align*}
Let us now consider $\ell \ge 2$. We have that
\begin{align*}
\Bigl(\fint_{\frac{\delta'}{2}}^{\delta'} &
\fint_{B(x,2^{j+1+\frac{k}{2}}\sqrt{\delta'})}\Bigl|\int_{\frac{t}{2}}^t 
e^{-(t-s)L}{\rm div}\,(1\!{\rm l}_{S_{\ell}(x,2^{j+1+\frac{k}{2}}\sqrt{\delta'})}
f(s,\cdot))(y)\,{\rm d}s\Bigr|^2\,{\rm d}y\,{\rm d}t\Bigr)^{\frac{1}{2}}
\\ 
& \lesssim
\Bigl(\int_{\frac{\delta'}{2}}^{\delta'}\Bigl(
\int_{\frac{t}{2}}^t  \frac{1}{\sqrt{\delta'}}\frac{1}{\sqrt{t-s}}
2^{\ell n}e^{-c\frac{4^{\ell+j}2^{k}\delta'}{t-s}}
(\fint_{B(x,2^{j+\ell+2+\frac{k}{2}}\sqrt{\delta'})} |f(s,y)|^2\,{\rm d}y)^{\frac{1}{2}}
\,{\rm d}s\Bigr)^{2}
\,{\rm d}t\Bigr)^{\frac{1}{2}}
\\ &
\lesssim
2^{\frac{\ell n}{2}}e^{-\frac{c}{2}4^{\ell+j}2^k}
\Bigl(\int_{\frac{\delta'}{4}}^{\delta'}
\fint_{B(x,2^{j+\ell+2+\frac{k}{2}}\sqrt{\delta '})}|f(s,y)|^2
\,{\rm d}y\,{\rm d}s\Bigr)^{\frac{1}{2}}.
\end{align*}
For $p=\infty$, summing in $j,k,\ell$, and using the change of angle lemma \cite{Au11} in 
$T^{p,2}$, we have that
\begin{align*}
\|{\mathcal{R}}_L f\|_{X^p} 
&\lesssim \sum \limits _{j,k,\ell } 2^{\frac{(j+\ell)n}{2}}
e^{-\frac{c}{2}4^{\ell+j}2^k} 
\Bigl\|x\mapsto \Bigl(\int_0^\infty
\fint_{B(x,2^{j+2+\ell+\frac{k}{2}}\sqrt{s})}|f(s,y)|^2\,{\rm d}y\,{\rm d}s
\Bigr)^{\frac{1}{2}}\Bigr\|_{L^p} 
\\
&\lesssim \sum_{j,k,\ell } 2^{\frac{(j+\ell)n}{2}}
e^{-\frac{c}{2}4^{\ell+j}2^k} 2^{(j+\ell+\frac k 2 )\tau}\|f\|_{T^{p,2}} 
\lesssim \|f\|_{T^{p,2}},
\end{align*}
the number $\tau$ depending on $n$ and $p$. This suffices to sum. For $p=\infty$, we argue similarly. We note that the proof applies directly to any $f\in T^{\infty,2}$ and gives a meaning to 
${\mathcal{R}}_L f$.  \end{proof}

The operators ${\mathcal{R}}_L$ and $\tilde{\mathcal{M}}_L$ are related 
in the following way.

\begin{proposition}
\label{prop:RvM}
Let $p \in (p_c,\infty)$ as in Proposition \ref{prop:AKMP} and $f \in T^{p,2}$.
Then $\nabla {\mathcal{R}}_Lf \in T^{p,2}$ and
$\nabla {\mathcal{R}}_Lf = \tilde{\mathcal{M}}_L f$
in $T^{p,2}$.
\end{proposition}

\begin{proof}
Given Propositions~\ref{prop:AKMP} and \ref{prop:subMR}, we only have to 
show that, for $f \in {\mathscr{D}}$,
$\nabla {\mathcal{R}}_Lf = \tilde{\mathcal{M}}_L f$ in ${\mathscr{D}}'$. 
Let $g \in {\mathscr{D}}$.
As in the proof of Lemma \ref{lem:tildeM_L^*}, we have that 
(where $\langle\cdot,\cdot\rangle$ is the $L^2$ inner product)
\begin{align*}
\int_{\mathbb{R}}\langle \tilde{\mathcal{M}}_Lf(t,\cdot),g(t,\cdot)\rangle\,{\rm d}t
=&\int_{\mathbb{R}}\int_{\mathbb{R}} 1\!{\rm l}_{(0,\infty)}(t-s)
\langle\nabla e^{-(t-s)L}{\rm div}\,f(s,\cdot),g(t,\cdot)\rangle\,{\rm d}s\,{\rm d}t
\\
=&-\int_{\mathbb{R}}\int_{\mathbb{R}} 1\!{\rm l}_{(0,\infty)}(t-s)
\langle e^{-(t-s)L}{\rm div}\,f(s,\cdot),{\rm div}\,g(t,\cdot)\rangle\,{\rm d}s\,{\rm d}t
\\
=&-\int_{\mathbb{R}}\langle {\mathcal{R}}_Lf(t,\cdot),{\rm div}\,g(t,\cdot)\rangle
\,{\rm d}t
=\int_{\mathbb{R}}\langle \nabla {\mathcal{R}}_Lf(t,\cdot),g(t,\cdot)\rangle\,{\rm d}t.
\qedhere
\end{align*}
\end{proof}

\section{$L^2$-theory and energy solutions}
\label{sec3}

\subsection{The space $\dot W(0,\infty)$}
\label{subsec:struct}

We start with a structural lemma about distributions 
$u \in \dot{W}(0,\infty)$.
Note that it is not restricted to solutions of our problem. 

\begin{lemma}
\label{lem:struct}
For all $u\in\dot W(0,\infty)$ there exist a unique 
$v \in \dot W(0,\infty)\cap{\mathscr{C}}_0(L^2({\mathbb{R}}^n))$ 
and $c \in {\mathbb{C}}$ such that $u=v+c$.  Moreover, 
$$
\|v\|_{L^\infty(L^2)} \leq 
\sqrt{2\|u\|_{L^2(\dot H^1)}
\|\partial_t u\|_{L^2(\dot H^{-1})}}\ .
$$
\end{lemma}

\begin{proof}
Set $w=\partial_t u+\Delta u$, and let $g\in L^2(L^2)$ be such 
that $w={\rm div}\,g$.
Given $t\geq0$, we denote by $\tau_t g$ the time translation of $g$ defined 
by $\tau_t g(s,.)=g(s+t,.)$ for all $s>0$. We now set, for all $t\geq 0$,
$$
v(t) = -\int_t^\infty e^{(s-t)\Delta}w(s)\,{\rm d}s
= -\int_0^\infty e^{s\Delta}{\rm div}(\tau_t g)(s)\,{\rm d}s,
$$
where the integral is defined weakly as shown below. Indeed,
for $f \in L^2({\mathbb{R}}^n)$ and $t\geq 0$, 
we have that
$$
\int_0^\infty |\langle \tau_t g(s),\nabla e^{s\Delta} f\rangle|\,{\rm d}s
\le \|\tau_t g\|_{L^2(L^2)} \|(s,x)\mapsto \nabla e^{s\Delta}f(x)\|_{L^2(L^2)}
\le \textstyle{\frac{1}{\sqrt{2}}}\,\|\tau_t g\|_{L^2(L^2)} \|f\|_{L^2},
$$
where the last inequality follows from a simple Fourier multiplier estimate. 
The argument also gives
$$
\|v(t)-v(t')\|_{L^2} \le \textstyle{\frac{1}{\sqrt{2}}}\,\|\tau_t g-\tau_{t'} g\|_{L^2(L^2)} 
\quad \forall t,t'>0,
$$
and therefore $v \in {\mathscr{C}}([0,\infty);L^2)$ as well as 
$\displaystyle{\lim_{t \to \infty}\|v(t)\|_{L^2} = 0}$ 
as $\|\tau_{t} g\|_{L^2(L^{2})} \xrightarrow[\tau \to \infty]{} 0$ for all $g \in L^{2}(L^2)$.
We now prove that $u-v$ is equal to a constant.
By Remark~\ref{rem:tildeM_L^*} we have that
\begin{align*}
&\|\nabla v\|_{L^2(L^2)}=\|\tilde{\mathcal{M}}_{-\Delta}^*g\|_{L^2(L^2)}
\lesssim\|g\|_{L^2(L^2)},
\\[4pt]
\mbox{hence}\qquad&\|\Delta v\|_{L^2(\dot H^{-1})} \le
\|\nabla v\|_{L^2(L^2)}\lesssim\|g\|_{L^2(L^2)}.
\end{align*}
Moreover, $\partial_tv\in L^2(\dot H^{-1})$ and
$\partial_tv+\Delta v=w$ in $L^2(\dot H^{-1})$. Indeed, for all 
$\phi\in{\mathscr{D}}$ we have that
\begin{align*}
\langle \partial_tv,\phi\rangle=
-\langle v,\partial_t\phi\rangle
=&\int_0^\infty\Bigl\langle\int_t^\infty e^{(s-t)\Delta}w(s)\,{\rm d}s,
\partial_t\phi(t)\Bigr\rangle\,{\rm d}t
\\
=&\int_0^\infty\Bigl\langle w(s),\int_0^se^{(s-t)\Delta}\partial_t\phi(t)\,{\rm d}t
\Bigr\rangle\,{\rm d}s
\\
=&\int_0^\infty\Bigl\langle w(s),\int_0^s\bigl[\partial_t\bigl(e^{(s-t)\Delta}\phi(t)\bigr)
+e^{(s-t)\Delta}\Delta\phi(t)\bigr]\,{\rm d}t\Bigr\rangle\,{\rm d}s
\\[4pt]
=&\langle w,\phi\rangle -\langle v,\Delta\phi\rangle.
\end{align*}
Consider the distribution $h:=u-v \in \dot W(0,\infty)$; we have that
$\partial_th+\Delta h=0$ in $L^2(\dot H^{-1})$. 
Since $h \in L^2(\dot{H}^1)$, we have that $h\in L^2({\mathscr{S}}')$. 
We can thus take the partial Fourier transform ${\mathcal{F}}_x$ in the 
${\mathbb{R}}^n$ variable, and obtain that the distribution
$\phi={\mathcal{F}}_xh\in L^2({\mathscr{S}}')$
satisfies
$$
\partial_t\phi-|\xi|^2\phi=0\quad\mbox{in }{\mathscr{D}}'
$$
where $m(t,\xi)T$ denotes the multiplication of 
$T\in {\mathscr{D}}'$ by the function $m$, here the polynomial 
$(t,\xi)\mapsto|\xi|^2$.
Solving the first order differential equation away from $\xi=0$, there exists 
$\alpha\in{\mathscr{D}}'({\mathbb{R}}^n\setminus\{0\})$ 
such that 
$$
\phi=  e^{t|\xi|^2}\alpha  \quad\mbox{in  } 
{\mathscr{D}}'((0,\infty)\times{(\mathbb{R}}^n\setminus\{0\})).
$$
Since $\xi\phi\in L^2(L^2)$ we have that 
$\xi\alpha e^{t|\xi|^2}\in L^2(L^2({\mathbb{R}}^n\setminus\{0\}))$. But for any 
compact set $K\subset {\mathbb{R}}^n\setminus\{0\}$, Fubini's theorem tells us that
$$
\int_0^\infty\int_K |\xi\alpha(\xi)e^{t|\xi|^2}|^2\, \,{\rm d}\xi\,{\rm d}t =\infty
$$
unless $\alpha=0$ almost everywhere on $K$. 
Thus $\alpha=0$ in ${\mathscr{D}}'({\mathbb{R}}^n\setminus\{0\})$.
This implies that $\phi$ is supported in $(0,\infty)\times\{0\}$, and
hence there exists $\tilde c\in{\mathscr{D}}'(0,\infty)$ such that
$\phi=\tilde c\otimes\delta_0$. But $\partial_t\phi \in L^2(\dot H^{-1})$ so 
$\tilde{c}$ is constant. Taking the inverse partial Fourier transform, we have 
shown that there exists a constant $c\in{\mathbb{C}}$ such that $u=v+c$.

\medskip

\noindent
To prove uniqueness, let $v_1,v_2\in W(0,\infty)$ be
such that there exists $c_1,c_2\in {\mathbb{C}}$ with
$u=v_1+c_1=v_2+c_2$ and define $w=v_1-v_2$. We have that
$w\in {\mathscr{C}}_0(L^2)$ and $w=c_2-c_1$. Therefore,
$w=0$, hence $c_1=c_2$ and the decomposition is unique.

\medskip

\noindent
We now prove the norm estimate. We have already shown that
\begin{align*}
\sup_{t \geq 0} \|v(t)\|_{L^2} &
\le \frac{1}{\sqrt{2}}\, \|w\|_{L^{2}(\dot H^{-1})}
\le \frac{1}{\sqrt{2}}\,\bigl(\|\partial_{t} u\|_{L^2(\dot H^{-1})}
+\|\Delta u\|_{L^2(\dot H^{-1})}\bigr)
\\
&\le \frac{1}{\sqrt{2}}\,\bigl(\|\partial_{t} u\|_{L^2(\dot H^{-1})}
+\|u\|_{L^2(\dot{H}^1)}\bigr).
\end{align*}
We now apply the result to the scaled functions 
$u_a:(t,x)\mapsto a^{\frac{n}{2}}u(t,ax)$, and obtain that
$$
\sup_{t \ge 0} \|v(t)\|_{L^2} \le \frac{1}{\sqrt{2}}\,
\Bigl(\frac{1}{a}\,\|\partial_t u\|_{L^2(\dot H^{-1})}
+a\|u\|_{L^2(\dot{H}^1)}\Bigr),
$$
for all $a>0$. Optimising in $a$ gives that
\begin{equation*}
\sup_{t \ge 0} \|v(t)\|_{L^2} \le 
\sqrt{2\,\|\partial_t u\|_{L^2(\dot H^{-1})}
\|u\|_{L^2(\dot{H}^1)}}.
\qedhere
\end{equation*}
\end{proof}

\begin{remark}
For each $u \in \dot{W}(0,\infty)$, the above lemma gives the existence of the limit
$\displaystyle{\lim_{t \to 0}\,u(t,\cdot)}$ in $\mathscr{D}'({\mathbb{R}}^n)$, 
equal to $v(0)+c$. We call this limit the trace of $u$, and denote it by ${\rm Tr}(u)$.
\end{remark}

\begin{remark}
\label{rem:<v,dtv>}
It is a well-known fact that for $0\le a<b<\infty$, and 
$u,v\in \dot W(a,b)\cap {\mathscr{C}}([a,b];L^2)$, we 
have that $t\mapsto \langle u(t),v(t)\rangle\in W^{1,1}(a,b)$ and
$$
\bigl({}_{L^2}\langle u(\cdot),v(\cdot)\rangle_{L^2}\bigr)'
={}_{\dot H^{-1}}\langle u'(\cdot),v(\cdot)\rangle_{\dot H^1}
+{}_{\dot H^1}\langle u(\cdot),v'(\cdot)\rangle_{\dot H^{-1}} \in L^1(a,b)
$$
See, e.g., \cite[\S14]{iSem18}.
\end{remark}

\begin{remark}
Lemma~\ref{lem:struct} is wrong if one replaces $\dot{W}(0,\infty)$ by 
$\dot{W}(a,b)$ for some finite $a<b$. To see this, take 
$f \in \dot{H}^1({\mathbb{R}}^n)\setminus L^2({\mathbb{R}}^n)$ 
and set $u(t,x) = f(x)$ for all $(t,x) \in (a,b)\times {\mathbb{R}}^n$.
\end{remark}

\subsection{A priori energy estimates}

As a corollary of Lemma~\ref{lem:struct}, we obtain the following a priori 
energy estimate.

\begin{corollary}
\label{cor:equiv}
Let $u\in {\mathscr{D}}'$ be a global weak solution of \eqref{eq1}
such that $\nabla u \in L^2(L^2)$. Then there exists 
a constant $c\in{\mathbb{C}}$ such that $v:=u-c\in{\mathscr{C}}_0(L^2)$ 
and is norm decreasing,    
$\nabla v=\nabla u\in L^2(L^2)$, $v$ is a weak solution of \eqref{eq1} and
$$
\|v(0)\|_{L^2} = \|v\|_{L^{\infty}(L^2)}
\le \sqrt{2\Lambda}\,\|\nabla v\|_{L^2(L^2)}
\le \textstyle{\sqrt{\frac{\Lambda}{\lambda}}}\,\|v(0)\|_{L^2},
$$
where $v(0)=v(0,.)$, and 
$\lambda, \Lambda$ are the ellipticity constants from \eqref{ell}.
\end{corollary}

\begin{proof}
Since $\partial_t u={\rm div}\,g$ in $\mathscr{D}'$ for 
$g = A\nabla u \in L^2(L^2)$, we have that
$$
|\langle \partial_tu,\overline{\phi} \rangle| \le
\Lambda \|\nabla u\|_{L^2(L^2)}\|\nabla \phi\|_{L^2(L^2)},
$$
hence $\partial_t u \in L^{2}(\dot{H}^{-1})$. Thus $u \in \dot{W}(0,\infty)$ 
and Lemma~\ref{lem:struct} imply that there exists a constant $c\in{\mathbb{C}}$ 
such that $v:=u-c\in \dot W(0,\infty)\cap {\mathscr{C}}_0(L^2)$,  and
\begin{align*}
\|v\|_{L^\infty(L^2)}&
\le 
\sqrt{2\,\|\partial_t u\|_{L^2(\dot H^{-1})}\|\nabla u\|_{L^2(L^2)}}
\\
&\le \sqrt{2\,\|g\|_{L^2(L^2)} \|\nabla u\|_{L^2(L^2)}} 
\le \sqrt{2\Lambda}\,\|\nabla u\|_{L^2(L^2)}
=\sqrt{2\Lambda}\,\|\nabla v\|_{L^2(L^2)}.
\end{align*}
Moreover, as constants are trivial weak solutions of \eqref{eq1}, so is $v$.  
Let $b>a>0$. For all $U \in L^2(a,b;\dot{H}^{1}({\mathbb{R}}^n))$, we have 
that
$$
\int_a^b
{}_{\dot H^{-1}}\langle \partial_{s}v(s,\cdot), U(s,\cdot) \rangle_{\dot H^1}\,{\rm d}s 
= -\int_a^b \int_{{\mathbb{R}}^n} A(s,x)\nabla v(s,x)
\cdot\overline{\nabla U(s,x)}\,{\rm d}x\,{\rm d}s.
$$
For  $U=v$, Remark~\ref{rem:<v,dtv>} and ellipticity give that
\begin{align*}
\|v(a,\cdot)\|_{L^2}^2-\|v(b,\cdot)\|_{L^2}^2 & =  
- 2\,\Re e\int_a^b {}_{\dot H^{-1}} 
\langle\partial_sv(s,\cdot),v(s,\cdot)\rangle_{\dot H^1}\,{\rm d}s
\\
& = 2\,\Re e \int_a^b \int_{{\mathbb{R}}^n} A(s,x)\nabla v(s,x)\cdot
\overline{\nabla v(s,x)} \,{\rm d}x\,{\rm d}s
\\
& \ge 2\lambda\,\|\nabla v\|_{L^2(a,b;L^2)}^2.
\end{align*}
This gives the norm decreasing property and letting $a\to 0$ and $b\to \infty$, yields 
$2\lambda\|\nabla v\|_{{L^2(L^2)}}^2 \le\|v(0,\cdot)\|_{L^2}^2$.
This completes the proof of Corollary~\ref{cor:equiv}.
\end{proof}

These a priori estimates can be localised. This is well-known, but we include an 
argu\-ment for the convenience of the reader, and to record some explicit 
constants for later use.  

\begin{proposition}
\label{prop:nrjloc}
Let $(a,b)\subset(0,\infty)$, $x \in {\mathbb{R}}^n$, $r>0$.
Let $u \in L^2(a,b;H^1(B(x,2r)))$ be a local weak solution of \eqref{eq1} 
on $(a,b)\times B(x,2r)$. Then $u \in {\mathscr{C}}([a,b];L^2(B(x,r)))$ and 
there exists $\kappa>0$ such that for all $c\in (a,b]$, we have
\begin{align*}
\|u(b,\cdot)\|_{L^2(B(x,r))}^2 & \leq
\Bigl(\frac{4\kappa^2\Lambda^2}{\lambda r^2}+\frac{1}{b-a}\Bigr)
\int_a^b \|u(s,\cdot)\|_{L^2(B(x,2r))}^2\,{\rm d}t,
\\[4pt]
\int_c^b \|\nabla u(s,\cdot)\|_{L^2(B(x,r))}^2\,{\rm d}s 
&\le\frac{1}{\lambda(c-a)}\,
\Bigl(1+(b-a)\frac{4\kappa^2\Lambda^2}{\lambda r^2}\Bigr)
\int_a^b \|u(s,\cdot)\|_{L^2(B(x,2r))}^2 \,{\rm d}s.
\end{align*}
\end{proposition}

\begin{proof}
Let $\eta\in{\mathscr{C}}_c^\infty({\mathbb{R}}^n)$ be a 
real-valued function supported in $B(x,2r)$, such that $\eta(y)=1$ for all 
$y \in B(x,r)$, $\|\eta\|_{\infty} \le 1$, and $\|\nabla \eta\|_{\infty} \le \frac{\kappa}{r}$. 
We have that 
$$
\|\nabla(\eta u)\|_{L^2((a,b)\times{\mathbb{R}}^n)}
\le\frac{2\kappa}{r}\, \|u\|_{L^2((a,b)\times B(x,2r))}
+\|\nabla u\|_{L^2((a,b)\times B(x,2r))}
<\infty.
$$ 
Therefore, $\eta u \in L^2((a,b),H^1_0(B(x,2r)))$. Note that this space is the 
closure of ${\mathscr{C}}_c^\infty((a,b)\times B(x,2r))$ in $L^2((a,b),H^1(B(x,2r)))$. 
Let $\phi \in {\mathscr{C}}_c^\infty((a,b)\times B(x,2r))$. Since $u$ is a local 
weak solution, we have 
\begin{equation*}
\Bigl|\iint_{(a,b)\times B(x,2r)}  u(t,y) \overline{\partial_t \phi(t,y)} 
\,{\rm d}y\,{\rm d}t  \Bigr|
\le \Lambda \|\nabla u\|_{L^2((a,b)\times B(x,2r))} 
\|\nabla \phi\|_{L^2((a,b)\times B(x,2r))}.
\end{equation*}
Using the known duality between $H^1_0(\Omega)$ and $H^{-1}(\Omega)$ 
for any open subset $\Omega$ of ${\mathbb{R}}^n$, this shows that 
$\partial_t u \in L^2((a,b),H^{-1}(B(x,2r)))$ 
and the same holds for $\partial_t(\eta u)$. Moreover, the integral on the left is 
$-{}_{L^2(a,b;H^{-1}(B(x,2r)))}\langle\partial_tu, 
\phi\rangle_{L^2(a,b;H^1_0(B(x,2r)))}$.
By Lions' result \cite[Proposition~3.1]{Li57}, 
$\eta u \in {\mathscr{C}}([a,b];L^2(B(x,2r)))$
(see also \cite[Theorem 1, Chapter XVIII]{DL88}). Calculating for all $a' \in (a,b)$:
\begin{align*}
&\|\eta u(b,\cdot)\|_{L^2}^2 - \|\eta u(a',\cdot)\|_{L^2}^2  
= 2\,\Re e \int_{a'}^b {}_{H^{-1}(B(x,2r))}\langle
\partial_t (\eta u)(t,\cdot),\eta u(t,\cdot)\rangle_{H^1_{0}(B(x,2r))}\,{\rm d}t
\\ 
= &\,2\,\Re e \int_{a'}^b {}_{H^{-1}(B(x,2r))}\langle
\partial_t u(t,\cdot),\eta^{2} u(t,\cdot)\rangle_{H^1_{0}(B(x,2r))}\,{\rm d}t
\\ 
=&-2\,\Re e \int_{a'}^b \int_{B(x,2r)} 
\eta(y)A(t,y)\nabla u(t,y)\cdot\eta(y)\overline{\nabla u(t,y)}\,{\rm d}y\,{\rm d}t
\\
&+4\,\Re e \int_{a'}^b \int_{B(x,2r)} \eta(y)A(t,y)\nabla u(t,y)
\cdot \overline{u(t,y)}\nabla\eta(y)\,{\rm d}y\,{\rm d}t.
\end{align*}
Therefore,
\begin{align*}
\|\eta u(b,\cdot)\|_{L^2}^2 &+2\lambda \int_{a'}^b 
\|\eta \nabla u(s,\cdot)\|_{L^2}^2 \,{\rm d}s 
\\
\leq &
\|\eta u(a',\cdot)\|_{L^2}^2 
+\lambda \int_{a'}^b \Bigl(\|\eta\nabla u(s,\cdot)\|_2^2
+\frac{4\kappa^2\Lambda^2}{\lambda r^2} 
\|u(s,\cdot)\|_{L^2(B(x,2r))}^2 \Bigr) \,{\rm d}s,
\end{align*}
and thus
\begin{equation}
\label{eq:control}
\|\eta u(b,\cdot)\|_2^2 +\lambda \int_{a'}^{b} 
\|\eta \nabla u(s,\cdot)\|_2^2\,{\rm d}s 
\leq \|\eta u(a',\cdot)\|_2^2
+\int_{a'}^{b} \frac{4\kappa^2\Lambda^2}{\lambda r^2} 
\|u(s,\cdot)\|_{L^{2}(B(x,2r))}^2\,{\rm d}s.
\end{equation}
Integrating in $a'$ between $a$ and $b$ gives the inequalities:
\begin{align*}
\|u(b,\cdot)\|_{L^2(B(x,r))}^2
&\le \Bigl(\frac{1}{b-a}+\frac{4\kappa^2\Lambda^2}{\lambda r^2}
\Bigr)
\int_a^b \|u(s,\cdot)\|_{L^2(B(x,2r))}^2\,{\rm d}s,
\\
\lambda (c-a) \int_c^b \|\nabla u(s,\cdot)\|_{L^2(B(x,r))}^2\,{\rm d}s 
&\le \lambda\int_a^b (s-a) \|\nabla u(s,\cdot)\|_{L^2(B(x,r))}^2\,{\rm d}s
\\
&\le
\Bigl(1+(b-a)\frac{4\kappa^2\Lambda^2}{\lambda r^2}\Bigr)
\int_a^b \|u(s,\cdot)\|_{L^2(B(x,2r))}^2\,{\rm d}s.
\qedhere
\end{align*}
\end{proof}

\begin{remark} 
\label{rem:local}
The above proof shows that whenever $u$ is a weak solution on 
$(a,b)\times \Omega$ with $u\in L^2(a,b; H^1(\Omega))$ then 
$\partial_t u\in L^2(a,b; H^{-1}(\Omega))$. One can thus take any 
$\varphi\in L^2(a,b; H^1_0(\Omega))$ as a test function in \eqref{eq1} and the 
integral $\iint u\overline{\partial_t\varphi}$ can be reinterpreted as  
$- \int \langle \partial_t u(t,\cdot), \varphi(t,\cdot)\rangle \,{\rm d}t$, where the 
brackets correspond to the  $H^{-1}(\Omega), H^1_0(\Omega)$ duality. Also 
$u\in {\mathscr{C}}([a,b];  L^2(\Omega'))$ for any $\Omega' $ with 
$\overline{\Omega'}\subset \Omega$.
\end{remark}
  
Similar estimates hold for the backward equation up to a time $T >0$:
\begin{equation}
\label{eq2}
\partial_s \phi(s,x)=-{\rm div}\,A(s,\cdot)^*\nabla \phi(s,x),
\quad 0\leq s \leq T, \quad x\in {\mathbb{R}}^n.\\
\end{equation}
Again a weak solution  to this equation on $(a,b)\times \Omega$ is a function 
$\phi \in L^2_{\rm loc}(a,b; H^1_{\rm loc}(\Omega))$ 
such that for all $\psi\in{\mathscr{C}}_c^\infty((a,b)\times\Omega)$,
\begin{equation}
\label{weak-backwards}
-\int_a^b \int_{\Omega}\phi(s,x)
\overline{\partial_s\psi(s,x)}\,{\rm d}x\,{\rm d}s
=\int_a^b\int_{\Omega}A(s,x)^{*}\nabla \phi(s,x)\cdot 
\overline{\nabla\psi(s,x)}\,{\rm d}x\,{\rm d}s.
\end{equation}

\begin{lemma}
\label{lem:link-backward-forward}
Let $\phi$ be a weak solution of \eqref{eq2} on $(0,T)\times \Omega$. Then 
$u:(t,x)\mapsto \phi(T-t,x)$ is a local weak solution on $(0,T)\times \Omega$ 
of \eqref{eq1} in which the matrices $A(t,x)$ are replaced by $A(T-t,x)^*$, 
$t\in[0,T]$, $x\in\Omega$.
\end{lemma}

\begin{proof}
Let $\psi\in{\mathscr{C}}_c^\infty((0,T)\times\Omega)$. Then
$\tilde\psi:(t,x)\mapsto\psi(T-t,x)\in{\mathscr{C}}_c^\infty((0,T)\times\Omega)$
and $\partial_t\tilde\psi(t,x)=-(\partial_t\psi)(T-t,x)$ for all $t\in[0,T]$ and
all $x\in\Omega$. Therefore, we have
\begin{align*}
\int_0^T\int_{\Omega}u(t,x)\overline{\partial_t\psi(t,x)}\,{\rm d}x\,{\rm d}t
&=\int_0^T\int_{\Omega}\phi(T-t,x)\overline{\partial_t\psi(t,x)}
\,{\rm d}x\,{\rm d}t
\\[4pt]
&=-\int_0^T\int_{\Omega}\phi(s,x)\overline{\partial_s\tilde\psi(s,x)}
\,{\rm d}x\,{\rm d}s
\\[4pt]
&=\int_0^T\int_{\Omega}A(s,x)^*\nabla\phi(s,x)\cdot
\overline{\nabla\tilde\psi(s,x)}\,{\rm d}x\,{\rm d}s
\\[4pt]
&=\int_0^T\int_{\Omega}A(T-t,x)^*\nabla u(t,x)\cdot\overline{\nabla\psi(t,x)}
\,{\rm d}x\,{\rm d}t
\end{align*}
where we have made the change of variable $s:=T-t$ twice and we have used 
\eqref{weak-backwards}.
\end{proof}

\begin{proposition}
\label{prop:nrjloc2}
Let $\phi\in L^2(a,b;H^1(B(x,2r)))$ be a weak solution of \eqref{eq2} on  
$(a,b)\times B(x,2r)$. Then $\phi \in {\mathscr{C}}([a,b];L^2(B(x,r)))$ 
and there exists $\kappa>0$ such that for all $d\in [a, b)$, we have
\begin{align*}
\|\phi(a,.)\|_{L^2(B(x,r))}^2 & \le
\Bigl(\frac{4\kappa^2}{\lambda r^2}+\frac{1}{b-a}\Bigr)
\int_a^b \|\phi(s,\cdot)\|_{L^2(B(x,2r))}^2,{\rm d}s,
\\
\int_a^d \|\nabla \phi(s,\cdot)\|_{L^2(B(x,r))}^2\,{\rm d}s 
&\le \frac{1}{\lambda(b-d)}\,
\Bigl(1+(b-a)\frac{4\kappa^2\Lambda^2}{\lambda r^2}\Bigr)
\int_a^b\|\phi(s,\cdot)\|_{L^2(B(x,2r))}^2\,{\rm d}s.
\end{align*}
\end{proposition}

\begin{proof}
Thanks to Lemma~\ref{lem:link-backward-forward}, we may apply the 
result of Proposition~\ref{prop:nrjloc} to $u(t,x):=\phi(a+b-t,x)$ for
$t\in (a,b)$, $x\in{\mathbb{R}}^n$ and $c:=a+b-d\in(a,b]$.
\end{proof}

\subsection{Well-posedness of energy solutions}
\label{subsec:wellposedness}

\begin{definition}
\label{def:wellposed}
Let $u_0\in L^2({\mathbb{R}}^n)$. The problem
$$
\partial_tu={\rm div}\,A\nabla u,
\quad
u\in \dot{W}(0,\infty),
\quad 
{\rm Tr}(u)=u_0
$$
is said to be well-posed if there exists a unique $u\in \dot{W}(0,\infty)$ 
global weak solution of \eqref{eq1} such that ${\rm Tr}(u)=u_0$.
\end{definition}

\begin{theorem}
\label{thm:wellposed}
For all $u_0\in L^2({\mathbb{R}}^n)$, the problem
$$
\partial_tu={\rm div}\,A\nabla u,
\quad
u\in \dot{W}(0,\infty),
\quad 
{\rm Tr}(u)=u_0
$$
is well-posed. Moreover, $u \in {\mathscr{C}}_0([0,\infty);L^2)$,  
$\|u(t,\cdot)\|_{L^2}$ is non increasing and 
$$
\|u_0\|_{L^2} = \|u\|_{L^\infty(L^2)}
\le \sqrt{2\Lambda} \|\nabla u\|_{L^2(L^2)}\le 
\textstyle{\sqrt{\frac{\Lambda}{\lambda}}}\,\|u_0\|_{L^2}.
$$
\end{theorem}

With some care because we are dealing with an unbounded time interval, 
it is possible to adapt the proof of Lions \cite[Theorem~5.1]{Li57} for the existence 
in order to construct a solution in $W(0,\infty)$. 
Nevertheless, we give a constructive approach to the $L^2$-existence theory, 
that plays a key role in the $L^p$ theory developed in later sections. 
The approach relies on approximations of $A$ and on taking weak$^*$ limits of 
the corresponding sequences of approximate solutions. We thus need the 
following lemma.

\begin{lemma}
\label{lem:Ak}
Let $A_k\in L^\infty
\bigl((0,\infty);L^\infty({\mathbb{R}}^n;{\mathscr{M}}_n({\mathbb{C}}))\bigr)$
for $k\in {\mathbb{N}}$ be such that \eqref{ell} holds uniformly in $k$ and
$$
A_k(t,x)\xrightarrow[k\to\infty]{} A(t,x)\quad\mbox{for almost every }
(t,x)\in(0,\infty)\times{\mathbb{R}}^n.
$$
Let $u_k$ be a global weak solution of $\partial_tu={\rm div}\,A_k\nabla u$ for 
all $k\in{\mathbb{N}}$, and assume that
$$
\sup_{k\in{\mathbb{N}}}\bigl(\|u_k\|_{L^\infty(L^2)}+\|\nabla u_k\|_{L^2}\bigr)<\infty.
$$
Then there exists a subsequence $(u_{k_j})_{j\in{\mathbb{N}}}$ such that
$(u_{k_j})_{j\in{\mathbb{N}}}$ weak$^*$ converges to $u$ in
$L^\infty(L^2)$ and 
$(\nabla u_{k_j})_{j\in{\mathbb{N}}}$ weak$^*$ converges in
$L^2(L^2)$. The limit $u\in L^\infty(L^2)$
is then a global weak solution of \eqref{eq1} such that $\nabla u\in L^2(L^2)$.
\end{lemma}

\begin{proof} 
Let $k\in {\mathbb{N}}$. Note first that $u_k \in \dot{W}(0,\infty)$ since it is a weak 
solution of $\partial_tu={\rm div}\,A_k\nabla u$ such that 
$\nabla u_k \in L^{2}(L^2)$. Since $u_k\in L^\infty(L^2)$,  
Lemma \ref{lem:struct} gives us that
$u_k \in{\mathscr{C}}_0(L^2)$. Therefore $(u_k(0,\cdot))_{k\in{\mathbb{N}}}$ 
is uniformly bounded in $L^2({\mathbb{R}}^n)$.
Moreover $(\nabla u_k)_{k\in{\mathbb{N}}}$ is uniformly bounded in $L^2(L^2)$.
We can thus extract a subsequence $(u_{k_j})_{j\in{\mathbb{N}}}$ using 
Banach-Alaoglu's theorem for which there exists $u\in L^\infty(L^2)$
and $u_0\in L^2({\mathbb{R}}^n)$ with
$$
\begin{array}{rcll}
u_{k_j}&\xrightarrow[j\to\infty]{}&u&\quad \mbox{weak$^*$ in }L^\infty(L^2),
\\
\nabla u_{k_j}&\xrightarrow[j\to\infty]{}&\nabla u&\quad \mbox{weak$^*$ in }
L^2(L^2),
\\
u_{k_j}(0,\cdot)&\xrightarrow[j\to\infty]{}&u_0&\quad \mbox{weak$^*$ in }L^2.
\end{array}
$$
For all $\phi\in{\mathscr{D}}({\mathbb{R}}^n)$ and all $t\ge 0$, 
Remark~\ref{rem:<v,dtv>} and the fact that $u_k \in {\mathscr{C}}_0(L^2)$ for all 
$k \in {\mathbb{N}}$ give that
$$
\int_{{\mathbb{R}}^n}u_{k_j}(t,y)\overline{\phi(y)}\,{\rm d}y=
\int_{{\mathbb{R}}^n}u_{k_j}(0,y)\overline{\phi(y)}\,{\rm d}y
-\int_0^t\int_{{\mathbb{R}}^n}A_{k_j}(s,y)\nabla u_{k_j}(s,y)
\cdot \overline{\nabla \phi(y)}\,{\rm d}y\,{\rm d}s.
$$
Since the right hand side converges to 
$\int_{{\mathbb{R}}^n}u_0(y)\overline{\phi(y)}\,{\rm d}y-
\int_0^t\int _{{\mathbb{R}}^n}A(s,y)\nabla u(s,y)\cdot \overline{\nabla \phi(y)}
\,{\rm d}y\,{\rm d}s$, the left hand side converges and its limit is equal to
$\int_{{\mathbb{R}}^n}u(t,y)\overline{\phi(y)}\,{\rm d}y$ for almost every $t>0$. 
Modifying $u$ for almost no $t>0$, we can assume that the equality holds 
everywhere. Differentiating in $t$ proves that
$\partial_tu(t,\cdot)={\rm div}\,A(t,\cdot)\nabla u(t,\cdot)$ in $\dot H^{-1}$
for almost every $t> 0$. Therefore $\partial_tu={\rm div}\,A\nabla u$ in
$L^2(\dot H^{-1})$ and thus $u$ is a weak solution of \eqref{eq1}.
\end{proof}

\begin{remark}
It is even possible to show strong convergence if $u_k(0, \cdot) $ are 
independent of~$k$. 
\end{remark}

\begin{proof}[Proof of Theorem~\ref{thm:wellposed}]
We start with the proof of existence of a solution $u\in W(0,\infty)$
satisfying \eqref{eq1} and $u(0,\cdot)=u_0$.

\medskip

\noindent
{\tt Step 0:} We first consider $A$ independent of $t$. We let $L=-{\rm div}\,A\nabla$ 
and $u(t)= e^{-tL}u_0.$ From semigroup theory, we know that 
$u \in {\mathscr{C}}_0([0,\infty);L^2({\mathbb{R}}^n))\cap 
\mathscr{C}^\infty(0,\infty; D(L))$ is a (strong) solution of $\partial_t u + Lu=0$. 
Moreover, $\nabla u\in L^2(L^2)$ and 
$$
2\lambda\|\nabla u\|_{L^2(L^2)}^2  
\le 2 \Re e \int_{0}^\infty {}_{L^2}\langle A\nabla u(t), \nabla u(t) \rangle_{L^2} 
\, {\rm d}t 
= -\int_{0}^\infty (\|u(t)\|_{L^2}^2)'\,{\rm d}t=\|u_{0}\|_{L^2}^2.
$$
Finally, one easily checks that $u$ is a global weak solution as well. 

\medskip

\noindent
{\tt Step 1:}
We next consider $A$ of the form
$$
A(t,x)=\sum_{k=0}^N 1\!{\rm l}_{[t_k,t_{k+1})}(t)A_k(x)+
1\!{\rm l}_{[t_{N+1},+\infty)}(t)A_{N+1}(x)
$$
for some $N\in{\mathbb{N}}$, $(t_k)_{0\le k\le N+1}$ an increasing sequence
in $[0,\infty)$ with $t_0=0$ and $(A_k)_{0\le k\le N+1}$ satisfying
\eqref{ell} uniformly. It is convenient to set $t_{N+2}=\infty$. For $j=0,...,N+1$, 
let $L_j=-{\rm div}\,A_j\nabla$ and define 
$$
\Gamma_A(t,s):=e^{-(t-t_j)L_j} e^{-(t_j-t_{j-1})L_{j-1}}
\dots e^{-(t_{i+1}-s)L_i}
$$
for $t\in[t_j,t_{j+1})$ and $s\in[t_i,t_{i+1})$. We define 
$u:t\mapsto \Gamma_{A}(t,0)u_0, t\ge 0$. That 
$u \in {\mathscr{C}}_0([0,\infty);L^{2}({\mathbb{R}}^n))$ is easily established 
using the properties of the semigroups $\bigl(e^{-tL_j}\bigr)_{t\ge0}$. 
We proceed inductively on $k$ to check the desired properties on $u$. Since
$-L_0$ generates an ana\-ly\-tic semigroup of contractions
$$
\bigl\|(t,x)\mapsto 1\!{\rm l}_{(0,t_1)}\Gamma_A(t,0)u_0(x)
\bigr\|_{L^{\infty}(L^2)}\leq \|u_0\|_{L^2}.
$$
Therefore
\begin{equation*}
\bigl\|(t,x)\mapsto 1\!{\rm l}_{(0,t_1)}(t)\nabla u(t,x) \bigr\|_{L^2(L^2)} 
\le 
\bigl\|(t,x)\mapsto 1\!{\rm l}_{(0,t_1)}(t)\nabla e^{-tL_{0}} u_0(x) \bigr\|_{L^2(L^2)}
\lesssim
\|u_0\|_{L^2},
\end{equation*}
Moreover $\partial_tu(t,\cdot)\in L^2({\mathbb{R}}^n)$ for all $t\in (0,t_1)$ and
$\partial_t(u(t,\cdot))=L_0u(t,\cdot)=L(t)u(t,\cdot)$ in $L^2({\mathbb{R}}^n)$ for
all $t\in(0,t_1)$. Now let $k\le N+1$ and assume that the following holds:
\begin{align*}
&\bigl\|(t,x)\mapsto 1\!{\rm l}_{(0,t_k)}(t)\Gamma_A(t,0)u_0(x)
\bigr\|_{L^{\infty}(L^2)}
\leq \|u_0\|_{L^2},
\\[4pt]
&\bigl\|(t,x)\mapsto 1\!{\rm l}_{(0,t_k)}(t)\nabla \Gamma_A(t,0)u_0(x)
\bigr\|_{L^2(L^2)}
\lesssim \|u_0\|_{L^2},
\\[4pt]
\mbox{and}\quad& 
\partial_tu(t,\cdot)=L(t)u(t,\cdot) \mbox{ in }L^2({\mathbb{R}}^n) 
\mbox{ for all }
t\in(0,t_k)\setminus\bigl\{t_0,\dots,t_{k-1}\bigr\}.
\end{align*}
Here, the implicit constants may depend on $N$ but we are inducting on a finite 
number of steps and we will get the dependence only on the ellipticity constants 
in \eqref{ell} eventually. We want to extend all this to $t_{k+1}$. 
For $t\in[t_k,t_{k+1})$ we have that
$$
u(t,\cdot)=\Gamma_A(t,s)u(s,\cdot)=e^{-(t-t_k)L_k}e^{-(t_k-s)L_{k-1}}u(s,\cdot)
$$
for all $s\in(t_{k-1},t_k)$. Therefore
$$
\bigl\|(t,x)\mapsto 1\!{\rm l}_{(0,t_{k+1})}(t)u(t,x)\bigr\|_{L^\infty(L^2)}
\leq \bigl\|(t,x)\mapsto 1\!{\rm l}_{(0,t_k)}(t)u(t,x)\bigr\|_{L^\infty(L^2)}
\leq \|u_0\|_{L^2}.
$$
Using $u(t,\cdot) = e^{-(t-t_k)L_k}u(t_k,\cdot)$, we have
$$
\bigl\|(t,x)\mapsto 1\!{\rm l}_{(t_k,t_{k+1})}(t)\nabla u(t,x)\bigr\|_{L^2(L^2)}
\lesssim \bigl\|u(t_k,\cdot)\bigr\|_{L^2}
\le \|u_0\|_{L^2}.
$$
We also have that $\partial_tu(t,\cdot)=-L_ku(t,\cdot)=-L(t)u(t,\cdot)$ in
$L^2({\mathbb{R}}^n)$ for all $t\in(t_k,t_{k+1})$.
This concludes the induction, which proves that 
$u\in L^\infty(L^2)\cap L^{2}(\dot{H}^1)$, and that $u$ satisfies
$$
\partial_tu(t,\cdot)=-L(t)u(t,\cdot) \quad 
\forall t \in (0,\infty)\setminus \bigl\{t_k \;;\; k \in {\mathbb{N}}\bigr\}.
$$
We now show that $u$ is a global weak solution of \eqref{eq1}. Let 
$\phi\in{\mathscr{D}}$, and pick $M>t_{N+1}$ such that 
${\rm supp}\,\phi \subset (0,M)\times {\mathbb{R}}^n$. 
For $j=0,\dots,N+1$, $t \mapsto \langle u(t,\cdot),\phi(t,\cdot)\rangle$ (where 
$\langle\cdot,\cdot\rangle$ denotes the $L^2$ duality) is ${\mathscr{C}}^1$ on 
$(t_j,t_{j+1})$ and continuous on $[t_j,t_{j+1}]$, hence
$$
\int_{t_j}^{t_{j+1}} \langle u(t,\cdot),\partial_t\phi(t,\cdot)\rangle\;{\rm d}t 
= \langle u(t_{j+1},\cdot),\phi(t_{j+1},\cdot) \rangle 
- \langle u(t_j,\cdot),\phi(t_j,\cdot) \rangle + 
\int_{t_j}^{t_{j+1}} \langle L_ju(t,.),\phi(t,.) \rangle\,{\rm d}t.
$$
Summing in $j$ and using $\langle L_j u(t,\cdot),\phi(t) \rangle 
= -\langle A_j \nabla u(t,\cdot),\nabla \phi(t) \rangle$ for all $t \in (t_j,t_{j+1})$, 
and the fact that ${\rm supp}\,\phi \subset (0,M)\times {\mathbb{R}}^n$, we have that
$$
\int_0^\infty\int_{{\mathbb{R}}^n}u(t,y)\,\overline{\partial_t\phi(t,y)}\,{\rm d}y\,{\rm d}t
=\int_0^\infty\int_{{\mathbb{R}}^n}A(t,y)\nabla u(t,y)\cdot\overline{\nabla\phi(t,y)}
\,{\rm d}y\,{\rm d}t,
$$
i.e.
$u$ is a weak solution of \eqref{eq1}. 

Therefore, by Corollary~\ref{cor:equiv},
$\|\nabla u\|_{L^2(L^2)}\sim\|u_0\|_{L^2}$
with constants depending only on $\lambda$ and $\Lambda$ from \eqref{ell}.

\medskip

\noindent
{\tt Step 2:}
We now consider $A$ of the form
$$
A:(t,x)\mapsto \sum_{k=0}^\infty 1\!{\rm l}_{[t_k,t_{k+1})}(t)A_k(x)
$$
for some increasing sequence $(t_k)_{k\in{\mathbb{N}}}$ with $t_0=0$ and
$\displaystyle{\lim_{k\to\infty}t_k=+\infty}$ and $(A_k)_{t\in{\mathbb{N}}}$
satisfying \eqref{ell} uniformly. Define 
$$
{\mathcal A}_N:(t,x)\mapsto \sum_{k=0}^N 1\!{\rm l}_{[t_k,t_{k+1})}(t)A_k(x)
+1\!{\rm l}_{[t_{N+1},+\infty)}(t)A_{N+1}(x)
$$
for all $N\in{\mathbb{N}}$. Then ${\mathcal A}_N(t,x)\xrightarrow[N\to\infty]{} A(t,x)$
for almost every $(t,x)\in(0,\infty)\times{\mathbb{R}}^n$. Let
$(u_N)_{N\in{\mathbb{N}}}$ be the corresponding sequence of weak solutions
to $\partial_tu={\rm div}{\mathcal A}_N\nabla u$ obtained in the previous step. By
Lemma~\ref{lem:Ak}, there exists a subsequence $(u_{N_j})_{j\in{\mathbb{N}}}$
converging to $u\in L^\infty(L^2)$ in the weak$^*$ topology, with
$u$ a weak solution of $\partial_tu={\rm div}A\nabla u$ and
$\|u\|_{L^\infty(L^2)}+\|\nabla u\|_{L^2(L^2)}\lesssim \|u_0\|_{L^2}$
with constants depending only on the ellipticity constants.

\medskip

\noindent
{\tt Step 3:}
We now turn to the case where $A\in{\mathscr{C}}\bigl([0,\infty);
L^\infty({\mathbb{R}}^n;{\mathscr{M}}_n({\mathbb{C}}))\bigr)$. Approximating
$A$ almost everywhere by matrices of the form $\displaystyle{
(t,x)\mapsto \sum_{k=0}^\infty 1\!{\rm l}_{[t_k,t_{k+1})}(t)A_k(x)}$,
with $A_k = A(t_k,\cdot)$, which satisfy \eqref{ell} uniformly in $k$, 
we obtain from Step~2 a family of weak solutions $(u_j)_{j\in{\mathbb{N}}}$,
uniformly bounded in $L^\infty(L^2)$ and such that
$\displaystyle{\sup_{j\in{\mathbb{N}}} \|\nabla u_j\|_{L^2(L^2)}<\infty}$.
Using Lemma~\ref{lem:Ak} again we obtain a weak solution $u$ of \eqref{eq1}
such that $\|u\|_{L^\infty(L^2)}+\|\nabla u\|_{L^2(L^2)}\lesssim \|u_0\|_{L^2}$.

\medskip

\noindent
{\tt Step 4:}
Finally, for a general $A\in L^\infty\bigl((0,\infty);
L^\infty({\mathbb{R}}^n;{\mathscr{M}}_n({\mathbb{C}}))\bigr)$ we can use
the approximations
$$
\Bigl(\tilde A_j:(t,x)\mapsto j\int_t^{t+\frac{1}{j}}A(s,x)\,{\rm d}s\Bigr)
\in \mathscr{C}\bigl([0,\infty);
L^\infty({\mathbb{R}}^n;{\mathscr{M}}_n({\mathbb{C}}))\bigr),
$$
for $j\ge 1$ and use Step~3 together with Lemma~\ref{lem:Ak} one more time.

\medskip

\noindent
{\tt Step 5:} 
Let us now prove uniqueness of solutions.
Let $u,v\in \dot W(0,\infty)$ be solutions of \eqref{eq1} with 
${\rm Tr}(u)={\rm Tr}(v)=u_0$. The function $w:=u-v\in \dot W(0,\infty)$ is 
a global weak solution of \eqref{eq1} such that ${\rm Tr}(w)=0$. By 
Corollary~\ref{cor:equiv}, we have that there exists $c \in {\mathbb{C}}$, and 
$\tilde{w} \in {\mathscr{C}}_0(L^2)$ such that
$w=\tilde{w}+c$. Since ${\rm Tr}(w)=0$, we have that $c=0$ and 
$\displaystyle{\lim_{t\to 0} \tilde{w}(t,\cdot)=0}$ in $L^2({\mathbb{R}}^n)$.
Corollary~\ref{cor:equiv} thus yields
$\|w\|_{L^\infty(L^2)}=\|\tilde{w}\|_{L^\infty(L^2)}=0$.
\end{proof}

\subsection{Propagators}
\label{subsec:propa}

\begin{lemma}
\label{lem:propag}
There exists a family of contractions 
$\{\Gamma(t,s) \;;\; 0 \le s \le t <\infty\}\subset {\mathscr{L}}(L^2)$ such that
\begin{enumerate}[{\rm (1)}]
\item
$\Gamma(t,t) = I \quad \forall t\geq 0$.
\item
$\Gamma(t,s)\Gamma(s,r) = \Gamma(t,r) \quad \forall t\geq s \geq r\ge 0$.
\item
For all $h \in L^2({\mathbb{R}}^n)$, and $s\geq 0$, 
$t\mapsto \Gamma(t,s)h \in {\mathscr{C}}_0([s,\infty);L^2({\mathbb{R}}^n))$.
\item
For all $u_0 \in L^2({\mathbb{R}}^n)$, $(t,x)\mapsto \Gamma(t,0)u_0(x)$ 
is a global weak solution of \eqref{eq1}.
\end{enumerate}
\end{lemma}

\begin{proof}
Let $u_0 \in L^2({\mathbb{R}}^n)$. Let $u$ be the solution of the Cauchy problem 
in Theorem~\ref{thm:wellposed}. We have 
$u\in{\mathscr{C}}_0(L^2)\cap L^2(\dot H^1)$, 
with $\|u(t,\cdot)\|_{L^2}\le \|u_{0}\|_{L^2}$ and we define $\Gamma(t,0)$ as the 
contraction on $L^2$ that maps $u_0$ to $u(t,\cdot)$. Similarly, we can start from 
any time $s\ge 0$ and from any data $h\in L^2$, and obtain a unique solution 
$v \in \dot W(s,\infty)$ with $u(s,\cdot)=h$. We define $\Gamma (t,s)$ 
as the operator mapping $h$ to $u(t,\cdot)$ when $t\geq s$.
Then (1), (3) and (4) follow by construction, while (2) follows from uniqueness.
\end{proof}

\begin{definition}\label{def:propa}
We call $\{\Gamma(t,s) \;;\; 0 \le s \le t <\infty\}$ the family of propagators for 
\eqref{eq1}.
\end{definition}

The restriction $s\ge 0$ only comes from the fact that we work on 
$(0,\infty)\times {\mathbb{R}}^n$. This means that, provided that $A$ 
is defined on ${\mathbb{R}}^{n+1}$ and satisfies \eqref{ell}, one can define 
$\Gamma(t,s)$ for $-\infty<s\le t <\infty$ and we have the 
same properties on the full range of $s$ and $t$. One works on $(s,\infty)$ 
for arbitrary $s$ and by uniqueness, any two families are consistent on the 
common time intervals. There is a similar family for the backward equation 
\eqref{eq2}.

\begin{lemma}
\label{lem:4.2}
Let $T > 0$. There exists a family of contractions 
$\{\tilde{\Gamma}(t,T) \;;\; t\in (-\infty,T]\}\subset {\mathscr{L}}(L^2)$ such that
\begin{enumerate}
\item
$\tilde{\Gamma}(T,T) = I$.
\item
For all $h \in L^2({\mathbb{R}}^n)$,  $t\mapsto \tilde{\Gamma}(t,T)h 
\in {\mathscr{C}}_0((-\infty,T];L^2)$.
\item
For all $h \in L^2({\mathbb{R}}^n)$, $(t,x)\mapsto \tilde{\Gamma}(t,T)h(x)$ 
is a global weak solution of \eqref{eq2} on $(-\infty,T)$.
\end{enumerate}
\end{lemma}

\begin{proof}
Define 
$$
\tilde A(s,x) =
\begin{cases} 
A^*(T-s ,x) \quad \mbox{if } \; (s,x)\in (-\infty,T] \times {\mathbb{R}}^n,
\\
A^*(0,x)  \quad \mbox{if } \; (s,x)\in (T,\infty) \times {\mathbb{R}}^n. 
\end{cases}
$$
Applying Theorem~\ref{thm:wellposed} on $(0,\infty)$ 
with $A$ replaced by $\tilde A$ we get the conclusion of Lemma \ref{lem:propag}. 
Denoting the corresponding family of propagators by 
$\{{\underset{\widetilde{ }}{\Gamma}}(t,s) \;;\; 0\le s\le t <\infty \}\subset 
{\mathscr{L}}(L^2)$, we define
$$
\tilde{\Gamma}(t,T):= {\underset{\widetilde{ }}{\Gamma}}(T-t,0) 
\quad \forall t\in (-\infty,T].
$$
It is immediate that $\tilde{\Gamma}$ satisfies points 1 and 2. By 
Lemma~\ref{lem:link-backward-forward}, we have that 
$(t,x)\mapsto \tilde{\Gamma}(t,T)h(x)$ is a weak solution of 
\eqref{eq2} on $(-\infty,T)\times {\mathbb{R}}^n$, which proves point 3. 
\end{proof}

\begin{proposition}
\label{prop:adjointGamma}
Let $T>0$. The families of propagators for \eqref{eq1} and \eqref{eq2} 
(up to time $T$) are related by
$$
\tilde{\Gamma}(t,T) = \Gamma(T,t)^* \quad \forall t \in [0,T].
$$
In particular, for all $h\in L^2({\mathbb{R}}^n)$, $t\mapsto \Gamma(T,t)^*h$ 
is strongly continuous from $[0,T]$ into $L^2({\mathbb{R}}^n)$ and 
$t\mapsto \Gamma(T,t)h$ is weakly continuous from $[0,T]$ into 
$L^2({\mathbb{R}}^n)$. 
\end{proposition}

\begin{proof}
Let $g,h \in L^2({\mathbb{R}}^n)$, and $0\le t \le s \le T$.
Let $u(s,x) = \tilde{\Gamma}(s,T)h(x) = 
{\underset{\widetilde{ }}{\Gamma}}(T-s,0)h(x)$, 
and $v(s,x) = \Gamma(s,t)g(x)$ for all $(s,x) \in [t,T]\times {\mathbb{R}}^n$. 
Since $u,v \in \dot{W}(t,T)$, we have, for almost every $s \in [t,T]$,
\begin{align*}
\langle \partial_s u(s,\cdot), v(s,\cdot) \rangle-
\langle A^*(s,\cdot)\nabla u(s,\cdot),\nabla v(s,\cdot) \rangle = 0,
\\
\langle u(s,\cdot),\partial_s v(s,\cdot) \rangle+
\langle\nabla u(s,\cdot), A(s,\cdot)\nabla v(s,\cdot) \rangle = 0.
\end{align*}
We therefore have (see Remark~\ref{rem:<v,dtv>}):
\begin{align*}
0 &= \int_t^T \partial_s \langle u(s,.),v(s,.) \rangle \,{\rm d}s =
\langle u(T,.), v(T,.) \rangle - \langle u(t,.), v(t,.) \rangle
\\
&= \langle h,\Gamma(T,t)g\rangle - \langle \tilde{\Gamma}(t,T)h,g \rangle.
\qedhere
\end{align*}
\end{proof}

\begin{remark}
\label{rem:back} 
The restriction $T>0$ is irrelevant in the previous results and is only made 
because we study \eqref{eq1} on $(0,\infty)$. The adjoint formula is independent 
of the choice of the extension of $A^*(T-t,x)$ for $t>T$ to construct 
$\tilde \Gamma$ in Lemma~\ref{lem:4.2}.  It follows from this adjoint formula 
that any result we obtain for \eqref{eq1} involving the propagators 
$\Gamma(t,s)$ has its counterpart for the adjoint backward equation \eqref{eq2} 
globally on $(-\infty,T)$ or locally on $(S,T)$,  with the propagators $\Gamma(t,s)^*$, 
provided the hypotheses made on the coefficients are stable under taking adjoints. 
\end{remark}

A key property of $\Gamma$ is that it satisfies the following $L^2-L^2$ 
off-diagonal bounds.  

\begin{proposition}
\label{prop:OD}
For all Borel sets 
$E,F\subset {\mathbb{R}}^n$, all $f \in L^2({\mathbb{R}}^n)$ and all 
$0\le s<t<\infty$, 
$$
\bigl\|1\!{\rm l}_E\Gamma(t,s)(1\!{\rm l}_Ff)\bigr\|_{L^2} 
\le e^{-\alpha\frac{d(E,F)^{2}}{t-s}}\,\bigl\|1\!{\rm l}_Ff\bigr\|_{L^2},
$$
with $\alpha=\frac{\lambda}{4\Lambda^2}$,  where $\Lambda,\lambda$ are the 
ellipticity constants from \eqref{ell} and $d(E,F)$ denotes the Hausdorff distance 
between $E$ and $F$ with Euclidean norm. 
\end{proposition}

\begin{proof} 
This result is already in \cite{Ar67}. The simple proof with this constant is 
taken from \cite{HK04}. There, $A$ was assumed 
to be smooth but this is not necessary. It also adapts to systems with 
G\aa rding inequality instead of pointwise lower bounds. We reproduce the 
argument for the convenience of the reader. 
It is enough to assume $s=0$ as one can translate the origin of time to~$s$. 
Let $\psi$ be a non negative, Lipschitz and bounded function on 
${\mathbb{R}}^n$ with $|\nabla \psi|\le \gamma$. 
For $f\in L^2({\mathbb{R}}^n)$, set 
$\Gamma^\psi(t,0)f=e^\psi\Gamma(t,0)(e^{-\psi}f)
\in L^2({\mathbb{R}}^n)$ as $\psi$ is bounded. Observe that 
$u(t)= e^{-\psi}\Gamma^\psi(t,0)f= \Gamma(t,0)(e^{-\psi}f)$ 
is a global energy solution of \eqref{eq1}. Using Remark \ref{rem:<v,dtv>}, 
we have the chain of equalities and inequalities for almost every $t>0$:
\begin{align*}
 \frac{{\rm d}}{{\rm d}t} \|\Gamma^\psi(t,0)f\|_{L^2}^2  & 
 =  \frac{{\rm d}}{{\rm d}t} \|e^\psi u(t)\|_{L^2}^2  
 \\
& = 2\,\Re e\ {}_{\dot H^{-1}}\langle \partial_t(e^\psi u(t)),
e^\psi u(t)\rangle_{\dot H^1} 
\\
& = 2\,\Re e\  {}_{\dot H^{-1}}\langle \partial_t u(t),e^{2\psi} u(t)\rangle_{\dot H^1} 
\\
& = -  2\,\Re e\  {}_{ L^2}\langle A(t)\nabla u(t),\nabla (e^{2\psi} u(t))\rangle_{ L^2}
\\
& = -  2\,\Re e\  {}_{ L^2}\langle A(t)e^\psi\nabla u(t),e^\psi\nabla  u(t)\rangle_{L^2}  
- 4\,\Re e\  {}_{ L^2}\langle A(t)e^\psi\nabla u(t),e^\psi u(t) \nabla  \psi\rangle_{L^2} 
\\
& \le -2\lambda \|e^\psi \nabla u(t)\|_{L^2}^2 +4\Lambda\gamma  
\|e^\psi \nabla u(t)\|_{L^2}  \|e^\psi  u(t)\|_{L^2}
\\
& \le \frac{2\Lambda^2}{\lambda}\gamma^2 \|e^\psi u(t)\|_{L^2}^2.
\end{align*} 
As $\Gamma^\psi(t,0)f\to f $ in $L^2({\mathbb{R}}^n)$ as $t\to 0$, we get
$$
\|\Gamma^\psi(t,0)f\|_{L^2} \le e^{\kappa \gamma^2 t} \|f\|_{L^2}, 
\quad \kappa=\frac{\Lambda^2}{\lambda}.
$$
Assume now that ${\rm supp}\,f \subset F$ and let $\psi(x)=\inf (\gamma d(x,F), N)$ 
for a large $N>\gamma d(E,F)$. We obtain
$$
\|\Gamma(t,0)f\|_{L^2(E)} \le e^{-\gamma d(E,F)} \ \|\Gamma^\psi(t,0)f\|_{L^2} 
\le e^{\kappa \gamma^2 t - \gamma d(E,F)}\|f\|_{L^2}.
$$
Optimizing with $\gamma= \frac{d(E,F)}{2\kappa t}$ completes the proof. 
\end{proof}

\subsection{Connection with earlier constructions}

Suppose we have constructed $\Gamma(t,s)$ for all 
$-\infty < s\leq t <\infty$ as explained above after Definition~\ref{def:propa}. 

\begin{proposition} 
Fix $T>0$, let $u_0\in L^2({\mathbb{R}}^n)$ and $u(t,\cdot)=\Gamma(t,0)u_0$ 
for $t\ge 0$. Then $u$ agrees with Aronson's energy solution on 
$(0,T)\times {\mathbb{R}}^n$ and with Lions' energy solution on 
$(0,T)\times {\mathbb{R}}^n$ of \eqref{eq1}. In particular, 
for $0\le s\le  t\le T$, $\Gamma(t,s)$ agrees with both Aronson's and 
Lions' propagators. 
\end{proposition}

\begin{proof} 
We begin with Lions's construction \cite[Theorem~5.1]{Li57} 
(see also \cite[Chap.\,XVIII, \S3]{DL88}). 
He proves well-posedness of \eqref{eq1}  in the class $W(0,T)$ with data $u_0$. 
By our construction, we have that $u\in W(0,\infty)$, hence its restriction to 
$(0,T)$ belongs to $W(0,T)$. Thus, $u$ agrees with Lions's energy solution 
on $(0,T)$.

We turn to Aronson's construction \cite{Ar67}. This particular 
part of his article does not use the specificity of real coefficients. He proves 
well-posedness in the class 
$A_T=L^\infty(0,T; L^2({\mathbb{R}}^n))\cap L^2(0,T; H^1({\mathbb{R}}^n))$ 
with data $u_0$. By our construction, we have that $u\in W(0,\infty)$, hence its 
restriction to $(0,T)$ belongs to $A_T$. Thus, $u$ agrees with 
Aronson's energy solution on $(0,T)$. 

The consequence for the propagator $\Gamma(t,0)$ is immediate. A translation 
of the origin of time to $s$ proves the result for $\Gamma(t,s)$.
\end{proof}

It follows from this lemma that our propagators are universal for any local in 
time problem. 
This is particularly noticeable for Aronson's work with real coefficients as he 
con\-s\-tructs the kernel of $\Gamma(t,s)$ by using approximations by the 
propagators obtained by the standard parametrix constructions for equations with 
smooth coefficients on bounded cylinders. Our approach is totally opposite 
as we construct the ``largest'' possible object and restrict it. It will be useful to 
have shown uniqueness in the largest possible energy class $\dot W(0,\infty)$ 
later on.   

\section{A priori estimates}
\label{sec4}

We first prove a priori estimates for arbitrary weak solutions.
We then turn to solutions of the form $(t,x)\mapsto \Gamma(t,0)f(x)$ for $f$ 
in an $L^p$ space.

\subsection{Reverse H\"older estimates and consequences}
\label{subsec:reverseHolder}

We consider the parabolic quasi-distance on $(0,\infty)\times{\mathbb{R}}^n$
defined by 
$$
d((t,x),(s,y))=\max\bigl\{\sqrt{|t-s|},|x-y|\bigr\},\quad 
(t,x),(s,y)\in (0,\infty)\times{\mathbb{R}}^n
$$
and denote by ${\mathcal B}((t,x),R)=[t-R^2,t+R^2]\times B(x,R)$ the corresponding
ball of radius $R$. Remark that $(0,\infty)\times{\mathbb{R}}^n$ with this
parabolic quasi-distance and the Lebesgue measure is a doubling quasi-metric 
measure space. The following lemma is a particular case of well-known 
$L^p(L^q)$ estimates for weak solutions. See \cite{Ar67}.  

\begin{lemma} 
\label{lem:5.1}
Let $q:=2+\frac{4}{n}$. There is a constant $C>0$ depending on dimension and 
the ellipticity constants in \eqref{ell},  such that for all $u$ global weak solution 
of \eqref{eq1}, for all $(t,x)\in (0,\infty)\times{\mathbb{R}}^n$, 
and all $r \in (0, \frac{\sqrt{t}}{4})$, we have 
\begin{equation}
\label{reverseHolder}
\Bigl(\fint_{{\mathcal{B}}((t,x), r)}|u(s,y)|^q
\,{\rm d}y\,{\rm d}s\Bigr)^{\frac{1}{q}}
\le C
\Bigl(\fint_{{\mathcal{B}}((t,x),4 r)}|u(s,y)|^2
\,{\rm d}y\,{\rm d}s\Bigr)^{\frac{1}{2}}.
\end{equation}
\end{lemma}

\begin{proof}
Let $(t,x)\in(0,\infty)\times{\mathbb{R}}^n$, and $r \in (0,\frac{\sqrt{t}}{4})$. Pick
$\varphi\in{\mathscr{C}}_c^\infty({\mathbb{R}}^n)$ supported in 
$B(x,2r)$ such that $0\le \varphi\le 1$, $\varphi=1$ on 
$B(x,r)$ and $\|\nabla\varphi\|_\infty\lesssim\frac{1}{r}$.
Let $\sigma\in[t-r^2,t+r^2]$. By Gagliardo-Nirenberg's inequality 
(see \cite[(2.2)]{Ni59}), we have that
\begin{align*}
&\int_{B(x,r)}|u(\sigma,y)|^q\,{\rm d}y
\le \int|u(\sigma,y)\varphi(y)|^q\,{\rm d}y
\lesssim \bigl\|\nabla\bigl(u(\sigma,\cdot)\varphi\bigr)\bigr\|_2^2\,
\bigl\|u(\sigma,\cdot)\varphi\bigr\|_2^{\frac{4}{n}}
\\[4pt]
&= \bigl\|\nabla\bigl(u(\sigma,\cdot)\varphi\bigr)\bigr\|_2^2\,
\Bigl(\int_{B(x,2r)}|u(\sigma,y)|^2\,{\rm d}y\Bigr)^{\frac{2}{n}}
\\[4pt]
&\lesssim\bigl\|\nabla\bigl(u(\sigma,\cdot)\varphi\bigr)\bigr\|_2^2\,
\Bigl(\frac{1}{r^2}\int_{t-r^2}^{t+r^2}\int_{B(x,4r)}
|u(s,y)|^2\,{\rm d}y\,{\rm d}s\Bigr)^{\frac{2}{n}},
\end{align*}
where we have used Proposition~\ref{prop:nrjloc} in the last step. Now let
$$
A:=\int_{{\mathcal{B}}((t,x),4r)} |u(s,y)|^2\,{\rm d}y\,{\rm d}s.
$$
We thus have that
\begin{align*}
&\int_{{\mathcal{B}}((t,x),r)}|u(\sigma,y)|^q\,{\rm d}y\,{\rm d}\sigma 
\\[4pt]
\lesssim&\Bigl(\frac{A}{r^2}\Bigr)^{\frac{2}{n}}
\Bigl[\int_{t-r^2}^{t+r^2}\int_{B(x,2r)}|\nabla u(\sigma,y)|^2
\,{\rm d}y\,{\rm d}\sigma
+\int_{t-r^2}^{t+r^2}\int_{B(x,2r)}
|u(\sigma,y)|^2 \frac{1}{r^2}
\,{\rm d}y\,{\rm d}\sigma\Big]
\\[4pt]
\lesssim&\Bigl(\frac{A}{r^2}\Bigr)^{\frac{2}{n}}
\Bigl[\frac{1}{r^2}\int_{t-16r^2}^{t+16r^2}\int_{B(x,4r)}
|u(\sigma,y)|^2\,{\rm d}y\,{\rm d}\sigma \Bigr]
= \Bigl(\frac{A}{r^2}\Bigr)^{\frac{q}{2}},
\end{align*}
where we have used Proposition~\ref{prop:nrjloc} again and $q=2+\frac{4}{n}$.
This proves \eqref{reverseHolder}.  
\end{proof}

Observe that the proof applies to any ball ${\mathcal{B}}((t,x),r)$ provided 
$t-16r^2>0$.  Hence we may apply Gerhing's lemma in the context of a space 
of homogeneous type. See for example a proof in \cite{bb}. As the constant 
$C$ is independent of $u$ and the radius of the ball, we obtain an improvement 
of $q$ to some $\tilde q>q$ that depends only on dimension and the ellipticity 
constants. Also, the exponent 2 can be lowered. See \cite[Theorem~2]{IN85} for 
the original euclidean proof, and \cite[Theorem~B1]{BCF15} for a proof valid in 
spaces of homogeneous type. 

\begin{corollary}
\label{cor:4.4}
There exist $C>0$ and $\tilde{q}>2+\frac{4}{n}$, depending on dimension and 
the ellipticity constants in \eqref{ell},  such that for all $u$ global weak solution 
of \eqref{eq1}, for all $(t,x)\in (0,\infty)\times{\mathbb{R}}^n$, and all 
$r \in (0, \frac{\sqrt{t}}{4})$, we have, for all $p \in [1,2]$,
\begin{equation}
\label{self-improve-1}
\Bigl(\fint_{{\mathcal{B}}((t,x),r)}|u(s,y)|^2
\,{\rm d}y\,{\rm d}s\Bigr)^{\frac{1}{2}}
\lesssim 
\Bigl(\fint_{{\mathcal{B}}((t,x),r)}|u(s,y)|^{\tilde{q}}
\,{\rm d}y\,{\rm d}s\Bigr)^{\frac{1}{\tilde{q}}}
\lesssim 
\Bigl(\fint_{{\mathcal{B}}((t,x),4r)}|u(s,y)|^{p}\,{\rm d}y\,{\rm d}s\Bigr)^{\frac{1}{p}}.
\end{equation}
\end{corollary}

These reverse H\"older inequalities are useful, among other things, to control 
the potential growth of $L^2_{\rm loc}$ norms for solutions in $L^\infty(L^p)$. 

\begin{proposition}
\label{prop:L2poids}
Let $p\in[1,\infty]$. Let $u$ be a global weak solution of \eqref{eq1}. Assume
that $u\in L^\infty(L^p)$. Then for all $b>a>0$, and all 
$w \in L^{p'}({\mathbb{R}}^n)$,
$$
\int_{{\mathbb{R}}^n}\Bigl(\int_a^b\int_{B(x,\sqrt{b})} |u(t,y)|^2\,{\rm d}y\,{\rm d}t
\Bigr)^{\frac{1}{2}}w(x)\,{\rm d}x<\infty.
$$
\end{proposition}

\begin{proof}
We first remark that the case $p=\infty$ is trivial, and assume from now on 
that $p<\infty$. By H\"older inequality, since $w\in L^{p'}({\mathbb{R}}^n)$, we have
\begin{align*}
I:=&
\int_{{\mathbb{R}}^n}\Bigl(\int_a^b\int_{B(x,\sqrt{b})} |u(t,y)|^2\,{\rm d}y\,{\rm d}t
\Bigr)^{\frac{1}{2}}w(x)\,{\rm d}x
\\
&\lesssim \Bigl(\int_{{\mathbb{R}}^n}
\Bigl(\int_a^b\int_{B(x,\sqrt{b})} |u(t,y)|^2\,{\rm d}y\,{\rm d}t\Bigr)^{\frac{p}{2}}
\,{\rm d}x\Bigr)^{\frac{1}{p}}.
\end{align*}
If $p\le 2$, then by Corollary~\ref{cor:4.4}, and a covering argument, we have 
the following for all $a' \in (0,a)$ and $b' \in (b,\infty)$:
$$
I \lesssim
\Bigl(\int_{a'}^{b'}\|u(t,\cdot)\|_p^p\,{\rm d}t\Bigr)^{\frac{1}{p}}
\lesssim \|u\|_{L^\infty(L^p)}.
$$
If $p>2$, then by H\"older inequality
$$
\Bigl(\int_a^b\int_{B(x,\sqrt{b})} |u(t,y)|^2\,{\rm d}y\,{\rm d}t\Bigr)^{\frac{1}{2}}
\lesssim
\Bigl(\int_a^b\int_{B(x,\sqrt{b})} |u(t,y)|^p\,{\rm d}y\,{\rm d}t\Bigr)^{\frac{1}{p}},
$$
We conclude as in the case where $p\le 2$ to obtain
$I\lesssim \|u\|_{L^\infty(L^p)}$.
\end{proof}

\begin{proposition}
\label{prop:tildeNinLp}
Let $p\in[1,\infty]$ and $u$ be a global weak solution of \eqref{eq1} such that
$\tilde N(u)\in L^p({\mathbb{R}}^n)$. Then for all $b>a> 0$, and all 
$w \in L^{p'}({\mathbb{R}}^n)$,
$$
\int_{{\mathbb{R}}^n}\Bigl(\int_a^b\int_{B(x,\sqrt{b})} |u(t,y)|^2\,{\rm d}y\,{\rm d}t
\Bigr)^{\frac{1}{2}}w(x)\,{\rm d}x<\infty.
$$
\end{proposition}

\begin{proof}
Given $b>a>0$, there exists $M \in {\mathbb{N}}$ such that $b \leq 2^{M+1}a < 2b$,
and there exists $N \in {\mathbb{N}}$ and 
$\{z_k\;;\; k=1,\dots,N\} \subset B(0,\sqrt{b})$ such that 
$B(0,\sqrt{b}) \subset \bigcup\limits_{k=1}^N B(z_k,\sqrt{a})$.
Therefore, for all $x\in{\mathbb{R}}^n$,
\begin{align*}
\Bigl(\int_a^b\int_{B(x,\sqrt{b})} |u(t,y)|^2\,{\rm d}y\,{\rm d}t
\Bigr)^{\frac{1}{2}}
&\leq \sum_{j=0}^M \sum_{k=1}^N 
\Bigl(\int_{2^ja}^{2^{j+1}a} \int_{B(x+z_k,\sqrt{a})} |u(t,y)|^2\,{\rm d}y\,{\rm d}t
\Bigr)^{\frac{1}{2}}\\
&\lesssim \sum_{k=1}^N \tilde{N}(u)(x+z_k).
\end{align*}
Since $w\in L^{p'}({\mathbb{R}}^n)$, this gives 
$$
\int_{{\mathbb{R}}^n}\Bigl(\int_a^b\int_{B(x,\sqrt{b})} |u(t,y)|^2\,{\rm d}y\,{\rm d}t
\Bigr)^{\frac{1}{2}}w(x)\,{\rm d}x
\lesssim
\bigl\|\sum_{k=1}^N \tilde{N}(u)(\cdot+z_k)\bigr\|_p\|w\|_{p'}
\lesssim \|\tilde N(u)\|_p,
$$
using the invariance by translation of the $L^p$ norm.
\end{proof}

\begin{remark}
\label{rk:esssup1}
For $p \in [1,\infty]$, note that, if $u$ is a global weak solution of \eqref{eq1} 
such that $\displaystyle{\esssup_{t>0} \|u(t,.)\|_{L^p}=M}$, then 
$\displaystyle{\sup_{t>0} \|u(t,.)\|_{L^p}=M}$. This follows from the continuity 
of $t\mapsto u(t,\cdot)$ in $L^2_{\rm loc}({\mathbb{R}}^n)$ and easy density 
arguments.
\end{remark}

\subsection{Estimates for the propagators}
\label{subsec:slice-spaces}

\begin{lemma}
\label{lem:slice-tildeN}
Let $p \in (2,\infty]$.
\begin{enumerate}[{\rm (1)}]
\item
For all $g \in L^2({\mathbb{R}}^n)$ supported in a ball $B(0,M)$, and all 
$t \in [0,\infty)$, 
$$
\|\Gamma(t,0)^*g\|_{L^{p'}} \lesssim_{M,t} \|g\|_{L^2}.
$$ 
Consequently, for all $h\in L^p({\mathbb{R}}^n)$, $\Gamma(t,0)h$ can be defined
in $L^2_{\rm loc}({\mathbb{R}}^n)$.
\item
For all $h \in L^p({\mathbb{R}}^n)$, 
$\|(t,x) \mapsto (\Gamma(t,0)h)(x)\|_{X^p} \sim \|h\|_{L^p}$.
\end{enumerate}
\end{lemma}

\begin{proof}
(1) 
Let $t\ge 0$ and $g \in L^2({\mathbb{R}}^n)$ supported in a ball  $B(0,M)$. 
Using Proposition~\ref{prop:OD}, for some  $c>0$, we have
\begin{align*}
\|\Gamma(t,0)^*g\|_{L^{p'}}
&\leq \sum_{j=1}^{\infty} \|1\!{\rm l}_{S_j(0,M)}\Gamma(t,0)^*g\|_{L^{p'}}
\lesssim_M \sum_{j=1}^{\infty} 2^{jn(\frac{1}{p}-\frac{1}{2})} 
\|1\!{\rm l}_{S_j(0,M)}\Gamma(t,0)^*g\|_{L^2}
\\
&\lesssim_M \|g\|_{L^2}+
\sum_{j=2}^{\infty} 2^{jn(\frac{1}{p}-\frac{1}{2})}e^{-c\frac{4^{j}}{t}}
\|g\|_{L^2}\lesssim_{M,t} \|g\|_{L^2}.
\end{align*}
(2)
Let $\delta>0$, $x\in{\mathbb{R}}^n$. Let  $h\in L^p({\mathbb{R}}^n)$, and
$j\ge 1$. Using Proposition~\ref{prop:OD} again, we have that
\begin{align*}
\Bigl(\fint_{\frac{\delta}{2}} ^{\delta} \fint_{B(x,\sqrt{\delta})}
|\Gamma(t,0)(1\!{\rm l}_{S_j(x,\sqrt{\delta})}h)(y)|^{2}
\,{\rm d}y\,{\rm d}t\Bigr)^{\frac{1}{2}}
& \lesssim  
2^{\frac{jn}{2}} e^{-c4^j}\Bigl(\fint_{\frac{\delta}{2}}^\delta
\fint_{B(x,2^{j+1}\sqrt{\delta})}
|h(y)|^2\,{\rm d}y\,{\rm d}t\Bigr)^{\frac{1}{2}}
\\
&\lesssim 2^{\frac{jn}{2}} e^{-c4^j}\bigl(M_{HL}|h|^2\bigr)^{\frac{1}{2}}(x),
\end{align*}
where $M_{HL}$ denote the uncentered Hardy-Littlewood maximal function. 
Therefore, 
$$
\bigl\|\tilde N\bigl((t,x)\mapsto\Gamma(t,0)h(x)\bigr)\bigr\|_{L^p}
\lesssim \sum_{j\ge 1}2^{\frac{jn}{2}} e^{-c4^j}
\bigl\|\bigl(M_{HL}|h|^2\bigr)^{\frac{1}{2}}\bigr\|_{L^p}
\lesssim\|h\|_{L^p}.
$$
We next  prove the reverse inequality. Fix $z\in {\mathbb{R}}^n$. We first 
remark that the same reasoning as above gives us
\begin{align*}
\Bigl(\fint_{\frac{\delta}{2}}^\delta \fint_{B(x,\sqrt{\delta})}
|\Gamma(t,0)(1\!{\rm l}_{S_j(z,1)}h)(y)|^2
\,{\rm d}y\,{\rm d}t\Bigr)^{\frac{1}{2}}
\lesssim  
2^{\frac{jn}{2}} e^{-c4^j}\bigl(M_{HL}|h|^2\bigr)^{\frac{1}{2}}(x),
\end{align*}
for all $j\ge 1$ and $\delta\in (0,1)$ and $x\in B(z,1)$.
Moreover, by continuity of $t\mapsto\Gamma(t,0)(1_{S_{j}(z,1)}h)$ in 
$L^2({\mathbb{R}}^n)$  and Lebesgue's differentiation theorem, 
we have that for all $j\ge 1$, 
$$
\Bigl(\fint_{\frac{\delta}{2}}^\delta\fint_{B(x,\sqrt{\delta})}
|\Gamma(t,0)(1_{S_{j}(z,1)}h)(y)|^2\,{\rm d}y\,{\rm d}t
\Bigr)^{\frac{1}{2}}\xrightarrow[\delta\to0]{}|(1\!{\rm l}_{S_j(z,1)}h)(x)|,
$$
for almost every $x\in B(z,1)$. As the right hand side is zero for $j\ge 2$ and 
$x\in B(z,1)$, we deduce by summing that 
$$
\Bigl(\fint_{\frac{\delta}{2}}^\delta\fint_{B(x,\sqrt{\delta})}
|\Gamma(t,0)(1_{B(z,2)^c}h)(y)|^2\,{\rm d}y\,{\rm d}t
\Bigr)^{\frac{1}{2}}\xrightarrow[\delta\to0]{}0
$$
almost everywhere for $x\in B(z,1)$ and by difference, 
$$
\Bigl(\fint_{\frac{\delta}{2}}^\delta\fint_{B(x,\sqrt{\delta})}
|\Gamma(t,0)h(y)|^2\,{\rm d}y\,{\rm d}t
\Bigr)^{\frac{1}{2}}\xrightarrow[\delta\to0]{}|h(x)|,
$$ 
almost everywhere on $B(z,1)$.  Hence this holds on ${\mathbb{R}}^n$ as $z$ 
is arbitrary. Since $(t,x)\mapsto (\Gamma(t,0)h)(x) \in X^p$, we are done if 
$p=\infty$, and, if $p<\infty$, Fatou's lemma gives us
\begin{equation*}
\|h\|_{L^p}\lesssim\|(t,x)\mapsto\Gamma(t,0)h(x)\|_{X^p}.
\qedhere
\end{equation*}
\end{proof}

\begin{lemma}
\label{lem:cont-in-slice}
\begin{enumerate}[{\rm (1)}]
\item
Let $p\in[1,\infty)$ and $\delta>0$. We have that 
$\displaystyle{\sup_{t\in [0,\delta]}\|\Gamma(t,0)\|_{\mathcal{L}(E_\delta^p)}
<\infty}$.
\item
Let $p\in[1,\infty)$. For all $\delta>0$ and all $f\in E_\delta^p$, one has 
$\displaystyle{\lim_{t\to0}\Gamma(t,0)f=f}$ in $E_\delta^p$.

\noindent
For all $t>0$, one also has
$\displaystyle{\lim_{s\to0}\Gamma(t,s)^{*}f=\Gamma(t,0)^{*}f}$ in $E_\delta^p$.
\item
For all $h \in L^{\infty}({\mathbb{R}}^n)$, we have that
$\displaystyle{\lim_{t\to0}\Gamma(t,0)h=h}$ in $L^{2}_{\rm loc}({\mathbb{R}}^n)$.
\end{enumerate}
\end{lemma}

\begin{proof}
(1)  
\cite[Proposition 4.2]{AM14} applies using Proposition~\ref{prop:OD}.

\noindent
(2)  
\cite[Proposition 4.4]{AM14} applies using
Proposition~\ref{prop:OD} and the continuity results proven in 
Proposition~\ref{prop:adjointGamma} and Proposition~\ref{lem:propag}.

\noindent
(3) 
Let $h \in L^\infty({\mathbb{R}}^n)$, and $M>0$. For $t>0$, as in 
Lemma~\ref{lem:slice-tildeN} using Proposition~\ref{prop:OD}, we see that 
${\sum_{j\ge 1} 1\!{\rm l}_{B(0,M)} (\Gamma(t,0)-I)(1\!{\rm l}_{S_{j}(0,M)}h)}$ 
converges in $L^2({\mathbb{R}}^n)$ to $1\!{\rm l}_{B(0,M)} (\Gamma(t,0)-I)h$, 
and moreover,
\begin{align*}
\Bigl(\int _{B(0,M)}  |(\Gamma(t,0)-I)h(y)|^2\,{\rm d}y\Bigr)^{\frac{1}{2}}
& \lesssim_M
\|(\Gamma(t,0)-I)(1_{S_{1}(0,M)}h)\|_{L^2}
+\sum_{j=2} ^\infty e^{-c\frac{4^{j}}{t}}\|h\|_{L^2},
\end{align*}
for some constant $c>0$ depending on $M$. We conclude using that  
$\|(\Gamma(t,0)-I)g\|_{L^{2}} \xrightarrow[t \to 0]{} 0$ for all 
$g \in L^2({\mathbb{R}}^n)$. 
\end{proof}

\begin{proposition}
\label{prop:propsol}
Let $p \in (2,\infty]$. For all $h \in L^p({\mathbb{R}}^n)$,
$u_h:(t,x) \mapsto (\Gamma(t,0)h)(x) \in X^p$ is a global weak solution 
of \eqref{eq1}.
\end{proposition}

\begin{proof}
Let $h \in L^p({\mathbb{R}}^n)$. We first show that 
$\nabla u_h \in L^2_{\rm loc}({\mathbb{R}}^{n+1}_+)$. Let $a,b,M>0$. For 
$j\ge 1$, set $h_j= 1\!{\rm l}_{S_{j}(0,M)}h\in L^2({\mathbb{R}}^n)$ and 
consider the global weak solutions $u_{h_j}$ of \eqref{eq1} with data $h_j$.
Applying Proposition~\ref{prop:nrjloc} and Proposition~\ref{prop:OD}, we obtain 
the following for some constant $\beta>0$:
\begin{align}
\label{eq:nablaL2loc}
\Bigl(\int_a^b\int_{B(0,M)}&|\nabla (\Gamma(t,0)h_j)(x)|^2
\,{\rm d}x\,{\rm d}t\Bigr)^{\frac{1}{2}} 
\nonumber\\
&
\underset{a,b,c,M} \lesssim
\Bigl(\int_{a/2}^b\int_{B(0,2M)}|(\Gamma(t,0)h_j)(x)|^2
\,{\rm d}x\,{\rm d}t\Bigr)^{\frac{1}{2}} 
\nonumber\\
& \lesssim
e^{-\beta4^j}\|1\!{\rm l}_{B(0,2^{j+1}M)}h\|_{L^2} 
\lesssim 2^{jn(\frac{1}{2}-\frac{1}{p})}e^{-\beta4^j}\|h\|_{L^p}.
\end{align}
We easily obtain from this  that  $\sum_{j\ge 1} \nabla u_{h_j}$ converges to
$\nabla u_h \in L^2(a,b;B(0,M))$. Also  $\sum_{j\ge 1}  u_{h_j}$ 
converges to $u_h$ in $L^2(a/2,b;B(0,2M))$ 

To show that $u_h$ satisfies \eqref{eq1} in the sense of distributions, let 
$\phi \in {\mathscr{D}}$  and pick $a,b,M>0$ such that
${\rm supp}\,\phi \subset [a,b]\times B(0,M)$. 
For each $j\ge 1$, 
$$
-\int_0^\infty\int_{{\mathbb{R}}^n}u_{h_j}\overline{\partial_t\phi} \,{\rm d}x\,{\rm d}t
+\int_0^\infty\int_{{\mathbb{R}}^n}A\nabla u_{h_j}\cdot\overline{\nabla\phi}\, 
\,{\rm d}x\,{\rm d}t = 0
$$
and by the above $L^2$ convergences, one can sum in $j\ge 1$ and obtain the 
conclusion for $u_h$. This shows it is a global weak solution of \eqref{eq1}.
\end{proof}

\begin{lemma}
\label{lem:ODp-2}
Let $q\in[1,2)$ and assume that 
$\displaystyle{\sup_{0\leq s\le t<\infty}\|\Gamma(t,s)\|_{{\mathscr{L}}(L^q)}<\infty}$.
Then, for all $r\in(q,2]$ there exists $\alpha>0$ such that for all
$E,F\subset{\mathbb{R}}^n$ Borel sets, for all $ 0\le s<t<\infty$ and all
$f\in L^r({\mathbb{R}}^n)$,
$$
\bigl\|1\!{\rm l}_E\Gamma(t,s)(1\!{\rm l}_Ff)\bigr\|_{L^2}\lesssim
(t-s)^{-\frac{n}{2}(\frac{1}{r}-\frac{1}{2})} e^{-\alpha\frac{d(E,F)^2}{t-s}}
\bigl\|1\!{\rm l}_Ff\bigr\|_{L^r}.
$$
\end{lemma}

\begin{proof}
Let $f\in L^2({\mathbb{R}}^n)\cap L^q({\mathbb{R}}^n)$, $x\in{\mathbb{R}}^n$,
$t>s\ge 0$. By Proposition~\ref{prop:nrjloc} we have that
$$
\Bigl(\fint_{B(x,\sqrt{t-s})}|\Gamma(t,s)f(y)|^2\,{\rm d}y\Bigr)^{\frac{1}{2}}
\lesssim \Bigl(\fint_s^t\fint_{B(x,2\sqrt{t-s})}|\Gamma(\sigma,s)f(y)|^2
\,{\rm d}y\,{\rm d}\sigma\Bigr)^{\frac{1}{2}}.
$$
Covering $B(x,2\sqrt{t-s})$ by a finite collection of balls
$\bigl\{B\bigl(x_j,\frac{\sqrt{t-s}}{2}\bigr);j=1,\dots,M\bigr\}$ with $M$
depending only on $n$ and 
$B\bigl(x_j,\frac{\sqrt{t-s}}{2}\bigr)\subset B(x,4\sqrt{t-s})$ for all 
$j=1,\dots,M$, we can apply Corollary~\ref{cor:4.4} to obtain
\begin{align*}
\Bigl(\fint_{B(x,\sqrt{t-s})}|\Gamma(t,s)f(y)|^2\,{\rm d}y\Bigr)^{\frac{1}{2}}
\lesssim&
\Bigl(\fint_s^t\fint_{B(x,4\sqrt{t-s})}|\Gamma(\sigma,s)f(y)|^q
\,{\rm d}y\,{\rm d}\sigma\Bigr)^{\frac{1}{q}}
\\
\lesssim&(t-s)^{-\frac{n}{2q}}\Bigl(\fint_s^t\|f\|_q^q\,{\rm d}\sigma\Bigr)^{\frac{1}{q}}
=(t-s)^{-\frac{n}{2q}}\|f\|_q.
\end{align*}
Therefore, for all $x\in{\mathbb{R}}^n$, $t>s\ge 0$ and $h\in L^q({\mathbb{R}}^n)$
\begin{equation}
\label{eq:lql2bdd}
\bigl\|1\!{\rm l}_{B(x,\sqrt{t-s})}\Gamma(t,s)h\bigr\|_{L^2}
\lesssim (t-s)^{-\frac{n}{2}(\frac{1}{q}-\frac{1}{2})}\|h\|_{L^q}.
\end{equation}
Let $\delta=\sqrt{t-s}$.
Consider the family of disjoint cubes 
${\mathscr{D}}_\delta:=\bigl\{\delta\,[0,1[^n+k\delta,k\in{\mathbb{Z}}^n\bigr\}$. 
We denote by $c_Q$ the center of a cube $Q\in{\mathscr{D}}_\delta$. We have that
$$
(t-s)^{\frac{n}{2}(\frac{1}{q}-\frac{1}{2})}1\!{\rm l}_Q\Gamma(t,s)1\!{\rm l}_R
\in {\mathscr{L}}(L^q({\mathbb{R}}^n),L^2({\mathbb{R}}^n))
$$
with norm independent of $t>s \ge 0$ and $Q,R\in{\mathscr{D}}_\delta$. 
Using Proposition~\ref{prop:OD} with  Riesz-Thorin interpolation, we have that, 
for all $r \in (q,2]$, there exists 
$\alpha_r>0$ such that for all $h\in L^r({\mathbb{R}}^n)$
$$
\bigl\|1\!{\rm l}_Q\Gamma(t,s)(1\!{\rm l}_Rh)\bigr\|_{L^2}
\lesssim (t-s)^{-\frac{n}{2}(\frac{1}{r}-\frac{1}{2})}e^{-\alpha_r\frac{d(Q,R)^2}{t-s}}
\|1\!{\rm l}_Rh\|_{L^r},
$$
for all $Q,R\in{\mathscr{D}}_\delta$, for all $t>s \ge 0$.
Therefore, there exists $c_r'>0$ such that
\begin{align*}
\|\Gamma(t,s)h\|_{L^2}
=&
\Bigl(\sum_{Q\in{\mathscr{D}}_\delta}\|1\!{\rm l}_Q\Gamma(t,s)h\|_{L^2}^2
\Bigr)^{\frac{1}{2}}
\\
\lesssim&
\sum_{k\in{\mathbb{Z}}^n}\Bigl(\sum_{Q\in{\mathscr{D}}_\delta}
\|1\!{\rm l}_{Q+k\delta}\Gamma(t,s)(1\!{\rm l}_Q h)\|_{L^2}^2\Bigr)^{\frac{1}{2}}
\\
\lesssim&
\sum_{k\in{\mathbb{Z}}^n}e^{-c_r'|k|}
(t-s)^{-\frac{n}{2}(\frac{1}{r}-\frac{1}{2})}\Bigl(\sum_{Q\in{\mathscr{D}}_\delta}
\|1\!{\rm l}_Q h\|_{L^r}^2\Bigr)^{\frac{1}{2}}
\\
\lesssim&
(t-s)^{-\frac{n}{2}(\frac{1}{r}-\frac{1}{2})}\Bigl(\sum_{Q\in{\mathscr{D}}_\delta}
\|1\!{\rm l}_Q h\|_{L^r}^r\Bigr)^{\frac{1}{r}}
=(t-s)^{-\frac{n}{2}(\frac{1}{r}-\frac{1}{2})}\|h\|_{L^r},
\end{align*}
where we have used that $\ell_r\subset\ell_2$ since $r\le 2$. Therefore
$$
\|\Gamma(t,s)\|_{{\mathscr{L}}(L^r,L^2)}
\lesssim (t-s)^{-\frac{n}{2}(\frac{1}{r}-\frac{1}{2})}
$$ 
uniformly for $0 \le s<t<\infty$. Using Riesz-Thorin interpolation again to interpolate
this uniform bound with the $L^2-L^2$ off diagonal bound from
Proposition~\ref{prop:OD} gives the result.
\end{proof}

\begin{lemma}
\label{lem:propsol2}
Let $q \in [1,2)$, and assume that $\displaystyle{\sup_{0\le s \le t <\infty}
\|\Gamma(t,s)\|_{\mathcal{L}(L^{q})}<\infty}$. Then, for all 
$h \in L^q({\mathbb{R}}^n)$, $u_h:(t,x)\mapsto (\Gamma(t,0)h)(x)$ is a 
global weak solution of \eqref{eq1}.
\end{lemma}

\begin{proof}
By \eqref{eq:lql2bdd}, we have that, for all $t>0$, all $h\in L^q({\mathbb{R}}^n)$, 
and all $M>0$:
$$
\|1\!{\rm l}_{B(0,M)}\Gamma(t,0)h\|_{L^2} \lesssim_{M,t} \|h\|_{L^q}.
$$
Applying Proposition~\ref{prop:nrjloc}, we obtain the following for all 
$c \in (a,b)$ and $M>0$:
\begin{equation*}
\Bigl(\int_c^b\int_{B(0,M)}|\nabla\Gamma(t,0)h(x)|^2
\,{\rm d}x\,{\rm d}t\Bigr)^{\frac{1}{2}}
\underset{a,b,c,M}\lesssim
\Bigl(\int_a^b
\|1\!{\rm l}_{B(0,2M)}\Gamma(t,0)h\|_{L^2} ^{2} dt \Big)^{\frac{1}{2}} 
\lesssim \|h\|_{L^q}.
\end{equation*}
To show that $u_h$ satisfies \eqref{eq1} in the sense of distributions, 
let $\varepsilon >0$, and pick
$h_0\in L^q({\mathbb{R}}^n)\cap L^2({\mathbb{R}}^n)$ such that
$\|h-h_0\|_{L^q}<\varepsilon$. The function $u_{h_0}$
is a global weak solution of \eqref{eq1}, and we thus have the following.
\begin{align*}
\Bigl|&-\int_0^\infty\int_{{\mathbb{R}}^n}u_h\overline{\partial_t\phi}
+\int_0^\infty\int_{{\mathbb{R}}^n}A\nabla u_h\cdot\overline{\nabla\phi}\Bigr|
\\
&\le \int_0^\infty\int_{{\mathbb{R}}^n}|u_{h-h_0}||\partial_t\phi|
+\Lambda\int_0^\infty\int_{{\mathbb{R}}^n}|\nabla u_{h-h_0}||\nabla\phi|
\\& \lesssim
\|u_{h-h_0}\|_{L^\infty(L^q)}
+\|\nabla (u_{h-h_0})\|_{L^2({\rm supp}(\nabla \phi))}\lesssim\|h-h_0\|_{L^q}
<\varepsilon.
\end{align*}
This proves that $u_h$ is a global weak solution of \eqref{eq1}.
\end{proof}

\begin{lemma}
\label{lem:OD2-q}
Let $q\in (2,\infty]$ and assume that 
$\displaystyle{\sup_{0\leq s\le t<\infty}\|\Gamma(t,s)\|_{{\mathscr{L}}(L^q)}<\infty}$.
Then, for all $r\in[2,q)$ there exists $\alpha>0$ such that for all
$E,F\subset{\mathbb{R}}^n$ Borel sets, for all $ 0\le s < t <\infty$ and all
$f\in L^2({\mathbb{R}}^n)$,
$$
\bigl\|1\!{\rm l}_E\Gamma(t,s)(1\!{\rm l}_Ff)\bigr\|_{L^r}\lesssim
(t-s)^{-\frac{n}{2}(\frac{1}{2}-\frac{1}{r})} e^{-\alpha\frac{d(E,F)^2}{t-s}}
\bigl\|1\!{\rm l}_Ff\bigr\|_{L^2}.
$$
\end{lemma}

\begin{proof}
Using Proposition~\ref{prop:OD} we only have to show that
$$
\sup_{0\le s< t<\infty}\|(t-s)^{-\frac{n}{2}(\frac{1}{q'}-\frac{1}{2})}
\Gamma(t,s)^*\|_{{\mathscr{L}}(L^{q'},L^2)}<\infty.
$$
For $0\le s\le t<\infty$, we have by Proposition~\ref{prop:adjointGamma} and the
proof of Lemma~\ref{lem:4.2}: 
$$
\Gamma(t,s)^*=\tilde{\Gamma}(s,t)={\underset{\widetilde{ }}{\Gamma}}(t-s,0),
$$
where ${\underset{\widetilde{ }}{\Gamma}}$ is the propagator for equation
\eqref{eq1} with $A$ replaced by
$$
\tilde A(s,x)=
\begin{cases}
A^*(t-s,x)\mbox{ if }(s,x)\in[0,t]\times{\mathbb{R}}^n,
\\
A^*(0,x)\mbox{ otherwise.}
\end{cases}
$$
Since $\tilde A$ satisfies \eqref{ell} with the same constants as $A$, 
Proposition~\ref{prop:OD} applies to $\tilde\Gamma$ and the result 
follows from interpolation between $q'$ and $2$.
\end{proof}

\begin{remark}
\label{rk:esssup2}
For $p \in [1,\infty]$, 
$$
{\esssup_{ 0\le s \le t <\infty} 
\|\Gamma(t,s)\|_{\mathcal{L}(L^{p})}=M}\quad  \Longrightarrow \quad
{\sup_{0\le s \le t <\infty}\|\Gamma(t,s)\|_{\mathcal{L}(L^p)}=M}.
$$
This follows, using Proposition~\ref{lem:propag} and 
Proposition~\ref{prop:adjointGamma}, from the continuity of 
$t\mapsto \langle\Gamma(t,s)f,g \rangle$ on $[s,\infty)$ and of 
$s\mapsto \langle\Gamma(t,s)f,g \rangle$ on $[0,t]$ for all $f,g \in {\mathscr{D}}$ 
(or $f \in L^\infty$, $g \in L^1$ compactly supported), a simple measure 
theoretical argument and density arguments.
\end{remark}

\subsection{Propagators with kernel  bounds}
\label{sec:kernel}

We say that the propagators $\Gamma(t,s)$, $ 0\le s <t<\infty$, have kernel 
bounds if their kernels  $k(t,s,.,.)$ are measurable functions with
\begin{align}
\label{p0}
|k(t,s,x,y)| \leq C(t-s)^{-\frac{n}{2}}e^{-c\frac{|x-y|^{2}}{4(t-s)}},
\end{align}
for some $C,c>0$, all $ 0\le s <t<\infty$, and almost all $x,y \in {\mathbb{R}}^n$.

In this case, $\Gamma(t,s)$ is an integral operator and one has the integral
representation 
$$
\Gamma(t,s)f(x) = \int_{{\mathbb{R}}^n} k(t,s,x,y)f(y)\,{\rm d}y
$$ 
for all $f \in L^2({\mathbb{R}}^n)$ and almost every $x \in {\mathbb{R}}^n$. 
Moreover, as the integral makes sense for $f\in L^p({\mathbb{R}}^n)$, 
$1\le p\le \infty$, one can extend $\Gamma(t,s)$ to a bounded operator on 
$L^p({\mathbb{R}}^n)$, uniformly in $t\ge s$ (recall that $\Gamma(s,s)=I$).   

As mentioned, Aronson's proved kernel bounds for propagators of real equations. 
At this point, it is worth pointing out that the following result, proven by Hofmann 
and Kim in \cite[Theorems 1.1 and 1.2]{HK04}, extends to our situation.   

\begin{proposition} 
\label{prop:HK}
\begin{enumerate}[{\rm (1)}]
\item
The propagators $\Gamma(t,s)$, $0\le s\le t <\infty$, have kernel bounds if 
weak solutions in $L^2_{\rm loc}(\overline{{\mathbb{R}}^{n+1}_+})$
of \eqref{eq1} and of \eqref{eq2} on ${\mathbb{R}}^{n+1}_+$ satisfy scale 
invariant local $L^2-L^\infty$ bounds of Moser type on parabolic cylinders. 
The constants $C,c$ in \eqref{p0} depend on the 
ellipticity constants in \eqref{ell} and the bounds in the local estimates. 
\item
Conversely, if the propagators $\Gamma(t,s)$, $0\le s\le t <\infty$, have kernel 
bounds then global weak solutions satisfy the scale invariant local 
$L^2-L^\infty$ bounds of Moser type on Whitney parabolic cylinders. 
\end{enumerate} 
\end{proposition}

The proof in \cite{HK04} is done for smooth coefficients. In this case, one can 
use the classical fundamental solution. However, once we have our notion of  
propagators, we can run the argument {\it mutatis mutandi}. In particular, 
supremum is replaced by essential supremum (or even supremum in time and 
essentiel supremum in $x$) in the local bounds. Also the argument for the first part 
is done for the propagators on the full range $-\infty<s<t<\infty$, but inspection 
reveals that, to get the estimate for $k(s,t,x,y)$, only local bounds on parabolic 
cylinders contained in the strip $[s,t]\times {\mathbb{R}}^n$ are used. This 
explains our hypothesis on the weak solutions in part (1). 

The converse is stated in \cite[Theorems 1.2]{HK04} for the full range. The 
argument there does not preserve strips $[s,t]\times {\mathbb{R}}^n$ 
(a modification of the argument could probably do it) but, if we restrict to 
parabolic cylinders of Whitney type (as in the definition of the maximal function 
$\tilde N$), then the argument gives the desired local bounds.  
  
\section{Existence and uniqueness results}
\label{sec5}

\subsection{Main result}
\label{sec:main}

Here, we prove interior representation from a weak control on solutions.

\begin{theorem}
\label{thm:THE_THEOREM}
Let $u$ be a local weak solution of \eqref{eq1} on $(a,b)\times {\mathbb{R}}^n$.
Assume
$$
M:=\int_{{\mathbb{R}}^{n}} \Bigl(\int_a^b \int_{B(x,\sqrt{b})} |u(t,y)|^{2}
\,{\rm d}y\,{\rm d}t\Bigr)^{\frac{1}{2}}e^{-\gamma |x|^2}\,{\rm d}x<\infty
$$
for some $\gamma<\frac{\alpha}{4(b-a)}$ where $\alpha$ is the constant 
in Proposition~\ref{prop:OD} ($\alpha=\frac{\lambda}{4\Lambda^2}$). 
Then $u(t,\cdot) = \Gamma(t,s)u(s,\cdot)$ 
for every $a<s\leq t<b$, in the following sense:
$$
\int_{{\mathbb{R}}^n} u(s,x)\,\overline{\Gamma(t,s)^*h(x)}\,{\rm d}x 
= \int_{{\mathbb{R}}^n} u(t,x)\, \overline{h(x)}\,{\rm d}x 
\quad \forall h \in {\mathscr{C}}_c({\mathbb{R}}^n).
$$
\end{theorem}

\begin{proof}
{\tt Step~0:}
For $h\in{\mathscr{C}}_c({\mathbb{R}}^n)$, its support being
included in $B(0,\rho)$ for some $\rho>0$, we have for all $j\ge 1$
and for all $x\in{\mathbb{R}}^n$ that
\begin{equation}
\label{eq:step0}
e^{-\beta4^j}\|1\!{\rm l}_{S_j(x,\sqrt{b})}h\|_{L^2}\lesssim
\bigl(1\!{\rm l}_{|x|\le\rho}+e^{-\beta\frac{(|x|-\rho)^2}{4b}}1\!{\rm l}_{|x|>\rho}\bigr)
\|h\|_{L^2}
\end{equation}
(recall that $S_j(x,\sqrt{b})$ denotes the annulus 
$B(x,2^{j+1}\sqrt{b})\setminus B(x,2^j\sqrt{b})$ if $j\ge 2$ and the ball 
$B(x,2\sqrt b)$ if $j=1$) and therefore, for $\gamma<\frac{\beta}{4b}$,
\begin{equation}
\label{eq:step0bis}
\sup_{x\in{\mathbb{R}}^n}\bigl(e^{\gamma|x|^2}
e^{-\beta4^j}\|1\!{\rm l}_{S_j(x,\sqrt{b})}h\|_{L^2}\bigr)\lesssim
\|h\|_{L^2}, \quad \forall\,j\in{\mathbb{N}}.
\end{equation}
If $|x|\le \rho$, then \eqref{eq:step0} is immediate. Let $|x|>\rho$:
$\|1\!{\rm l}_{S_j(x,\sqrt{b})}h\|_2\neq0$ only if 
$B(0,\rho)\cap S_j(x,\sqrt{b})\neq\emptyset$. 
Pick $y\in B(0,\rho)\cap S_j(x,\sqrt{b})$ and we have that
$$
|x|\le |x-y|+|y|\le 2^{j+1}\sqrt{b}+\rho, \quad\mbox{and then}\quad
\frac{1}{4b}(|x|-\rho)^2\le 4^j.
$$
This implies \eqref{eq:step0}. Now, for $\gamma<\frac{\beta}{4b}$, we have that 
$$
\sup_{x\in{\mathbb{R}^n}}\bigl(e^{\gamma|x|^2} e^{-\beta\frac{(|x|-\rho)^2}{4b}}\bigr)
=e^{\frac{\beta\gamma\rho^2}{\beta-4b\gamma}}<\infty\quad\mbox{and}\quad
\sup_{x\in{\mathbb{R}^n}}\bigl(1\!{\rm l}_{|x|\le\rho}e^{\gamma|x|^2} \bigr)
=e^{\gamma\rho^2}<\infty,
$$
which proves \eqref{eq:step0bis}.

\medskip

\noindent
{\tt Step~1:}
{\sl We show that for all $a<s\le t<b$ and all 
$h \in {\mathscr{C}}_c({\mathbb{R}}^n)$,
we have that $u(s,\cdot)\overline{\Gamma(t,s)^*h} \in L^1({\mathbb{R}}^n)$.}

\medskip 
 
\noindent 
Let $\rho>0$ be such that ${\rm supp}\,h \subset B(0,\rho)$. Using 
Proposition~\ref{prop:nrjloc} and Proposition~\ref{prop:OD}, we have the 
following (with constants depending on $t,s,b,\rho$):
\begin{align*}
\int_{{\mathbb{R}}^n} & |u(s,x)||\Gamma(t,s)^*h(x)|\,{\rm d}x  =
\int_{{\mathbb{R}}^n} \Bigl(\fint_{B(x,\frac{\sqrt{b}}{2})} |u(s,y)||\Gamma(t,s)^*h(y)|
\,{\rm d}y\Bigr)\,{\rm d}x 
\\ 
&\le \sum_{j=1}^\infty \int_{{\mathbb{R}}^n} 
\Bigl(\fint_{B(x,\frac{\sqrt{b}}{2})} |u(s,y)|^2\,{\rm d}y\Bigr)^{\frac{1}{2}}
\|1\!{\rm l}_{B(x,\frac{\sqrt{b}}{2})}\Gamma(t,s)^*(1\!{\rm l}_{S_j(x,\sqrt{b})}h)\|_{L^2}
\,{\rm d}x 
\\
&\lesssim  \sum_{j=1}^\infty \int_{{\mathbb{R}}^n} 
\Bigl(\int_a^b \int_{B(x,\sqrt{b})} |u(\sigma,y)|^2
\,{\rm d}y\,{\rm d}\sigma\Bigr)^{\frac{1}{2}}
e^{-\alpha\frac{4^j b}{b-a}}\|1\!{\rm l}_{S_j(x,\sqrt{b})}h\|_{L^2}\,{\rm d}x
\\
&\lesssim  M \|h\|_{L^2}
<\infty.
\end{align*}
where we have used, for any $\varepsilon>0$,  
$\alpha\frac{4^j b}{b-a}=\varepsilon 4^j+\beta 4^j$, the fact that 
$\displaystyle{\sum_{j=1}^\infty e^{-\varepsilon 4^j}<\infty}$ and
\eqref{eq:step0bis} with $\beta=\frac{\alpha b}{(b-a)}-\varepsilon$ and 
$\gamma<\frac{\beta}{4b}$ in the last line.

\medskip

\noindent
{\tt Step 2:} 
{\sl Some identities.}

\medskip

\noindent
We fix $h \in {\mathscr{C}}_c({\mathbb{R}}^n)$ and let $a<t<b$.
Define $\phi(s,x):=\Gamma(t,s)^*h(x)$ for all $s \in [0,t]$ and $x\in {\mathbb{R}}^n$. 
By construction, the function $\phi$ is a weak solution of the backward 
equation \eqref{eq2} with  $\nabla\phi\in L^2(0,t; L^2({\mathbb{R}}^n))$ 
and one has $\phi\in{\mathscr{C}}([0,t];L^2({\mathbb{R}}^n))$. 
Let $\chi \in {\mathscr{C}}_c^\infty({\mathbb{R}}^n;{\mathbb{R}})$, 
and let $\eta \in {\mathscr{C}}_c^\infty((a,t);{\mathbb{R}})$. Denote by 
$\Omega$ a bounded open set containing the support of $\chi$. Since $u$ is a 
weak solution of \eqref{eq1}, we have that
$$
{}_{L^2(a,b;H^{-1}(\Omega))}\langle \partial_su,
\phi \chi^2\eta\rangle_{L^2(a,b;H^1_0(\Omega))}
=-\int_a^b \int_\Omega A(s,x)\nabla u(s,x)\cdot
\nabla \bigl(\overline{\phi(s,x)}\chi^{2}(x)\eta(s)\bigr) \,{\rm d}x\,{\rm d}s.
$$
Since $\phi$ is a weak solution of \eqref{eq2}, $\nabla u\in L^2_{\rm loc}$ and 
$u\in{\mathscr{C}}([a,b],L^2_{\rm loc})$ by Proposition \ref{prop:nrjloc},  
we have that
$$
{}_{L^2(a,b;H^1_0(\Omega))}\langle u\chi^2\eta,
\partial_s\phi\rangle_{L^2(a,b;H^{-1}(\Omega))}
= \int_a^b \int_\Omega \nabla\bigl(u(s,x)\chi^2(x)\eta(s)\bigr)\cdot
\overline{A(s,x)^{*}\nabla \phi(s,x)}\,{\rm d}x\,{\rm d}s.
$$
Noting that 
\begin{align*}
{}_{L^2(H^{-1})}\langle \partial_su,
\phi \chi^2\eta\rangle_{L^2(H^1_0)}&+
{}_{L^2(H^1_0)}\langle u\chi^2\eta,\partial_s\phi\rangle_{L^2(H^{-1})}
+\int_a^b \int_\Omega \chi^2(x)u(s,x)\overline{\phi(s,x)}\eta'(s)\,{\rm d}x\,{\rm d}s
\\
=&\int_a^b\partial_{s}\Bigl(\int_\Omega \bigl(u(s,x)
\overline{\phi(s,x)}\chi^2(x)\eta(s)\bigr)\,{\rm d}x\Bigr)\,{\rm d}s = 0,
\end{align*}
we get, adding the three equations above,
\begin{align*}
&\int_a^b \int_\Omega \chi^2(x)u(s,x)\overline{\phi(s,x)}\eta'(s)\,{\rm d}x\,{\rm d}s 
\\
=&\int_a^b \eta(s) \int_\Omega \Bigl(A(s,x)\nabla u(s,x)\cdot
\nabla\bigl(\overline{\phi(s,x)}\chi^2(x)\bigr) -\nabla\bigl(u(s,x)\chi^2(x)\bigr)\cdot
\overline{A(s,x)^{*}\nabla\phi(s,x)}\Bigr)\,{\rm d}x\,{\rm d}s.
\end{align*}
Calculating, some terms cancel and we obtain
\begin{align}
\label{eq:step2}
&\int_a^b \int_\Omega \chi^2(x)u(s,x)\overline{\phi(s,x)}\eta'(s)\,{\rm d}x\,{\rm d}s 
\\
=&\int_{a}^{b} \eta(s) 
\int_{\Omega} \Bigl(\bigl(A(s,x)\nabla u(s,x)\cdot\nabla \chi^{2}(x)\bigr) 
\overline{\phi(s,x)}
-u(s,x)\bigl(\nabla\chi^{2}(x)\cdot \overline{A(s,x)^{*}\nabla \phi(s,x)}\bigr)\Bigr)
\,{\rm d}x\,{\rm d}s.
\nonumber
\end{align}

\noindent
{\tt Step 3:} 
{\sl We now prove that 
\begin{equation}
\label{eq:chigone}
\int_a^b \int_{{\mathbb{R}}^n} u(s,x)\overline{\phi(s,x)}\eta'(s)\,{\rm d}x\,{\rm d}s = 0.
\end{equation}
}

\noindent
We choose $\chi$ of the form 
$x\mapsto \psi\bigl(\frac{|x|}{R}\bigr)$ for $R>0$ and 
$\psi \in {\mathscr{C}}_c^\infty([0,\infty))$ supported on 
$[0,2]$, and equal to $1$ on $[0,1]$. Note that $\chi(y)=1$ for all 
$y \in B(0,R)$ and $\|\nabla \chi\|_{\infty} \lesssim R^{-1}$. 
We have already shown that $u(s,\cdot)\overline{\phi(s,\cdot)} \in 
L^1({\mathbb{R}}^n)$ for every $s\in (a, t]$. 
Thus the left hand side of \eqref{eq:step2} goes to the left hand side of
\eqref{eq:chigone} as $R$ goes to $\infty$ by dominated convergence.
To prove \eqref{eq:chigone}, it remains to show that 
$\overline{\phi}\nabla u \in L^1((c,d)\times{\mathbb{R}}^n)$ and that 
$u\overline{\nabla \phi} \in L^1((c,d)\times{\mathbb{R}}^n)$  with 
$a<c<d<t$ such that
${\rm supp}\,\eta\subset[c,d]$, so that dominated convergence applies as well as $R$  goes to $\infty$.   
Using Proposition~\ref{prop:nrjloc} and Proposition~\ref{prop:OD}, this is done 
as a simple modification of the argument used in Step~0 and Step~1.
\begin{align*}
&\int_c^{d} \int_{{\mathbb{R}}^n}  |\nabla u(s,x)||\Gamma(t,s)^*h(x)|\,{\rm d}x\,{\rm d}s  
=\int_c^{d} \int_{{\mathbb{R}}^n} \Bigl(\fint_{B(x,\frac{\sqrt{b}}{2})} 
|\nabla u(s,y)||\Gamma(t,s)^*h(y)|\,{\rm d}y\Bigr)\,{\rm d}x\,{\rm d}s 
\\ 
&\ \lesssim  \sum_{j= 1 }^\infty \int_{{\mathbb{R}}^n} \Bigl(\int_c^b \fint_{B(x,\frac{\sqrt{b}}{2})}
|\nabla u(s,y)|^2\,{\rm d}y\,{\rm d}s\Bigr)^{\frac{1}{2}}
 \Bigl(\int_c^d \|1\!{\rm l}_{B(x,\frac{\sqrt{b}}{2})} 
\Gamma(t,s)^*(1\!{\rm l}_{S_j(x,\sqrt{b})}h)\|_{L^2}^2\,{\rm d}s \Bigr)^{\frac{1}{2}} \,{\rm d}x 
\\
& \ \lesssim  \sum_{j= 1  }^\infty  \int_{{\mathbb{R}}^n} \Bigl(\int_a^b 
\int_{B(x,\sqrt{b})} |u(s,y)|^2\,{\rm d}y\,{\rm d}s\Bigr)^{\frac{1}{2}}
e^{-\frac{\alpha b}{b-a} 4^j}\|1\!{\rm l}_{S_j(x,\sqrt{b})}h\|_{L^2}\,{\rm d}x .
\end{align*}
We conclude by \eqref{eq:step0bis} with $\beta<\frac{\alpha b}{(b-a)}$ and  this gives $\overline{\phi}\nabla u \in L^1((c,d)\times{\mathbb{R}}^n)$. 
Using Proposition~\ref{prop:nrjloc2}, instead of Proposition~\ref{prop:nrjloc},
and Proposition~\ref{prop:OD}, and $d<t$:
\begin{align*}
& \int_{ c }^d \int_{{\mathbb{R}}^n}  |u(s,x)||\nabla \Gamma(t,s)^*h(x)|\,{\rm d}x\,{\rm d}s  
= \int_{ c }^d\int_{{\mathbb{R}}^n} \Bigl(\fint_{B(x,\frac{\sqrt{b}}{2})} |u(s,y)|
|\nabla \Gamma(t,s)^*h(y)|\,{\rm d}y\Bigr)\,{\rm d}x\,{\rm d}s 
\\ 
&\ \lesssim \sum_{j=1}^\infty \int_{{\mathbb{R}}^n} \Bigl(\int_a^b 
\fint_{B(x,\frac{\sqrt{b}}{2})} |u(s,y)|^2\,{\rm d}y\,{\rm d}s\Bigr)^{\frac{1}{2}}
\Bigl( \int_{ c }^d \|1\!{\rm l}_{B(x,\frac{\sqrt{b}}{2})}
\nabla \Gamma(t,s)^*(1\!{\rm l}_{S_j(x,\sqrt{b})}h)\|_{L^2}^2\,{\rm d}s
\Bigr)^{\frac{1}{2}}\,{\rm d}x 
\\
&\ \lesssim  \sum_{j=1}^\infty \int_{{\mathbb{R}}^n} 
\Bigl(\int_a^b \int_{B(x,\sqrt{b})} |u(s,y)|^2\,{\rm d}y\,{\rm d}s\Bigr)^{\frac{1}{2}}
\Bigl(\int_a^t \|1\!{\rm l}_{B(x,\sqrt{b})}
\Gamma(t,s)^*(1_{S_j(x,\sqrt{b})}h)\|_{L^2}^2\,{\rm d}s\Bigr)^{\frac{1}{2}}\,{\rm d}x 
\\
&\ \lesssim  \sum_{j=1}^\infty \int_{{\mathbb{R}}^n} 
\Bigl(\int_a^b \int_{B(x,\sqrt{b})} |u(s,y)|^2
\,{\rm d}y\,{\rm d}s\Bigr)^{\frac{1}{2}}e^{-\alpha \frac{4^jb}{b-a}}
\|1\!{\rm l}_{S_j(x,\sqrt{b})}h\|_{L^2}\,{\rm d}x.
\end{align*}
By \eqref{eq:step0bis} with $\beta<\frac{\alpha b}{(b-a)}$, this gives 
$u\overline{\nabla \phi} \in L^1((c,d)\times{\mathbb{R}}^n)$. 
We have thus established \eqref{eq:chigone}.

\medskip

\noindent
{\tt Step 4:}
{\sl Choosing a specific $\eta$.}

\medskip

\noindent
We now pick $v \in {\mathscr{C}}_c^\infty\bigl(-\frac{1}{2},\frac{1}{2}\bigr)$ 
such that $\int_{-\frac{1}{2}}^{\frac{1}{2}} v(y)dy=1$. For 
$\varepsilon\in\bigl(0,\frac{t-s}{3}\wedge \frac{s-a}{3} \bigr)$, we consider 
$\eta\in{\mathscr{C}}_c^\infty(a,t)$ such that
$$
\eta'(\sigma) = \frac{1}{2\varepsilon}\,
v\Bigl(\frac{\sigma-(s-\varepsilon)}{2\varepsilon}\Bigr)
-\frac{1}{2\varepsilon}\,
v\Bigl(\frac{\sigma-(t-2\varepsilon)}{2\varepsilon}\Bigr) \quad \forall\,\sigma \in (a,t).
$$
Remark that the support of $\eta'$ is contained in 
$[a+\varepsilon, t-\varepsilon]$ and as it has mean value 0, the same 
hold for $\eta$. 
From \eqref{eq:chigone} (with $s$ becoming $\sigma$ in the integral), 
we thus get that
\begin{align*}
&\frac{1}{2\varepsilon} \int_{s-2\varepsilon}^s v
\Bigl(\frac{\sigma-(s-\varepsilon)}{2\varepsilon}\Bigr) 
\Bigl(\int_{{\mathbb{R}}^n} u(\sigma,x)\overline{\phi}(\sigma,x)\,{\rm d}x\Bigr)
\,{\rm d}\sigma
\\
=&\frac{1}{2\varepsilon} \int_{t-3\varepsilon}^{t-\varepsilon} v
\Bigl(\frac{\sigma-(t-2\varepsilon)}{2\varepsilon}\Bigr) 
\Bigl(\int_{{\mathbb{R}}^n} u(\sigma,x)\overline{\phi}(\sigma,x)
\,{\rm d}x\Bigr)\,{\rm d}\sigma,
\end{align*}
and thus, changing variables:
\begin{align*}
&\int_{-\frac{1}{2}}^{\frac{1}{2}} v(\sigma)\Bigl(\int_{{\mathbb{R}}^n} 
u(s-\varepsilon(1-2\sigma),x)\overline{\phi}(s-\varepsilon(1-2\sigma),x)
\,{\rm d}x\Bigr)\,{\rm d}\sigma
\\
=& \int_{-\frac{1}{2}}^{\frac{1}{2}} v(\sigma)\Bigl(\int_{{\mathbb{R}}^n} 
u(t-2\varepsilon(1-\sigma),x)\overline{\phi}(t-2\varepsilon(1-\sigma),x)
\,{\rm d}x\Bigr)\,{\rm d}\sigma.
\end{align*}
Recall that $\phi(t,x)=h(x)$ and $\phi(s,x)=\Gamma(t,s)^*h(x)$.
The result will be proven once we have established that 
\begin{equation}
\label{eq:etagone1}
\lim_{\varepsilon \to 0} 
\int_{-\frac{1}{2}}^{\frac{1}{2}} v(\sigma) \int_{{\mathbb{R}}^n} 
u(t-2\varepsilon(1-\sigma),x) \overline{\phi}(t-2\varepsilon(1-\sigma),x)
\,{\rm d}x\,{\rm d}\sigma
=\int_{{\mathbb{R}}^n} u(t,x)\overline{\phi}(t,x)\,{\rm d}x
\end{equation}
and that
\begin{equation}
\label{eq:etagone2}
\lim_{\varepsilon \to 0}  
\int_{-\frac{1}{2}}^{\frac{1}{2}} v(\sigma) \int_{{\mathbb{R}}^n} 
u(s-\varepsilon(1-2\sigma),x)
\overline{\phi}(s-\varepsilon(1-2\sigma),x)\,{\rm d}x\,{\rm d}\sigma 
= \int_{{\mathbb{R}}^n} u(s,x)\overline{\phi}(s,x)\,{\rm d}x,
\end{equation}

\noindent
{\tt Step~5:} 
{\sl Proof of \eqref{eq:etagone1}.}

\medskip

\noindent
Set $f(\tau,x)=u(t-\tau,x)\overline{\phi}(t-\tau,x)$ for $\tau\in [0,t-a]$ and 
$g(\tau,x)= \fint_{B(x,\frac{\sqrt{b}}{2})} f(\tau,y)\, {\rm d}y.
$ 
After averaging, we have to show that 
$$
\lim_{\varepsilon \to 0} 
\int_{-\frac{1}{2}}^{\frac{1}{2}}  \int_{{\mathbb{R}}^n} 
v(\sigma) g(2\varepsilon(1-\sigma),x) \, {\rm d}x\, {\rm d}\sigma = 
\int_{-\frac{1}{2}}^{\frac{1}{2}} \int_{{\mathbb{R}}^n} v(\sigma) g(0,x) 
\,{\rm d}x\, {\rm d}\sigma.
$$
It follows from Proposition~\ref{prop:nrjloc} and Proposition~\ref{prop:nrjloc2} 
that for all $x\in {\mathbb{R}}^n$, 
$f \in {\mathscr{C}}([0,t-a];L^1(B(x,\frac{\sqrt{b}}{2}))$. 
Hence $g(2\varepsilon(1-\sigma),x) \to g(0,x)$ when $\varepsilon\to 0$ for all 
$(\sigma,x)$. 

For $\tau=2\varepsilon(1-\sigma)$, we have $\tau\in [0,3\varepsilon]\subset [0, t-a]$.
To apply dominated convergence, we show that  
$\sup_{\tau\in [0,t-a]}|g(\tau,x)|$ is integrable on ${\mathbb{R}}^n$. 
This is a variant of Step 1 to get uniformity. Indeed, 
for all $x\in{\mathbb{R}}^n$ and $\tau\in [0, t-a]$, by Proposition~\ref{prop:nrjloc} 
and Proposition~\ref{prop:OD} with $\tau\le b-a$
\begin{align*}
|g(\tau,x)|\le &\fint_{B(x,\frac{\sqrt{b}}{2})}
|u(t-\tau,y)||\Gamma(t,t-\tau)^*h(y)|
\,{\rm d}y
\\
\lesssim &\Bigl(\fint_{B(x,\frac{\sqrt{b}}{2})} |u(t-\tau,y)|^2
\,{\rm d}y\Bigr)^{\frac{1}{2}}
\Bigl(\sum_{j=1}^\infty 
\|1\!{\rm l}_{B(x,\frac{\sqrt{b}}{2})}\Gamma(t,t-\tau)^{*}
(1\!{\rm l}_{S_j(x,\sqrt{b})}h)\|_{L^2}\Bigr)  
\\
\lesssim & 
\Bigl(\int_a^b \int_{B(x,\sqrt{b})} |u(\sigma',y)|^2
\,{\rm d}y\,{\rm d}\sigma'\Bigr)^{\frac{1}{2}}
\Bigl(\sum_{j=1}^\infty e^{-\alpha\frac{4^jb}{b-a}}
\|1\!{\rm l}_{S_j(x,\sqrt{b})}h\|_{L^2}\Bigr).
\end{align*}
This estimate is uniform with respect to $\tau$ and we get integrability as 
in Step~1 using Step~0.

\medskip

\noindent
{\tt Step~6:}
{\sl Proof of \eqref{eq:etagone2}.}

The proof is exactly the same as that of \eqref{eq:etagone1} taking now  
$f(\tau,x)=u(s-\tau,x)\overline{\phi}(s-\tau,x)$ for $\tau\in [0,s-a]$ and using 
$\tau=\varepsilon(1-2\sigma)$. 
\end{proof}

\subsection{Results for $p \ge 2$}
\label{subsec:ex!p>2}

Our uniqueness  results will be based on the following well-know fact. Let $X$ 
be a Banach space and $Y$ its dual space.  If $(y_k)_{k\in {\mathbb{N}}}$
is a sequence weakly$^*$ converging to $y$ in $Y$, and 
$(x_k)_{k \in {\mathbb{N}}}$ is a sequence strongly converging to $x$ in $X$, 
then $(\langle y_k, x_k\rangle)_{k\in {\mathbb{N}}}$ 
converges to $\langle y, x \rangle$. Of course, when $X$ is reflexive, 
weak$^*$ and weak convergence coincide. 
 
We illustrate this principle by first proving that $L^\infty(L^2)$ is always a class 
of uniqueness for $L^2$ data. Next  we look at $L^p$ data for $p>2$ using 
non-tangential maximal estimates.  

\begin{theorem}
\label{thm:linfl2}
For $u \in {\mathscr{D}}'$, the following assertions are equivalent.
\begin{align}
&\exists !\,f\in L^2({\mathbb{R}}^n)\mbox{ such that, for all } t>0, \; 
u(t,\cdot)=\Gamma(t,0)f\mbox{ in } L^2({\mathbb{R}}^n);
\label{eq521}\\
&u \mbox{ is a global weak solution of \eqref{eq1} in }L^\infty(L^2).
\label{eq522}
\end{align}
\end{theorem}

\begin{proof}
Proposition~\ref{lem:propag} gives us that \eqref{eq521} implies \eqref{eq522}.
We now assume \eqref{eq522}, and note that
$\displaystyle{\sup_{t>0} \|u(t,\cdot)\|_{L^2}<\infty}$ by Remark~\ref{rk:esssup1}.
Let $t>0$, and pick $(t_k)_{k\in{\mathbb{N}}}$ a decreasing sequence
of real numbers converging to $0$, with $t_0=\frac{t}{2}$, such that there exists 
$f \in L^2({\mathbb{R}}^n)$ with
$$
u(t_k,\cdot)\rightharpoonup f\ \ \mbox{ as }k\to\infty,\quad
\mbox{weakly in }L^2({\mathbb{R}}^n).
$$
By Proposition~\ref{prop:L2poids}, we can apply
Theorem~\ref{thm:THE_THEOREM}, and get that,
for $k\in{\mathbb{N}}$ and $h\in{\mathscr{C}}_c({\mathbb{R}}^n)$,  
$$
\int_{{\mathbb{R}}^n}u(t,x)\overline{h(x)}\,{\rm d}x=
\int_{{\mathbb{R}}^n}u(t_k,x)\overline{\Gamma(t,t_k)^*h(x)}\,{\rm d}x.
$$
By the continuity results in Proposition~\ref{prop:adjointGamma}, we have 
that $\|\Gamma(t,t_k)^*h-\Gamma(t,0)^*h\|_{2} \xrightarrow[k\to\infty]{}0$, 
and thus
$$
u(t,\cdot) = \Gamma(t,0)f \quad \mbox{in }L^2({\mathbb{R}}^n). 
$$
This also implies that $\displaystyle{f =\lim_{t\to 0}u(t,\cdot)}$ 
strongly in $L^2({\mathbb{R}}^n)$, and proves the uniqueness of $f$.
\end{proof}

The following corollary is an immediate consequence of 
Theorem~\ref{thm:linfl2} and Theorem~\ref{thm:wellposed}.
Recall from local estimates and Lions' result that a global weak solution of 
\eqref{eq1} in $L^\infty(L^2)$ is a priori in ${\mathscr{C}}([0,+\infty);L^2_{\rm loc})$.

\begin{corollary}
For all $u_0\in L^2({\mathbb{R}}^n)$, the problem
$$
\partial_tu={\rm div}\,A\nabla u,
\quad
u\in L^{\infty}(L^2),
\quad 
u(0,.)=u_0,
$$
is well posed. Its solution $u$ agrees with the energy solution, and, therefore, 
is such that
$$
\|u_0\|_{L^2} = \|u\|_{L^\infty(L^2)}
\le \sqrt{2\Lambda} \|\nabla u\|_{L^2(L^2)}\le 
\textstyle{\sqrt{\frac{\Lambda}{\lambda}}}\,\|u_0\|_{L^2}.
$$
\end{corollary}

We now consider $p>2$.

\begin{theorem}
\label{thm:ex!p>2}
Let $p\in(2,\infty]$. For $u \in {\mathscr{D}}'$,
the following assertions are equivalent.
\begin{align}
&\exists !\,f\in L^p({\mathbb{R}}^n)\mbox{ such that, for all } t> 0, \; 
u(t, \cdot)=\Gamma(t,0)f \mbox{ in } L^2_{\rm loc}({\mathbb{R}}^n);
\label{ex!}\\
&u \mbox{ is a global weak solution of \eqref{eq1} with }
\tilde N(u)\in L^p({\mathbb{R}}^n).
\label{tildeNinLp}
\end{align}
\end{theorem}

\begin{proof}
Lemma~\ref{lem:slice-tildeN} and Proposition~\ref{prop:propsol} give us that
\eqref{ex!}$\implies$\eqref{tildeNinLp}. We now consider the other direction
and assume that $\tilde N(u)\in L^p({\mathbb{R}}^n)$.
Since $p>2$, Lemma~\ref{lem:4.7} and 
Proposition~\ref{prop:nrjloc} yield, for all $0<t\le\delta$:
\begin{align*}
\|u(t,\cdot)\|_{E_\delta^p}
&\lesssim
\|u(t,\cdot)\|_{E_{\frac{t}{4}}^p}=\Bigl(\int_{{\mathbb{R}}^n}
\Bigl(\fint_{B(x,\frac{\sqrt{t}}{2})}
|u(t,y)|^2\,\,{\rm d}y\Bigr)^{\frac{p}{2}}\,{\rm d}x\Bigr)^{\frac{1}{p}}
\\
&\lesssim
\Bigl(\int_{{\mathbb{R}}^n}\Bigl(\fint_{\frac{t}{2}}^{t}\fint_{B(x,\sqrt{t})}
|u(s,y)|^2\,\,{\rm d}y\,{\rm d}s\Bigr)^{\frac{p}{2}}\,{\rm d}x\Bigr)^{\frac{1}{p}}
\le\|\tilde N(u)\|_p= \|u\|_{X^p},
\end{align*}
the constants being independent of $t,\delta$ (with the usual modification if 
$p=\infty$). Fix $\delta>0$, and 
let $f_\delta\in E_\delta^p$, and $(t_k)_{k\in{\mathbb{N}}}$ be a 
decreasing sequence such that $t_k\xrightarrow[k\to\infty]{}0$, $t_0<\delta$ and
$$
u(t_k,\cdot)\xrightarrow[]{weak^*}f_\delta\quad\mbox{in }E_\delta^p.
$$
For each $j\ge 1$, as $ E_{\frac{\delta}{j}}^p= E_\delta^p$ with equivalent 
norm, the weak$^*$ convergence holds in $ E_{\frac{\delta}{j}}^p$  and
$\displaystyle{\|f_\delta\|_{E_{\frac{\delta}{j}}^p}\le 
\liminf_{k \to \infty}\|u(t_k,.)\|_{E^p_{\frac{\delta}{j}}}\lesssim \|u\|_{X^p}}$,
the constant being independent of $j\ge 1$ and $\delta >0$. Therefore
$$
\Bigl[x\mapsto \Bigl(\fint_{B(x,\frac{\delta}{j})}|f_\delta(y)|^2
\,{\rm d}y\Bigr)^{\frac{1}{2}}
\Bigr]\in L^p({\mathbb{R}}^n) \quad\forall\,j\ge 1.
$$
Moreover
$$
\Bigl(\fint_{B(x,\frac{\delta}{j})}|f_\delta(y)|^2\,{\rm d}y\Bigr)^{\frac{1}{2}}
\xrightarrow[j\to\infty]{}|f_\delta(x)|\quad\mbox{for }a.e.\ x\in{\mathbb{R}}^n
$$
by Lebesgue differentiation theorem. By Fatou's lemma, 
$f_\delta\in L^p({\mathbb{R}}^n)$ and 
$$
\|f_\delta\|_p\le \liminf_{j \to \infty}\|f_{\delta}\|_{E^p_{\frac{\delta}{j}}}
\lesssim \|u\|_{X^p}.
$$
By Proposition~\ref{prop:tildeNinLp}, we can apply 
Theorem~\ref{thm:THE_THEOREM} to obtain,
for all $k \in {\mathbb{N}}$ and $h\in{\mathscr{C}}_c({\mathbb{R}}^n)$, 
$$
\int_{{\mathbb{R}}^n} u(t_k,x)\, \overline{\Gamma(t,t_k)^*h(x)}\,{\rm d}x 
=  \int_{{\mathbb{R}}^n} u(t,x)\,\overline{h(x)}\,{\rm d}x.
$$
Applying Lemma~\ref{lem:cont-in-slice}, we have that
$$
\Gamma(t,t_k)^*h\xrightarrow[k\to\infty]{}\Gamma(t,0)^*h
\quad\mbox{in }E_\delta^{p'}.
$$
Therefore $\int_{{\mathbb{R}}^n} u(t,x)\,\overline{h(x)}\,{\rm d}x
=\int_{{\mathbb{R}}^n} f_{\delta }(x)\, \overline{\Gamma(t,0)^*h(x)}\,{\rm d}x $
for all $t>0$ and all $h\in{\mathscr{C}}_c({\mathbb{R}}^n)$, which gives us
that $u(t,\cdot)=\Gamma(t,0)f_\delta$ in $E_\delta^p$. This implies that
$\displaystyle{f_\delta=\lim_{t\to0}u(t,\cdot)}$ strongly in $E_\delta^p$ for 
$p<\infty$ and in $L^2_{\rm loc}$ for $p=\infty$ by
Lemma~\ref{lem:cont-in-slice}. Therefore, 
$\displaystyle{f_\delta=\lim_{t\to0}u(t,\cdot)}$ strongly in $L^2_{\rm loc}$ 
in all cases and $f_\delta$ is independent of $\delta$. We write $f=f_\delta$. 
This $f$ is unique as $\displaystyle{\lim_{t\to0}u(t,\cdot)}$ in $L^2_{\rm loc}$, 
and $f\in L^p({\mathbb{R}}^n)$ with $\|f\|_p\lesssim\|\tilde N(u)\|_p$ 
as proven above.
\end{proof}

The following corollary is now immediate.

\begin{corollary}
\label{thm:uniqueXp}
Let $p\in (2,\infty]$ and $u_0\in L^p({\mathbb{R}}^n)$. There exists a unique 
global weak solution $u$ of \eqref{eq1} in $X^p$  such that 
$\displaystyle{\lim_{t\to 0}u(t,\cdot)=u_0}$ in $L^{2}_{\rm loc}$. Moreover,
$\|u\|_{X^{p}} \sim \|u_{0}\|_{L^{p}}$.
\end{corollary}

An interesting consequence of our result in $X^{\infty}$ is the following 
conservation property of the propagators.

\begin{corollary}
\label{cor:cons} Let $t>s$. Then 
$$
\Gamma(t,s)1\!{\rm l} = 1\!{\rm l} \mbox{ in } L^2_{\rm loc}({\mathbb{R}}^n) .
$$
Similarly
$$
\Gamma(t,s)^*1\!{\rm l} = 1\!{\rm l} \mbox{ in } L^2_{\rm loc}({\mathbb{R}}^n).
$$

\end{corollary}

\begin{proof} 
We may assume $s=0$ without loss of generality
The constant function $1\!{\rm l}$ on ${\mathbb{R}}^{n+1}_+$ is a global 
weak solution of \eqref{eq1} and belongs to $X^\infty$. By 
Theorem~\ref{thm:ex!p>2}, we have that, for almost every 
$(t,x) \in {\mathbb{R}}^{n+1}_+$, $1\!{\rm l} = \Gamma(t,0)f(x)$ for a unique 
$f \in L^\infty$ such that
$\displaystyle{f = \lim_{t \to 0} \Gamma(t,0)f }$ in 
$L^2_{\rm loc}({\mathbb{R}}^n)$. Thus, $f=1$ almost everywhere on 
${\mathbb{R}}^n$ and we have shown the equality in 
$L^2_{loc}({\mathbb{R}}^{n+1}_+)$. As weak solutions are continuous in time 
with values in $L^2_{\rm loc}({\mathbb{R}}^n)$, the conclusion follows. 
The formula for the adjoint is obtained similarly using that we get the same 
$X^\infty$ result for the backward equation on $(-\infty, t)$. See 
Remark~\ref{rem:back}.  
\end{proof}

We finish with a result valid in full generality, getting closer to $L^p$ estimates. 

\begin{proposition}
\label{prop:rh} 
Let $\tilde q>2+\frac{4}{n}$ be the exponent in the reverse H\"older estimates
of Corollary~\ref{cor:4.4}. Fix $p\in (2,\tilde q)$. For $u\in {\mathscr{D}}'$, 
the following assertions are equivalent.
\begin{align}
&\exists !\,f\in L^p({\mathbb{R}}^n)\mbox{ such that, for all } t> 0, \; 
u(t, \cdot)=\Gamma(t,0)f
\mbox{ in } L^2_{\rm loc}({\mathbb{R}}^n);
\label{Ex}\\
&u \mbox{ is a global weak solution of \eqref{eq1} with }
\sup_{a>0} \Big\| \fint_a^{2a} |u(t,\cdot)|\, {\rm d}t \Big\|_{L^p} <\infty.
\label{meanL^p}
\end{align}
In this case, $\displaystyle{\|f\|_{L^p}\sim \sup_{a>0} 
\Big\| \fint_{a}^{2a} |u(t,\cdot)|\, {\rm d}t \Big\|_{L^p}}$ and 
$\fint_a^{2a} u(t,\cdot)\, {\rm d}t$ converges to $f$ in $L^p({\mathbb{R}}^n)$ as 
$a\to 0$. 
\end{proposition}

\begin{proof} 
For the direct part, let $f\in L^p({\mathbb{R}}^n)$ and $u(t, \cdot)=\Gamma(t,0)f$. 
By Theorem~\ref{thm:ex!p>2}, we know that $u$ is a global weak solution and that
$u\in X^p$. Using the reverse H\"older estimates of Corollary~\ref{cor:4.4}, 
we see that we may replace the $L^2$ averages by $L^p$ averages in the 
definition of $\tilde N(u)$ (up to modifying slightly the parameters). Hence by 
H\"older's inequality and averaging 
$$
\Big\| \fint_a^{2a} |u(t,\cdot)|\,{\rm d}t \Big\|_{L^p}^p 
\le \int_{{\mathbb{R}}^n} \fint_a^{2a} |u(t,x)|^p \,{\rm d}t \,{\rm d}x 
\lesssim \|\tilde N(u)\|_{L^p}^p.
$$
This proves the direct part. In addition, this implies that 
$\fint_a^{2a}\Gamma(t,0)\,{\rm d}t$ are  bounded operators on 
$L^p({\mathbb{R}}^n)$ uniformly with respect to $a$. This is true for all 
$p\in [2, \tilde q)$. At the same time, they converge strongly in $\mathcal{L}(L^2)$ 
to $I$ when $a\to 0$. By an interpolation argument (see the proof of the next 
result, Theorem~\ref{thm:uniqueness_p<2}) this implies the strong continuity 
at 0 in ${\mathscr{L}}(L^p)$. In particular, this yields the norm comparison in 
the statement. 

Let us now prove the converse and assume that $u$ is a global weak solution 
of \eqref{eq1} with $\displaystyle{M=\sup_{a>0} \Big\| \fint_a^{2a} |u(t,\cdot)|\, 
{\rm d}t \Big\|_{L^p} <\infty}$.
For all $\delta >0$ and $t \leq \delta$, using Lemma~\ref{lem:4.7} with $p>2$,  
Proposition~\ref{prop:nrjloc} and the reverse H\"older estimate of 
Corollary~\ref{cor:4.4} again, we have  
\begin{align*}
\|u(t,\cdot)\|_{E_\delta^p}
&\lesssim
\|u(t,\cdot)\|_{E_{\frac{t}{32}}^p}=\Bigl(\int_{{\mathbb{R}}^n}
\Bigl(\fint_{B(x,\sqrt{\frac{t}{32}})}
|u(t,y)|^2\,\,{\rm d}y\Bigr)^{\frac{p}{2}}\,{\rm d}x\Bigr)^{\frac{1}{p}}
\\
&\lesssim 
\Bigl(\int_{{\mathbb{R}}^n}\Bigl(\fint_{\frac{7t}{8}}^{\frac{9t}{8}}
\fint_{B(x,\sqrt{\frac{t}{8}})}
|u(s,y)|^2\,\,{\rm d}y\,{\rm d}s\Bigr)^{\frac{p}{2}}\,{\rm d}x\Bigr)^{\frac{1}{p}}
\\
&
\lesssim \Bigl(\int_{{\mathbb{R}}^n}\Bigl(\fint_{\frac{t}{2}}^t
\fint_{B(x,\sqrt{\frac{t}{2}})}
|u(s,y)|\,\,{\rm d}y\,{\rm d}s\Bigr)^p\,{\rm d}x\Bigr)^{\frac{1}{p}}
\\
& 
= \Bigl(\int_{{\mathbb{R}}^n}\Bigl(\fint_{B(x,\sqrt{\frac{t}{2}})}
\fint_{\frac{t}{2}}^t|u(s,y)|\,\,{\rm d}s\,{\rm d}y\Bigr)^p\,{\rm d}x\Bigr)^{\frac{1}{p}}
\\
&
\le\Bigl(\int_{{\mathbb{R}}^n}\fint_{B(x,\sqrt{\frac{t}{2}})}
\Bigl(\fint_{\frac{t}{2}}^{t}|u(s,y)|\,\,{\rm d}s\Bigr)^p
\,{\rm d}y \,{\rm d}x\Bigr)^{\frac{1}{p}}\ \le \ M.
\end{align*}
Thus we have the uniform estimate in the slice space $E^p_\delta$ as in the 
proof of Theorem~\ref{thm:ex!p>2} and the same argument applies. This proves 
the converse.
\end{proof}

\begin{remark} 
In the previous theorem, $u$ has further regularity: 
$(\fint_a^{2a} |u(t,\cdot)|^p\,{\rm d}t)^{\frac{1}{p}} \in L^p({\mathbb{R}}^n)$ 
uniformly in $a>0$ and $\|f\|_{L^p}\sim 
\sup_{a>0} \Big\| (\fint_a^{2a} |u(t,\cdot)|^p\,{\rm d}t)^{\frac{1}{p}}\Big\|_{L^p}$ 
as one can check. The largest class in this scale for uniqueness is the one 
in the statement. 
\end{remark}

\subsection{Results for $p<2$}
\label{subsec:!p<2}
For $p<2$ we do not know general results without imposing further properties 
of the propagators. Here we assume boundedness of the propagators acting 
on $L^p$, and consider solutions in $L^\infty(L^p)$.
Note that, by Remarks~\ref{rk:esssup1} and~\ref{rk:esssup2}, we can 
assume uniform boundedness rather than almost everywhere boundedness.

\begin{theorem}
\label{thm:uniqueness_p<2}
Let $1\le q< p< 2$. Assume that 
$\displaystyle{\sup_{0\leq s \leq t<\infty}
\|\Gamma(t,s)\|_{{\mathscr{L}}(L^q)}<\infty}$. 
Let $u\in L^\infty((0,\infty);L^p({\mathbb{R}}^n))$ be a  global 
weak solution of \eqref{eq1}. Then there exists $u_0\in L^p({\mathbb{R}}^n)$
such that $u(t,\cdot)=\Gamma(t,0)u_0$ in $L^p({\mathbb{R}}^n)$ for all
$t>0$. Moreover, $u\in{\mathscr{C}}_0([0,\infty);L^p({\mathbb{R}}^n))$ 
and, in particular, $u_0$ is unique.
\end{theorem}

\begin{proof}
Let $u\in L^\infty(L^p)$ be a global 
weak solution of \eqref{eq1}. Let $t>0$, and pick $(t_k)_{k\in{\mathbb{N}}}$ 
a decreasing sequence of real numbers converging to $0$, with 
$t_{0}<\frac{t}{2}$, such that there exists $u_0\in L^p({\mathbb{R}}^n)$ with
$$
u(t_k,\cdot)\rightharpoonup u_0\ \ \mbox{ as }k\to\infty,\quad
\mbox{weakly$^*$ in }L^p({\mathbb{R}}^n).
$$
By Proposition~\ref{prop:L2poids} and Theorem~\ref{thm:THE_THEOREM}, 
for $k\in{\mathbb{N}}$ and  $h\in{\mathscr{C}}_c({\mathbb{R}}^n)$, 
we have that
$$
\int_{{\mathbb{R}}^n}u(t,x)\overline{h(x)}\,{\rm d}x=
\int_{{\mathbb{R}}^n}u(t_k,x)\overline{\Gamma(t,t_k)^*h(x)}\,{\rm d}x.
$$
It remains to prove that $\|\Gamma(t,t_k)^*h-\Gamma(t,0)^*h\|_{p'}
\xrightarrow[k\to\infty]{}0$. By Lemma~\ref{lem:ODp-2}, and the fact that 
$t-t_k \sim t$ for all $k \in {\mathbb{N}}$, we have that
$\displaystyle{\sup_{k\in{\mathbb{N}}}
\|\Gamma(t,t_k)^*\|_{{\mathscr{L}}(L^2,L^{r'})}<\infty}$ for all $r'\in(p',q')$. 
Let $\theta\in[0,1)$ be defined by
$\frac{1}{p'}=\frac{\theta}{r'}+\frac{1-\theta}{2}$. For all 
$h\in{\mathscr{C}}_c^\infty({\mathbb{R}}^n)$,
\begin{align*}
\|\Gamma(t,t_k)^*h-\Gamma(t,0)^*h\|_{L^{p'}}
&\le \|\Gamma(t,t_k)^*h-\Gamma(t,0)^*h\|_{L^{r'}}^\theta
\|\Gamma(t,t_k)^*h-\Gamma(t,0)^*h\|_{L^2}^{1-\theta}
\\[4pt]
&\lesssim
\|\Gamma(t,t_k)^*h-\Gamma(t,0)^*h\|_{L^2}^{1-\theta}\xrightarrow[k\to\infty]{}0.
\end{align*}
We now show that $u\in{\mathscr{C}}([0,\infty);L^p({\mathbb{R}}^n))$.
Let $\varepsilon>0$ and $v_0\in{\mathscr{C}}_c^\infty({\mathbb{R}}^n)$ be
such that $\|u_0-v_0\|_{L^p}<\varepsilon$. Let $s,t>0$:
\begin{align*}
\|\Gamma(t,0)u_0-\Gamma(s,0)u_0\|_{L^p}
&\le
\|\Gamma(t,0)(u_0-v_0)\|_{L^p}+\|\Gamma(t,0)v_0-\Gamma(s,0)v_0\|_{L^p}
+\|\Gamma(s,0)(v_0-u_0)\|_{L^p}
\\[4pt]
&\lesssim \varepsilon+\|\Gamma(t,0)v_0-\Gamma(s,0)v_0\|_{L^q}^\theta
\|\Gamma(t,0)v_0-\Gamma(s,0)v_0\|_{L^2}^{1-\theta}
\end{align*}
for $\theta\in(0,1]$ such that $\frac{1}{p}=\frac{\theta}{q}+\frac{1-\theta}{2}$. 
Therefore
$$
\|\Gamma(t,0)u_0-\Gamma(s,0)u_0\|_{L^p}
\lesssim\varepsilon+\|v_0\|_{L^q}^\theta
\|\Gamma(t,0)v_0-\Gamma(s,0)v_0\|_{L^2}^{1-\theta}.
$$
Since $\bigl(t\mapsto\Gamma(t,0)v_0\bigr)
\in{\mathscr{C}}([0,\infty);L^2({\mathbb{R}}^n))$,
there exists $\delta>0$ such that for all $t,s>0$ with $|t-s|<\delta$,
$\|\Gamma(t,0)v_0-\Gamma(s,0)v_0\|_{L^2}\le 
\Bigl(\frac{\varepsilon}{\|v_0\|_{L^q}^\theta}\Bigr)^{\frac{1}{1-\theta}}$. 
This proves that
$$
\|\Gamma(t,0)u_0-\Gamma(s,0)u_0\|_{L^p}\lesssim \varepsilon\quad
\forall\,t,s>0,\ |t-s|<\delta,
$$
and then the fact that $\bigl(t\mapsto\Gamma(t,0)u_0\bigr)$ is continuous 
in $L^p({\mathbb{R}}^n)$. In particular, 
$$
u_0=\lim_{t\to0}\Gamma(t,0)u_0=\lim_{t\to0}u(t,\cdot).
$$
Since we know moreover that $\bigl(t\mapsto\Gamma(t,0)v_0\bigr)
\in{\mathscr{C}}_0([0,\infty);L^2({\mathbb{R}}^n))$, the same reasoning
shows that $\|\Gamma(t,0)u_0\|_{L^p}\xrightarrow[t\to \infty]{}0$.
\end{proof}

\begin{corollary}
\label{cor:kioloa}
Let $1\le q<p<2$. Assume that 
$\displaystyle{\sup_{0\leq s\le t<\infty}\|\Gamma(t,s)\|_{{\mathscr{L}}(L^q)}<\infty}$.
For $u\in {\mathscr{D}}'$, the following assertions are equivalent.
\begin{align}
&u \mbox{ is a global weak solution of \eqref{eq1} in } 
L^\infty((0,\infty);L^p({\mathbb{R}}^n)) ;
\label{borneLp}\\[4pt]
&\exists !\,u_0\in L^p({\mathbb{R}}^n) \mbox{ such that }u(t,\cdot)=\Gamma(t,0)u_0
\mbox{ in }L^p({\mathbb{R}}^n)\mbox{ for all }t>0 ;
\label{repLp}\\[4pt]
& u \mbox{ is a global weak solution of \eqref{eq1} in }X^p.
\label{tildeN}
\end{align}
In this case, $u\in{\mathscr{C}}_0(L^p)$ 
and $\|u_0\|_p\sim \|u\|_{L^\infty(L^p)}\sim\|u\|_{X^p}$.
\end{corollary}

\begin{proof}
\eqref{borneLp}$\implies$\eqref{repLp} is proven in 
Theorem~\ref{thm:uniqueness_p<2}. The implication 
\eqref{repLp}$\implies$\eqref{borneLp} is a consequence of 
$\displaystyle{\sup_{0\leq s \leq t<\infty}\|\Gamma(t,s)\|_{{\mathscr{L}}(L^p)}<\infty}$, 
and Lemma \ref{lem:propsol2}.
So is the norm estimate $\|u_0\|_p\sim\|u\|_{L^\infty(L^p)}$.

\noindent
\eqref{tildeN}$\implies$\eqref{borneLp}: Let $t>0$. Using 
Proposition~\ref{prop:nrjloc} and H\"older's inequality, we have that
\begin{align*}
\|u(t,\cdot)\|_p
=&\Bigl(\int_{{\mathbb{R}}^n}\fint_{B(x,\frac{\sqrt{t}}{2})}|u(t,y)|^p
\,{\rm d}y\,{\rm d}x\Bigr)^{\frac{1}{p}}
\\
\lesssim&
\Bigl(\int_{{\mathbb{R}}^n}\Bigl(\fint_{B(x,\frac{\sqrt{t}}{2})}|u(t,y)|^2
\,{\rm d}y\Bigr)^{\frac{p}{2}}\,{\rm d}x\Bigr)^{\frac{1}{p}}
\\
\lesssim&
\Bigl(\int_{{\mathbb{R}}^n}\Bigl(\fint_{\frac{t}{2}}^t\fint_{B(x,\sqrt{t})}|u(\sigma,y)|^2
\,{\rm d}y\,{\rm d}\sigma\Bigr)^{\frac{p}{2}}\,{\rm d}x\Bigr)^{\frac{1}{p}}
\le \|\tilde N(u)\|_p.
\end{align*}
\eqref{repLp}$\implies$\eqref{tildeN}:
Let $r\in(q,p)$, and $x \in {\mathbb{R}}^n$, $\delta>0$. Using 
Lemma~\ref{lem:ODp-2} we have that
$$
\Bigl(\fint_{\frac{\delta}{2}}^\delta\fint_{B(x,\sqrt{\delta})}|\Gamma(t,0)u_0(y)|^2
\,{\rm d}y\,{\rm d}t\Bigr)^{\frac{1}{2}}
\lesssim \bigl(M_{HL}|u_0|^r(x)\bigr)^{\frac{1}{r}},
$$
with constants independent of $x,\delta$.
Therefore $\|\tilde N(u)\|_{L^p}\lesssim \|u_{0}\|_{L^p}\le \|u\|_{L^\infty(L^p)}$ 
as we have shown in the proof of Theorem~\ref{thm:uniqueness_p<2} that 
$t\mapsto \Gamma(t,0)u_{0}$ is continuous in $L^p({\mathbb{R}}^n)$.
Moreover $u$ is a global weak solution of \eqref{eq1} by Lemma~\ref{lem:propsol2}.
\end{proof}

\subsection{Further results}

Without any assumption on the propagators, we have proven well posedness 
results in the class $X^p$ for $p>2$. We now consider solutions in 
$L^\infty(L^p)$ under an $L^p$ boundedness assumption on the propagators. 
Note that, contrary to the case $p<2$, we do not need to make assumptions 
about the boundedness of the propagators for different values of $p$.

\begin{proposition}
\label{cor:cons1}
Let $p \in (2,\infty] $. Assume
that $\displaystyle{\sup_{0\leq s\le t<\infty}\|\Gamma(t,s)\|_{{\mathscr{L}}(L^p)}<\infty}$. 
For $u \in {\mathscr{D}}'$,
the following assertions are equivalent.
\begin{align}
&u \mbox{ is a global weak solution of \eqref{eq1} in } L^\infty(L^p) ;
\label{un1}\\[4pt]
&\exists !\,u_0\in L^p({\mathbb{R}}^n) \mbox{ such that }u(t,\cdot)=\Gamma(t,0)u_0
\mbox{ in }L^p({\mathbb{R}}^n)\mbox{ for all }t>0.
\label{un2}
\end{align}
In this case, $\|u_0\|_p\sim \|u\|_{L^\infty(L^p)} \sim \|u\|_{X^p}$.

\noindent
Moreover, if $p<\infty$ and $\displaystyle{\sup_{0\leq s\le t<\infty}
\|\Gamma(t,s)\|_{{\mathscr{L}}(L^r)}<\infty}$ for some $r \in (p,\infty)$ then
$u\in{\mathscr{C}}_0(L^p)$.
\end{proposition}

\begin{proof}
Proposition~\ref{prop:propsol} and the assumption give us that \eqref{un2} 
implies \eqref{un1}, with $\|u_0\|_{L^p} \sim \|u\|_{L^\infty(L^p)}$. 
We now prove that \eqref{un1} implies \eqref{un2}. Proceeding as in the proof 
of Theorem~\ref{thm:uniqueness_p<2}, we only have to show that
$$
\|\Gamma(t,s)^*h-\Gamma(t,0)^*h\|_{L^{p'}}
\xrightarrow[s \to 0]{}0
$$ 
for all $t>0$ and $h \in {\mathscr{C}}_c({\mathbb{R}}^n)$. 
Let $M>0$ be such that $h$ is supported in $B(0,M)$. For all $j\ge 1$ 
and $t>s>0$, we have that
$$
\|1\!{\rm l}_{S_j(0,M)}(\Gamma(t,s)^*-\Gamma(t,0)^*)h\|_{L^{p'}}
\lesssim_M 2^{jn(\frac{1}{2}-\frac{1}{p})}
\|1\!{\rm l}_{S_j(0,M)}(\Gamma(t,s)^{*}
-\Gamma(t,0)^*)h\|_{L^2}.
$$
For each $j\ge 1$, the right hand side converges to 0 when $s\to 0$ by strong 
continuity of $s\mapsto \Gamma(t,s)^*h$ by Proposition~\ref{prop:adjointGamma}. 
Combining this estimate with Proposition~\ref{prop:OD}, we have the following 
for all $j \geq 2$, and some constant $c>0$:
$$
\|1\!{\rm l}_{S_j(0,M)}(\Gamma(t,s)^*-\Gamma(t,0)^*)h\|_{L^{p'}}
\lesssim_M 2^{jn(\frac{1}{2}-\frac{1}{p})}e^{-c\frac{4^{j}}{t}}
\|h\|_{L^2},
$$
with constant independent of $s$ when $s<t/2$. Therefore, we can apply 
dominated convergence for sums to obtain 
$$
\|\Gamma(t,s)^*h-\Gamma(t,0)^*h\|_{L^{p'}} 
\le \sum_{j\ge 1} \|1\!{\rm l}_{S_j(0,M)}(\Gamma(t,s)^*-\Gamma(t,0)^*)h\|_{L^{p'}} 
\xrightarrow[s \to 0]{}0.
$$
The uniqueness of $u_0$ follows from convergence in $L^2_{\rm loc}$ of
$u$, since we know that $u(t,\cdot)=\Gamma(t,0)u_{0}$ for all $t>0$. 
The equivalence of norms follows from the above and 
Corollary~\ref{thm:uniqueXp}.
 
If we assume that $p<\infty$ and that $\displaystyle{
\sup_{0\leq s\le t<\infty}\|\Gamma(t,s)\|_{{\mathscr{L}}(L^r)}<\infty}$ 
for some $r \in (p,\infty)$, then we obtain that $u \in {\mathscr{C}}_0(L^p)$ 
exactly as in the proof of Theorem~\ref{thm:uniqueness_p<2}.
\end{proof}

An interesting corollary is the following weak maximum principle without continuity.

\begin{corollary} 
Assume that $\displaystyle{C=
\sup_{0\leq s\le t<\infty}\|\Gamma(t,s)\|_{{\mathscr{L}}(L^\infty)}<\infty}$. 
Then any global weak solution $u$ of \eqref{eq1} in 
$L^\infty({\mathbb{R}}^{n+1}_+)$ satisfies
$$
\sup_{t>0}\|u(t,\cdot)\|_{L^\infty({\mathbb{R}}^n)} 
\le C \|f\|_{L^\infty({\mathbb{R}}^n)},
$$
where $f$ is the initial value of $u$ (which exists as limit in the 
$L^2_{\rm loc}$ sense). 
\end{corollary}
 
We end this section with another corollary assuming pointwise bounds.  Remark that this does not include $p=1$. 
  
\begin{corollary} 
Assume the propagators $\Gamma(t,s)$, $0\le s <t<\infty$, have kernels bounds. 
Let $1<p\le\infty$. For $u\in {\mathscr{D}}'$, the following assertions are equivalent.
\begin{align}
&u \mbox{ is a global weak solution of \eqref{eq1} in } L^\infty(L^p) ;
\\[4pt]
&\exists !\,u_0\in L^p({\mathbb{R}}^n) \mbox{ such that }u(t,\cdot)=\Gamma(t,0)u_0
\mbox{ in }L^p({\mathbb{R}}^n)\mbox{ for all }t>0 ;
\\[4pt]
&u \mbox{ is a global weak solution of \eqref{eq1} such that }
\tilde u\in L^p({\mathbb{R}}^n), 
\end{align}
where   
$$
\tilde u(x) =\sup_{t>0}\esssup_{y; |y-x|< 4\sqrt{t}} |u(t,y)|, \quad x \in {\mathbb{R}}^n.
$$ 
In this case, $u\in{\mathscr{C}}_0(L^p)$ 
and $\|u_0\|_p\sim \|u\|_{L^\infty(L^p)}\sim\|\tilde u\|_p \sim \|u\|_{X^p}$.
\end{corollary}

Recall that solutions have no reason to be defined at each point, hence the 
variant of the pointwise maximal function.  

\begin{proof} 
As mentioned, $\Gamma(t,s)$ extends to uniformly bounded operators on 
$L^p$ when $t\ge s>0$. Corollary~\ref{cor:kioloa} and Proposition~\ref{cor:cons1} 
thus yield the result, at least for the modified non-tangential maximal function 
$\tilde{N}(u)$ instead of the standard non-tangential maximal function $u^*$. 
However, $\|\tilde{N}(u)\|_{L^p} \sim \|\tilde u\|_{L^p}$.
Indeed, we first observe that $\tilde{N}(u)\leq \tilde u$. 
A converse inequality $\tilde u \lesssim \tilde{N}_\beta(u)$, for some 
$\beta>0$, follows from the local boundedness properties of solutions as stated 
in Proposition~\ref{prop:HK}. Since $\|\tilde{N}_\beta(u)\|_{L^p} 
\sim \|\widetilde{N}(u)\|_{L^p}$, the proof is complete.
\end{proof}

\section{Close to constant or bounded variation time dependency}
\label{sec6}

In this section, we obtain well-posedness results for $L^p$ data when 
$p<2$. It seems to us that one should be able to extend the following results 
to $p>2$ but this would require other methods and we leave this open. 

\subsection{More about gradient bounds for semigroups}

We need to use the following quantified version of the boundedness property 
for the gradient of semigroups for autonomous problems.

\begin{definition}
\label{def:unifpmoins} 
For $1 \le q <2$,  $\Lambda,\lambda>0$, 
$M:[2,q')\to(0,\infty)$, let us define
${\mathcal{M}}(\Lambda,\lambda, q ,M) \subset 
L^\infty({\mathbb{R}}^{n};{\mathscr{M}}_n({\mathbb{C}}))$ by
$A \in {\mathcal{M}}(\Lambda,\lambda,q,M)$
if and only if $A$ satisfies \eqref{ell} with constants $\Lambda, \lambda$, 
and the following holds for $L=-{\rm div}\,A\nabla$,
\begin{align*}
\sup_{t>0}\|\sqrt{t}\nabla e^{-tL^*}\|_{{\mathscr{L}}(L^r)}
\le M(r)<\infty \quad \forall r\in [2,q'). 
\end{align*}
\end{definition}

As mentioned in the proof of Proposition~\ref{prop:AKMP}, this implies that 
there exists a function $M': [1,2] \cap \bigl(\frac {nq}{n+q},2\bigr]\to (0,\infty)$, 
such that 
$$
\sup_{t\ge0}\|e^{-tL}\|_{{\mathscr{L}}(L^p)}\le M'(p)<\infty \quad 
\forall\, p\in[1,2] \cap \bigl(\textstyle{\frac {nq}{n+q}},2\bigr]
$$
Recall that for $p=2$, $M'(2)=1$ by the contraction property ot the semigroup. 

\begin{remark} 
Any $A$ constant, or even continuous and periodic or almost periodic on 
${\mathbb{R}}^n$ belongs to ${\mathcal{M}}(\Lambda,\lambda,1,M)$ for 
some function $M$ (see \cite[Section 3]{Au07} and the references therein).
\end{remark}

\begin{definition}
Let $A \in L^\infty({\mathbb{R}}^{n+1}_+;{\mathscr{M}}_n({\mathbb{C}}))$ 
and $I\subset {\mathbb{R}}_+$ be a bounded interval. We define
$A_I = \fint_I A(t,.)\,{\rm d}t \in 
L^\infty({\mathbb{R}}^n;{\mathscr{M}}_n({\mathbb{C}}))$.
\end{definition}

\begin{lemma}
\label{lem:AonI}
If $A\in L^\infty({\mathbb{R}}^{n+1}_+;{\mathscr{M}}_n({\mathbb{C}}))$ 
satisfies \eqref{ell}, then there exist $q\in [1,2)$, and $M:[2,q')\to(0,\infty)$ 
such that 
$$
A_I\in{\mathscr{M}}(\Lambda,\lambda,q,M)
\quad \mbox{for all bounded interval } I.
$$
\end{lemma}

\begin{proof}
It is immediate that $A_I$ satisfies \eqref{ell} with 
constants $\Lambda,\lambda$. We need the existence of $q$ and $M$ 
that works for all $A_I$. This is provided by Remark~\ref{rem:ell}.   
\end{proof}

\subsection{Existence and uniqueness for $p<2$ with $BV(L^\infty)$ coefficients}
\label{subsec:BVex}

\begin{definition}
\label{def:BV(Linfty)}
We denote by $BV(L^\infty):= 
BV([0,\infty);L^\infty({\mathbb{R}}^n;{\mathscr{M}}_n({\mathbb{C}})))$ 
the space of functions 
$A: (0,\infty) \to L^\infty({\mathbb{R}}^n; {\mathscr{M}}_n({\mathbb{C}}))$  
with (semi-)norm
$$
\|A\|_{BV(L^\infty)}=\sup\Bigl\{\sum_{k=0}^\infty
\|A(t_{k+1},\cdot)-A(t_k,\cdot)\|_{L^\infty} ;
(t_k)_{k\in{\mathbb{N}}}\mbox{ non decreasing in }[0,\infty)\Bigr\}.
$$
\end{definition} 

If the semi-norm is zero then $A$ is independent of $t$. The $BV$ condition 
can thus be seen as a (large) perturbation of the autonomous case. 

Let 
$$
A(t,x)= \sum_{k=0}^\infty 1\!{\rm l}_{[t_k,t_{k+1})}(t)A_k(x)
$$
with $A_k\in L^\infty({\mathbb{R}}^n; {\mathscr{M}}_n({\mathbb{C}}))$ 
for all $k\in {\mathbb{N}}$, and $(t_k)_{k \in {\mathbb{N}}}$ increasing from 
$t_0=0$ to $\infty$. It is easy to see that  $A\in  BV(L^\infty)$ if and only if 
$\displaystyle{\sum_{k=0}^\infty \|A_{k+1}-A_k\|_{L^\infty} <\infty}$, and in 
this case the sum equals $\|A\|_{BV(L^\infty)}$. Moreover, if all $A_k$ 
satisfy \eqref{ell} with same ellipticity constants $\lambda,\Lambda$, then so 
does $A$. This is representative of the general situation thanks to the next lemma.    

\begin{lemma}
\label{lem:bvapp}
Let $A \in L^\infty({\mathbb{R}}^{n+1}_+;{\mathscr{M}}_n({\mathbb{C}}))
\cap BV(L^\infty)$ satisfy \eqref{ell} with constants $\Lambda, \lambda$.
For $j\in {\mathbb{N}}$, and $(t,x) \in {\mathbb{R}}^{n+1}_+$, let us define
$$
A_j(t,x) = \sum_{m=0}^\infty 1\!{\rm l}_{[\frac{m}{2^j},\frac{m+1}{2^j})}(t) 
\fint_{\frac{m}{2^j}}^{\frac{m+1}{2^j}} A(s,x)\,{\rm d}s.
$$
Then 
\begin{enumerate}[(i)]
\item
For all $j \in {\mathbb{N}}$, $A_j$ satisfies \eqref{ell} with constants 
$\Lambda, \lambda$.
\item
For almost every $(t,x)\in (0,\infty)\times {\mathbb{R}}^n$, 
$A_j(t,x) \xrightarrow[j \to \infty]{}A(t,x)$.
\item
For all $j \in {\mathbb{N}}$, $\|A_j\|_{BV(L^\infty)} \leq \|A\|_{BV(L^\infty)}$.
\end{enumerate}
\end{lemma}

\begin{proof}
$(i)$ and $(ii)$ follow directly from the definition of $A_j$ and Lebesgue's 
differentiation theorem. We turn to $(iii)$. By the discussion above 
\begin{align*}
\|A_j\|_{BV(L^\infty)} &= \sum_{m=0}^\infty 
\Big\|\fint_{\frac{m+1}{2^j}}^{\frac{m+2}{2^j}} A(s,x)ds - 
\fint_{\frac{m}{2^j}}^{\frac{m+1}{2^j}} A(s,x)ds \Big\|_{L^\infty}
\\
& \le \fint_0^{2^{-j}} \sum_{m=0}^{\infty}  
\Bigl\|A\Bigl(\frac{m+1}{2^j}+s,\cdot\Bigr)- 
A\Bigl(\frac{m}{2^j}+s,\cdot\Bigr)\Bigr\|_{L^\infty}\,{\rm d}s
\  \le \|A\|_{BV(L^\infty)}.
\qedhere
\end{align*}
\end{proof}

\begin{lemma}
\label{lem:Aescalier}
Let $q\in [1,2)$, $M:[2,q')\to(0,\infty)$, and $\Lambda,\lambda>0$.
Let $A\in L^\infty({\mathbb{R}}^{n+1}_+;{\mathscr{M}}_n({\mathbb{C}}))
\cap BV(L^\infty)$ be of the form
$$
A(t,x)= \sum_{k=0}^\infty 1\!{\rm l}_{[t_k,t_{k+1})}(t)A_k(x)
$$
with $A_k\in{\mathcal{M}}(\Lambda,\lambda,q,M)$ for all $k\in {\mathbb{N}}$,
and $(t_k)_{k \in {\mathbb{N}}}$ increasing from $t_0=0$ to $\infty$. Then, 
for all $p\in\bigl(\max\{(1,\frac{2n}{n+q'}\},2\bigr)$, and $v_0\in L^p({\mathbb{R}}^n)$, 
$$
\bigl[v:(t,x)\mapsto \Gamma(t,0)v_0(x)\bigr]
\in L^\infty(L^p)
$$
and $\|v_0\|_p\sim\|v\|_{L^\infty(L^p)}$, with constants depending only on 
$p,q,M, \lambda,\Lambda$ and the $BV$ norm of~$A$. 
\end{lemma}

\begin{remark}
\label{rem:huang2}  
The range of $p$ within $[1,2)$ depends only on the one of 
Proposition~\ref{prop:AKMP}. According to Remark~\ref{rem:huang}, 
$\frac{2n}{n+q'}$ can be improved to $\frac{qn}{n+q}$. 
\end{remark}
  
\begin{proof}
By density, it is enough to assume 
$v_0 \in L^2({\mathbb{R}}^n) \cap L^p({\mathbb{R}}^n)$. 
Then $v \in {\mathscr{C}}_0(L^2)$. For $k \in {\mathbb{N}}$, set 
$v_k = \Gamma(t_k,0)v_0$ and
$$
w_k(t,\cdot) = 
\begin{cases}
e^{-(t-t_k)L_k}v_k, & \mbox{if } t\geq t_k,\\
0, & \mbox{if } t< t_k,
\end{cases}
$$
where $L_k=-{\rm div}\,A_k\nabla$.
Observe that, for $t \in [t_k,t_{k+1})$, and $s \in [t_i,t_{i+1})$ with 
$i \leq k$, we have that
$$
\Gamma(t,s)=e^{-(t-t_k)L_k}e^{-(t_k-t_{k-1})L_{{k-1}}}\ldots e^{-(t_{i+1}-s)L_i}.
$$
This was proven for finite sequences $(t_j)_{j=0,\dots,N+1}$, but uniqueness in 
Theorem~\ref{thm:wellposed} gives us this formula even for infinite sequences. 
Thus we have that $v(t,\cdot)=w_k(t,\cdot)$ for all $t \in [t_k,t_{k+1}]$. 
Observe that for all $w \in L^2({\mathbb{R}}^n)$ and $t\geq t_{k+1}$,
\begin{align*}
e^{-(t-t_{k+1})L_{k+1}}&(e^{-(t_{k+1}-t_k)L_k}w)
\\
=&e^{-(t-t_k)L_k}w
-\int_{t_{k+1}}^te^{-(t-\sigma)L_{k+1}}{\rm div}\,(A_{k+1}-A_k)
\nabla e^{-(\sigma-t_k)L_k}w\,{\rm d}\sigma.
\end{align*}
Hence, for $t\geq t_{k+1}$,
$$
w_{k+1}(t,\cdot) = w_k(t,\cdot)-\int_{t_{k+1}}^t e^{-(t-\sigma)L_{k+1}}
{\rm div}\,(A_{k+1}-A_k)\nabla w_k(\sigma,\cdot)\,{\rm d}\sigma.
$$
Therefore, by Proposition~\ref{prop:AKMP} and the value of $p$, we have that
$$
\|1\!{\rm l}_{[t_{k+1},\infty)}\nabla(w_{k+1}-w_k)\|_{T^{p,2}}
\le \|\tilde{\mathcal{M}}_{L_{k+1}}\|_{{\mathscr{L}}(T^{p,2})}
\|A_{k+1}-A_{k}\|_{L^\infty}\|1\!{\rm l}_{[t_{k+1},\infty)}\nabla w_k\|_{T^{p,2}}.
$$
The norms $\|\tilde{\mathcal{M}}_{L_{k+1}}\|_{{\mathscr{L}}(T^{p,2})}$ are 
uniformly bounded from our assumption 
$A_k\in{\mathcal{M}}(\Lambda,\lambda,q,M)$ for all $k\in {\mathbb{N}}$. 
Thus there exists a constant $C>0$, depending only on $p, q, M$ and the 
ellipticity constants in \eqref{ell}, such that   
$$
\|\nabla w_{k+1}\|_{T^{p,2}}
\leq (1+C\|A_{k+1}-A_k\|_{L^\infty})\|
\nabla w_{k}\|_{T^{p,2}}.
$$
Iterating, and using \cite[Proposition~2.1]{AHM12}, we have that
\begin{align*}
\|\nabla w_k\|_{T^{p,2}}
\le \prod_{j=0}^k (1+C\|A_{j+1}-A_{j}\|_{L^\infty})
\|\nabla w_0\|_{T^{p,2}}
\le e^{C\|A\|_{BV(L^\infty)}}\|v_0\|_{L^p},
\end{align*}
since $\displaystyle{\sum_{j=0}^\infty \|A_{j+1}-A_j\|_{L^\infty}
=\|A\|_{BV(L^{\infty})}}$ for this particular $A$. So far, we have not used 
that $1< p<2$ in the statement. We note for further use that
\begin{equation}
\label{tp2est}
\sup_{k \in {\mathbb{N}}} \|1\!{\rm l}_{[t_k,t_{k+1})}\nabla v\|_{T^{p,2}}
\le e^{C\|A\|_{BV(L^\infty)}}\|v_0\|_{L^p}.
\end{equation}
The estimate on $w_k$ is sufficient to control $\|v(t,\cdot)\|_{L^p}$ when 
$1< p<2$ as we now show. Let $t \in [t_k,t_{k+1})$ for some $k \in {\mathbb{N}}$. 
Using successively that $A_k\in{\mathcal{M}}(\Lambda,\lambda,q,M)$ for all 
$k\in {\mathbb{N}}$, \cite[Corollary~3.6 and Theorem~5.1]{Au07} , a change of variable $s\mapsto s-t_k$ in the 
 fourth   line and $p<2$ in the  fifth line in applying 
\cite[Proposition~2.1]{AHM12}, we have the following chain of inequalities, 
with constants independent of $t$ and $k$:
\begin{align*}
\|v(t,\cdot)\|_{L^p}
&=\|e^{-(t-t_k)L_k}v_k\|_{L^p} 
\\
&  \lesssim \|v_k\|_{L^p}
\\
&\lesssim \Bigl\|\Bigl(\int_0^\infty\Bigl|\nabla e^{-sL_k}v_k\Bigr|^2
\,{\rm d}s\Bigr)^{\frac{1}{2}}\Bigr\|_{L^p} 
\\
&=\Bigl\|\Bigl(\int_0^\infty\Bigl|\nabla w_k(s,.)\Bigr|^2
\,{\rm d}s\Bigr)^{\frac{1}{2}}\Bigr\|_{L^p} 
\\
&
\lesssim\Bigl\|\nabla w_k\Bigr\|_{T^{p,2}}
\le e^{C\|A\|_{BV(L^\infty)}}\|v_0\|_{L^p}.
\qedhere
\end{align*}
\end{proof}

\begin{theorem}
\label{thm:BV}
Let $A \in L^\infty({\mathbb{R}}^{n+1}_+;{\mathscr{M}}_n({\mathbb{C}}))
\cap BV(L^\infty)$ satisfy \eqref{ell} with constants $\Lambda, \lambda$.
Let  $q\in [1,2)$,  and $M:[2,q')\to(0,\infty)$ be such that 
$A_I \in {\mathcal{M}}(\Lambda,\lambda,q,M)$ for all bounded intervals 
$I$ of ${\mathbb{R}}_+$.

\noindent
Let $p \in \bigl(\max\{1,\frac{2n}{n+q'}\},2\bigr)$\footnote{The range can be 
larger according to Remark \ref{rem:huang2}.} and $u_0 \in L^p({\mathbb{R}}^n)$. 
Then 
\begin{enumerate}[(i)]
\item
$\displaystyle{\sup_{0\le s\le t<\infty}\|\Gamma(t,s)\|_{{\mathscr{L}}(L^p)}<\infty}$.
\item
The function $u:(t,x)\mapsto \Gamma(t,0)u_0(x)$ is the unique global weak solution 
of \eqref{eq1} in $L^\infty(L^p)$ or in $X^p$ such that $u(0,\cdot)=u_0$. 

\noindent
Moreover, $u \in {\mathscr{C}}_0(L^p)$, and
$\|u\|_{L^\infty(L^p)} \sim \|u_0\|_{L^p} \sim \|u\|_{X^p}$. 
\item
The solution $u$ given in $(ii)$ is such that $\nabla u \in T^{p,2}$, and 
$\|\nabla u\|_{T^{p,2}} \sim \|u_0\|_{L^p}$.
\end{enumerate}
\end{theorem}

\begin{proof}
By Corollary~\ref{cor:kioloa}, we have that $(i)$ implies $(ii)$. 

Next, $(ii)\implies(iii)$ is proven, using independent arguments that do not 
rely on the $BV(L^\infty)$ assumption, in Proposition~\ref{prop:NbyS} 
and Theorem~\ref{thm:SbyN<2}.

Let us now prove $(i)$. Assume that 
$u_0 \in L^2({\mathbb{R}}^n)\cap L^p({\mathbb{R}}^n)$. Let
$\{A_j,j\in {\mathbb{N}}\}$ be the family of approximations of $A$ defined in 
Lemma~\ref{lem:bvapp}. Let $u^{(j)}$ denote the corresponding global 
weak solution of $\partial_t v = {\rm div}\,A_j\nabla v$.
By Lemma~\ref{lem:Aescalier}, we have that 
$\|u_0\|_{L^p} \sim \|u^{(j)}\|_{L^\infty(L^p)}$ for all 
$j \in {\mathbb{N}}$ with implied constants independent of $j$. Moreover 
$\|u_0\|_{L^2} \sim \|u^{(j)}\|_{L^\infty(L^2)} \sim
\|\nabla u^{(j)}\|_{L^2(L^2)}$ uniformly in $j \in {\mathbb{N}}$. 
Therefore, there exists a subsequence $(v^{(j)})_{j\in {\mathbb{N}}}$ of 
$(u^{(j)})_{j\in {\mathbb{N}}}$, a function $v \in L^\infty(L^2)$, and a function 
$u \in L^\infty(L^p)$ such that
$$
\begin{array}{rcll}
v^{(j)}&\xrightarrow[j\to\infty]{}&v&\quad \mbox{weak$^*$ in }L^\infty(L^2),
\\
\nabla v^{(j)}&\xrightarrow[j\to\infty]{}&\nabla v&\quad \mbox{weak$^*$ in }L^2(L^2),
\\
v^{(j)}&\xrightarrow[j\to\infty]{}&u&\quad \mbox{weak$^*$ in }L^\infty(L^p).
\end{array}
$$
We have that $v=u$ as distributions, and that $v$ is a global weak solution 
of \eqref{eq1}. 
By Theorem~\ref{thm:wellposed}, it follows that $v(t,x) = \Gamma(t,0)u_0(x)$ 
for all $t\ge 0$ and almost every $x \in {\mathbb{R}}^n$.
Therefore, for all $t \ge 0$, $\|\Gamma(t,0)u_0\|_{L^p} 
=\|v(t, .)\|_{L^p}\lesssim \|u_0\|_{L^p}$.
Thus $\Gamma(t,0)$ extends to a bounded operator on $L^p({\mathbb{R}}^n)$, 
with norm independent of $t$. Starting at $s>0$ instead of $0$ gives in the 
same way that
$\displaystyle{\sup_{t\in [s, \infty)} \|\Gamma(t,s)\|_{{\mathscr{L}}(L^p)}}$ 
is controlled by the $BV(L^\infty)$ norm of $A$ on the interval $[s,\infty)$, 
which is smaller than the one on $[0,\infty)$. 
\end{proof}

\begin{remark}
Curiously, we are not able to prove $(iii)$ using the approach of 
Lemma~\ref{lem:Aescalier}.
\end{remark}

\begin{remark} 
In the general situation, we can obtain all values of $p\in (1,2)$ if $n=1,2$, 
and all values $p \in \bigl(\frac{2n}{n+2}-\varepsilon(\Lambda,\lambda), 2\bigr)$ 
if $n\ge 3$. If $A(t,x)$ depends only on $t$ or is periodic and continuous with respect to $x$ for all $t>0$ with common period, or even almost periodic 
 for all $t>0$, then we obtain $p\in (1,2)$ in any dimension. 
\end{remark}
 
\subsection{Existence and uniqueness for $p<2$:  small perturbations 
of autonomous equations or continuous coefficients}

We now turn to an existence and uniqueness result for small perturbations of 
an autonomous problem or for continuous coefficients on a finite interval. 
We start with the following variant of Duhamel's formula. 

\begin{lemma}
\label{lem:babyDuhamel}
Let $f \in L^2(L^2)$ and $h \in L^2({\mathbb{R}}^n)$. Let 
$\underline{A}\in L^\infty({\mathbb{R}}^n,{\mathscr{M}}_n({\mathbb{C}}))$ 
satisfy \eqref{ell} and $L=-{\rm div}\,\underline{A} \nabla$.
Define, for all $t>0$,
$$
u(t,\cdot) = e^{-tL}h + {\mathcal{R}}_Lf(t,\cdot),
$$
where 
$$
{\mathcal{R}}_Lf: (t,x) \mapsto
\int_0^t e^{-(t-s)L}{\rm div}\,f(s,\cdot)(x)\,{\rm d}s,
$$
is the bounded operator from $T^{2,2}$ to $X^2$ from 
Proposition~\ref{prop:subMR}.
Then $u$ is the unique element of $\dot{W}(0,\infty)$ such that, for all 
$\phi \in {\mathscr{D}}$,
$$
\langle u, \partial_t \phi \rangle = \langle \underline{A} \nabla u, \nabla \phi \rangle 
+ \langle f, \nabla \phi \rangle,
$$
and ${\rm Tr}(u) = h$.
\end{lemma}

\begin{proof}
We first assume that $f \in {\mathscr{D}}$. Define
$v_0:(t,x)\mapsto e^{-tL}h(x)$ and
$$
v = v_0 + \mathcal{R}_Lf.
$$
By semigroup theory, $v \in {\mathscr{C}}(L^2)$, and satisfies
$\partial_t v = -Lv_0 +{\rm div}\, f$. By Proposition~\ref{prop:RvM} and
Step~0 of the proof of Theorem~\ref{thm:wellposed}, we have that
$\nabla v \in L^2(L^2)$, and thus
$$
\langle v, \partial_t \phi \rangle = \langle \underline{A} \nabla v, \nabla \phi \rangle 
+ \langle f, \nabla \phi \rangle,
$$
for all $\phi \in {\mathscr{D}}$, as well as ${\rm Tr}(v) = h$. 

Now, we turn to a general $f \in L^2(L^2)$, and let $(f_k)_{k \in {\mathbb{N}}}$ be a 
sequence of functions in ${\mathscr{D}}$ converging to $f$ in $L^2(L^2)$. 
Define, for all $k\in {\mathbb{N}}$, 
$$
u_k = v_0 + {\mathcal{R}}_Lf_k,\quad\mbox{ and }\quad
u = v_0 + {\mathcal{R}}_Lf.
$$
Then $u_k \xrightarrow[k \to \infty]{} u$ in $X^2$ and 
$\nabla u_k\xrightarrow[k \to \infty]{} \nabla u$ in $L^2(L^2)$,
using Propositions~\ref{prop:subMR} and~\ref{prop:AKMP}.
Therefore
$$
\langle u, \partial_t \phi \rangle = \langle \underline{A} \nabla u, \nabla \phi \rangle 
+ \langle f, \nabla \phi \rangle,
$$
for all $\phi \in {\mathscr{D}}$. Since ${\rm Tr}(v_k)=h$ for all $k \in {\mathbb{N}}$, 
and ${\rm Tr}$ is continuous from $\dot{W}(0,\infty)$ to $L^2$ by 
Lemma~\ref{lem:struct}, we also have that ${\rm Tr}(u)=h$. 

We turn to uniqueness. Let 
$\tilde{u} \in \dot{W}(0,\infty)$ be another solution of
$$
\langle \tilde u, \partial_t \phi \rangle 
= \langle \underline{A} \nabla \tilde u, \nabla \phi \rangle 
+ \langle f, \nabla \phi \rangle,
$$
for all $\phi \in {\mathscr{D}}$, with ${\rm Tr}(\tilde{u})=h$. Then 
$w=u-\tilde{u}$ is a solution of
$$
\partial_t w={\rm div}\,\underline{A}\nabla w,
\quad
w \in \dot{W}(0,\infty),
\quad 
{\rm Tr}(w)=0,
$$
and thus $u=\tilde{u}$ by Theorem~\ref{thm:wellposed}.
\end{proof}

\begin{corollary}
Let $A\in L^\infty({\mathbb{R}}^{n+1}_+,{\mathscr{M}}_n({\mathbb{C}}))$ and 
$\underline{A}\in L^\infty({\mathbb{R}}^n,{\mathscr{M}}_n({\mathbb{C}}))$ 
satisfy \eqref{ell}. Let $L=-{\rm div}\,\underline{A} \nabla$.
For all $t>0$ and $h \in L^2({\mathbb{R}}^n)$, the following holds in 
$L^2({\mathbb{R}}^n)$:
\begin{equation}
\label{Duhamel}
\Gamma(t,0)h = e^{-tL}h + \int_0^t e^{-(t-s)L}{\rm div}\,(A(s,.)-\underline{A})
\nabla \Gamma(s,0)h\,{\rm d}s.
\end{equation}
\end{corollary}

\begin{proof}
Let $h \in L^2({\mathbb{R}}^n)$, define $v_0(t,\cdot)=e^{-tL}h$, and
$f(t,\cdot) = (A(t,\cdot)-\underline{A})\nabla \Gamma(t,0)h$
for all $t>0$. We have that $f \in L^2(L^2)$ by Theorem~\ref{thm:wellposed}. 
Define $u = v_0 + {\mathcal{R}}_Lf$, and $\tilde{u}(t,.)=\Gamma(t,0)h$ for all 
$t>0$. Using Lemma~\ref{lem:babyDuhamel}, we have the following, for 
all $\phi \in {\mathscr{D}}$:
$$
\langle u, \partial_t \phi \rangle
= \langle \underline{A}\nabla u, \nabla \phi \rangle
+\langle (A-\underline{A})\nabla \tilde{u}, \nabla \phi \rangle,
$$
Since $\tilde{u} \in \dot{W}(0,\infty)$ is a global weak solution of \eqref{eq1} with 
${\rm Tr}(\tilde{u}) = h$, we have that $u-\tilde{u}\in \dot{W}(0,\infty)$ is a global 
weak solution of $\partial_t(u-\tilde{u})={\rm div}\,\underline{A} \nabla (u-\tilde{u})$, 
with ${\rm Tr}(u-\tilde{u})=0$. Therefore $u=\tilde{u}$ 
by Theorem~\ref{thm:wellposed}.
\end{proof}

\begin{theorem}
\label{thm:pert}
Let  $q\in[1,2)$, and $M:[2,q')\to(0,\infty)$, and let 
$\underline{A} \in {\mathcal{M}}(\Lambda,\lambda,q,M)$. 
Let $p\in \bigl(\max\{1,\frac{2n}{n+q'}\},2\bigr)$ and assume 
\begin{equation}
\label{eq:small}
\varepsilon:=\|A-\underline{A}\|_{L^\infty}
<\frac{1}{\|\tilde{\mathcal{M}}_{L}\|_{\mathcal{L}(T^{p,2})}},
\end{equation}
where $L=-{\rm div}\,\underline{A} \nabla$. Then, 
$$
\sup_{ 0\le s \le t <\infty}\|\Gamma(t,s)\|_{{\mathscr{L}}(L^p)}<\infty.
$$
Consequently, the conclusions of Corollary~\ref{cor:kioloa} hold in any 
open subinterval $(r, 2)$ on which \eqref{eq:small} is valid. 
\end{theorem}
 
\begin{proof} 
Let $u_0 \in L^p({\mathbb{R}}^n)\cap L^2({\mathbb{R}}^n)$, and define 
$u(t,\cdot) = \Gamma(t,0)u_0$. We want to show that 
$\|u\|_{L^\infty(L^p)} \lesssim \|u_0\|_{L^p}$ with constant independent of $u_0$. 

Let us first assume that $A$ is of the form
$$
A_j(t,x) = \sum_{m=0}^{4^j-1} 1\!{\rm l}_{[\frac{m}{2^j},\frac{m+1}{2^j})}(t) 
\fint_{\frac{m}{2^j}}^{\frac{m+1}{2^j}} A(s,x)\,{\rm d}s 
+ 1\!{\rm l}_{[2^j,\infty)}(t)\underline{A}(x), 
$$
for some $j \in {\mathbb{N}}$ and almost every 
$(t,x) \in (0,\infty)\times{\mathbb{R}}^n$.
With this hypothesis, applying \eqref{tp2est} $4^j+1$ times, we have the 
a priori information that  $\nabla u \in T^{p,2}$ with norm depending on $j$.  
However, we first show that $\|\nabla u\|_{T^{p,2}} \sim \|u_0\|_{L^p}$ 
independently of $j$. Then we deduce a bound 
on $\|u\|_{X^p}$ and, finally, a bound on $\|u\|_{L^\infty(L^p)}$. 

\medskip

\noindent
{\tt Step 1.}
Using the representation \eqref{Duhamel} with $v(t,\cdot)=e^{-tL}u_0$ 
for all $t>0$, we also have that 
$$
\|\nabla u\|_{T^{p,2}} \le \|\nabla v\|_{T^{p,2}} 
+ \|\tilde{\mathcal{M}}_{L}(A-\underline{A})\nabla u\|_{T^{p,2}}.
$$
Using \cite[Corollary 6.10]{Au07} and Proposition~\ref{prop:AKMP} 
(Recall that $p>\frac{nq}{n+q}$ which is the exponent found in \cite{Au07} 
and $p<2$), this gives us that
$$
\|\nabla u\|_{T^{p,2}} \le C\|u_0\|_{L^p} 
+ \varepsilon\|\tilde{\mathcal{M}}_{L}\|_{\mathcal{L}(T^{p,2})}\|\nabla u\|_{T^{p,2}},
$$
for some constant $C>0$. Therefore, with 
$C'= C (1-\varepsilon\|\tilde{ \mathcal{M}}_L\|_{\mathcal{L}(T^{p,2})})^{-1}$, 
independent of $j$,  we have $\|\nabla u\|_{T^{p,2}} \le C' \|u_0\|_{L^p}$.
Using \eqref{Duhamel} with $L=-\Delta$ and  $w(t,\cdot)=e^{t\Delta}u_0$ 
for all $t>0$, together with a classical conical Littlewood-Paley estimate for 
$w$, we have that
$$
\|u_0\|_{L^p} \sim \|\nabla w\|_{T^{p,2}} 
\lesssim
\|\nabla u\|_{T^{p,2}} + \|\tilde{\mathcal{M}}_{-\Delta}(A-I)\nabla u\|_{T^{p,2}} 
\lesssim \|\nabla u\|_{T^{p,2}}.
$$
{\tt Step 2.} Using Proposition \ref{prop:subMR} together with the 
representation \eqref{Duhamel}  with $L=-\Delta$ and step 1, we have that
$$
\|u\|_{X^p} \lesssim \|w\|_{X^p}+\|\nabla u\|_{T^{p,2}}
\lesssim \|w\|_{X^p}+\|u_0\|_{L^p}.
$$
The $L^p$ boundedness of the non-tangential  maximal function for $w$ yields
$\|w\|_{X^p} \lesssim \|u_0\|_{L^p}$, hence $\|u\|_{X^p} \lesssim \|u_0\|_{L^p}$.

\medskip

\noindent
{\tt Step 3.} For $t>0$, using H\"older's inequalities as $p<2$
and Proposition~\ref{prop:nrjloc}, we have 
\begin{align*}
\|u(t,\cdot)\|_{L^p}^p = \int_{{\mathbb{R}}^n} \fint_{B(x,\sqrt{t})} 
|u(t,x)|^p\,{\rm d}x\,{\rm d}y
\lesssim \int_{{\mathbb{R}}^n} \Bigl(\fint _{B(y,\sqrt{t})} |u(t,x)|^2\,{\rm d}x
\Bigr)^{\frac{p}{2}}\,{\rm d}y \lesssim \|u\|_{X^p}^p.
\end{align*}
For all $u_0 \in L^p({\mathbb{R}}^n)\cap L^2({\mathbb{R}}^n)$, we thus 
have obtained from this and step~2 that  
$$
\sup_{t>0}\|u(t,\cdot)\|_{L^p} \lesssim \|u_0\|_{L^p}.
$$
The operators $\Gamma(t,0)$ thus extend to  bounded operators on 
$L^p({\mathbb{R}}^n)$, and one has the uniform estimate 
 $\displaystyle{\sup_{0< t} \|\Gamma(t,0)\|_{{\mathscr{L}}(L^p)}<\infty}$.
Furthermore, we obtain strong continuity of $\Gamma(t,0)$ at 0 in 
$L^p({\mathbb{R}}^n)$ from the one on $L^2({\mathbb{R}}^n)$ as we work for 
$p$ in an open interval. Thus, $\|u_0\|_{L^p} \le \|u\|_{L^\infty(L^p)}$. 
In conclusion, we have shown that
$$
\|u\|_{X^p} \sim \|u_0\|_{L^p} \sim \|\nabla u\|_{T^{p,2}} \sim \|u\|_{L^\infty(L^p)}.
$$
The same reasoning gives us that $\displaystyle{\sup_{0\le s \le t<\infty} 
\|\Gamma(t,s)\|_{{\mathscr{L}}(L^p)}<\infty}$. Note that the bound is uniform 
for all $j$.

\noindent  
The rest of the proof is identical to the proof of Theorem~\ref{thm:BV}, using 
the family $\bigl\{A_j,j\in {\mathbb{N}}\bigr\}$ of approximations of $A$ at the 
beginning of the proof rather than the approximations given by 
Lemma~\ref{lem:bvapp}.
\end{proof}

\begin{theorem}
\label{thm:cont} 
Assume $A\in {\mathscr{C}}([0,T]; 
L^\infty({\mathbb{R}}^n;\mathscr{M}_{n}({\mathbb{C}})))$ and that 
there are $q\in[1,2)$, and $M:[2,q')\to(0,\infty)$ such that  
$A(s,\cdot) \in {\mathcal{M}}(\Lambda,\lambda,q,M)$ for all 
$s \in [0,T]$. For $p\in \bigl(\max\{1,\frac{2n}{n+q'}\},2\bigr)$, we have that
\begin{equation}
\label{T}
\sup_{0\le s \le t <\le T}\|\Gamma(t,s)\|_{\mathcal{L}(L^p)}<\infty.
\end{equation}
Consequently, the conclusions of Corollary~\ref{cor:kioloa} hold replacing 
$t>0$ by $t\in (0,T]$, global solutions by local solutions on $(0,T)$, and 
${\mathscr{C}}_0(L^p)$ by ${\mathscr{C}}([0,T]; L^p)$. 
\end{theorem}

\begin{proof}
Let $u_0 \in L^p({\mathbb{R}}^n)\cap L^2({\mathbb{R}}^n)$ and define 
$u(t,\cdot) =\Gamma(t,0)u_0$.  We want to show \eqref{T}. 
To do so, we adapt the proof of Theorem~\ref{thm:pert}. 
For $\varepsilon>0$, choose $\delta>0$ according to the uniform continuity of 
$A$ on $[0,T]$ such that $\|A(s,\cdot)-A(s',\cdot)\|_{L^\infty}<\varepsilon$ if 
$|s-s'|\le 2\delta$. We may assume that $\delta =\frac{T}{2^k}$ for some 
$k\in {\mathbb{N}}$. 

We begin by replacing $A$ by $\displaystyle{A_j(t,x) 
= \sum_{m=0}^{2^j-1} 1\!{\rm l}_{\bigl[\frac{mT}{2^j},\frac{(m+1)T}{2^j}\bigr)}(t) 
A\bigl(\textstyle{\frac{mT}{2^j}}, x\bigr)}$, for $j >k$. We still denote the solution 
by $u$ to keep the notation simple. With such coefficients, we know from 
\eqref{tp2est} that $\|\nabla u\|_{T^{p,2}}<\infty$ qualitatively.

\noindent
{\tt Step 1.}
Using the representation \eqref{Duhamel} with $v(,\cdot)= e^{-tL}u_0$, 
$L=-{\rm div}\,A(0,\cdot)\nabla$, we have
$$
\|1\!{\rm l}_{(0,2\delta)}\nabla u\|_{T^{p,2}} 
\le \|1\!{\rm l}_{(0,2\delta)}\nabla v\|_{T^{p,2}} 
+ \|1\!{\rm l}_{(0,2\delta )}\tilde{\mathcal{M}}_L(A_j-{A(0,\cdot}))
\nabla u)\|_{T^{p,2}}.
$$
Note that the truncation implies that the only values of $A(s,x)$ that play a role 
are those for $s\in [0,2\delta]$. Thus as $j\ge k$, we obtain,   
$$
\|1\!{\rm l}_{(0,2\delta)} \nabla u\|_{T^{p,2}} \leq C\|u_0\|_{L^p}
+ \varepsilon\|\tilde{\mathcal{M}}_L\|_{{\mathscr{L}}(T^{p,2})}
\|1\!{\rm l}_{(0,2\delta)}\nabla u\|_{T^{p,2}},
$$
for some constant $C>0$ independent of $j$. 
Therefore, having  first chosen $\varepsilon>0$ with 
$\varepsilon\|{\mathcal{M}}_L\|_{{\mathscr{L}}(T^{p,2})}<1$, using the finiteness of 
$\|1\!{\rm l}_{(0,2\delta)} \nabla u\|_{T^{p,2}}$, we have
$\|1\!{\rm l}_{(0,2\delta)} \nabla u\|_{T^{p,2}} \le C' \|u_0\|_{L^p}$, with 
$C'=C (1-\varepsilon\|\tilde {\mathcal{M}}_L\|_{{\mathscr{L}}(T^{p,2})})^{-1}$.

\medskip

\noindent
{\tt Step 2.} Set $w(t,\cdot)=e^{t\Delta}u_0$. Using Proposition~\ref{prop:subMR} 
together with the representation \eqref{Duhamel} with $L=-\Delta$ and step~1, 
we have that
$$
\|1\!{\rm l}_{(0,2\delta)} u\|_{X^p} \lesssim \|1\!{\rm l}_{(0,2\delta)} w\|_{X^p}
+\|1\!{\rm l}_{(0,2\delta)} \nabla u\|_{T^{p,2}}
\lesssim \|w\|_{X^p}+\|u_0\|_{L^p}.
$$
The $L^p$ boundedness of the non-tangential maximal function for $w$ yields
$\|w\|_{X^p} \lesssim \|u_0\|_{L^p}$, hence 
$\|1\!{\rm l}_{(0,2\delta)} u\|_{X^p} \lesssim \|u_0\|_{L^p}$.

\medskip

\noindent
{\tt Step 3.} For $0<t\le \delta$, using H\"older's inequalities as $p<2$
and Proposition~\ref{prop:nrjloc}, we have 
$$
\|u(t,\cdot)\|_{L^p}^p 
=\int_{{\mathbb{R}}^n} \fint_{B(x,\sqrt{t})} |u(t,x)|^p\,{\rm d}x\,{\rm d}y
\lesssim \int_{{\mathbb{R}}^n} \Bigl(\fint _{B(y,\sqrt{t})} |u(t,x)|^2\,{\rm d}x
\Bigr)^{\frac{p}{2}}\,{\rm d}y \lesssim \|1\!{\rm l}_{(0, 2\delta)} u\|_{X^p}^p.
$$
For all $u_0 \in L^p({\mathbb{R}}^n)\cap L^2({\mathbb{R}}^n)$, we thus have that  
$$
\sup_{0<t\le \delta}\|u(t, \cdot)\|_{L^p} \lesssim  \|u_0\|_{L^p}.
$$
Therefore, the operators $\Gamma(t,0)$ extend to bounded operators on 
$L^p({\mathbb{R}}^n)$, and one has the uniform estimate 
 $\displaystyle{\sup_{ 0< t\le \delta} \|\Gamma(t,0)\|_{{\mathscr{L}}(L^p)}<\infty}$. 
Furthermore, we obtain strong continuity of $\Gamma(t,0)$ at 0 in 
$L^p({\mathbb{R}}^n)$ from the one on $L^2({\mathbb{R}}^n)$ as we work for 
$p$ in an open interval. In conclusion, we have shown that
$$
\|u_0\|_{L^p} \sim \sup_{ 0\le  t\le \delta} \|u(t, \cdot)\|_{L^p}.
$$
Given the form of $A_j$, one can obtain similarly 
$\displaystyle{\sup_{m\delta  \le s \le t\le (m+1)\delta } 
\|\Gamma(t,s)\|_{{\mathscr{L}}(L^p)}<\infty}$.
Iterating at most $\frac{T}{\delta}$ times, using the reproducing formula for 
the propagators, we obtain $\displaystyle{\sup_{0\le s \le t\le T} 
\|\Gamma(t,s)\|_{{\mathscr{L}}(L^p)}<\infty}$.

We conclude for $A$ as in Theorem~\ref{thm:BV} using the above 
approximations $A_j$ of $A$, remarking that the bound obtained for the 
propagators of $A_j$ are uniform for $j$ large enough and depend solely on 
the uniform continuity assumption and $T$.   
\end{proof}
  
\begin{remark} 
In this argument, we only used properties of the semigroups for each coefficients 
$A(s,\cdot)$ on a bounded interval. Thus we may replace the assumption  
$A(s,\cdot) \in {\mathcal{M}}(\Lambda,\lambda,q,M)$ for all 
$s \in [0,T]$, by $A(s,\cdot) \in {\mathcal{M}}_{T}(\Lambda,\lambda,q,M)$ for all 
$s \in [0,T]$. The subscript $T$ means that we consider the supremum in 
Definition~\ref{def:unifpmoins} taken on $(0,T]$. For example, any 
$A\in L^\infty({\mathbb{R}}^n,{\mathscr{M}}_n({\mathbb{C}}))$ that is uniformly 
continuous (or even that belongs to $VMO$) on ${\mathbb{R}}^n$ belongs to 
${\mathcal{M}}_T(\Lambda,\lambda,1,M)$ for some $M$ and all $T>0$; see 
\cite{Au96}.  
In particular for any $A\in {\mathscr{C}}([0,\infty); L^\infty({\mathbb{R}}^n; 
{\mathscr{M}}_n({\mathbb{C}})))$ such that $A(s,\cdot)$ is uniformly continuous 
on ${\mathbb{R}}^n$, uniformly for $s\ge 0$ (the uniformity in $s$ is imposed 
to guarantee that we have the same function $M$ for all $A(s,\cdot)$), we can 
apply our result with $p\in (1,2)$ and obtain global solutions in 
${\mathscr{C}}(L^p)$ (but not bounded). For example, it applies to any 
$A$ which is uniformly continuous on $\overline{{\mathbb{R}}^{n+1}_+}$. 
\end{remark}

\section{Square functions and maximal functions a priori estimates}
\label{sec7}

We prove here some comparisons between conical square functions 
in $L^p({\mathbb{R}}^n)$, namely $\|\nabla u\|_{T^{p,2}}$ and 
non-tangential maximal functions in $L^p$, namely $\|u\|_{X^p}$ for weak solutions of $\partial_tu= {\rm div}\,A \nabla u$. 
In the case of autonomous equations, such bounds are obtained in \cite{HM09} 
for $p\ge 1$.  It is tempting to study the cases where 
$p<1$ as well, but this is outside the scope of the present work.

\subsection{Controlling the maximal function by the square function for $1\le p< \infty$}

As a consequence of Proposition~\ref{prop:subMR} and classical 
Littlewood-Paley theory, we first obtain the following control of the maximal 
function by the square function.

\begin{proposition}
\label{prop:NbyS}
Let $1\leq p< \infty$, $u_0 \in L^2({\mathbb{R}}^n)$, and 
$u(t,\cdot) = \Gamma(t,0)u_0$ for all $t>0$. If $\nabla u \in T^{p,2}$, then 
$u \in X^p$, and
$$
\|u\|_{X^p} \lesssim \|\nabla u\|_{T^{p,2}}
$$
with implicit constant independent of $u$. 
\end{proposition}

\begin{proof} 
Set $v(t,.)= e^{t\Delta}u_{0}$ for all $t>0$. 
Using \eqref{Duhamel} with $L=-\Delta$ and Proposition~\ref{prop:subMR}, 
we have that
$$
\|u\|_{X^p} \lesssim \|v\|_{X^p} + 
\|{\mathcal{R}}_{-\Delta}\|_{\mathcal{L}(T^{p,2},X^p)} \|A-I\|_{L^\infty}
\|\nabla u\|_{T^{p,2}}.
$$
Using \eqref{Duhamel} again, together with the classical conical 
Littlewood-Paley estimate, and Proposition~\ref{prop:AKMP},
we also have that
\begin{equation*}
\|v\|_{X^p} \lesssim \|v^*\|_{L^p}\lesssim  \|\nabla v\|_{T^{p,2}} \lesssim
\|\nabla u\|_{T^{p,2}} + \|\tilde{\mathcal{M}}_{-\Delta}
(A-I)\nabla u\|_{T^{p,2}} \lesssim \|\nabla u\|_{T^{p,2}}.
\qedhere
\end{equation*}
\end{proof}

Note that the range can be improved to $p>\frac n {n+1}$, which is the same range as for the classical conical 
Littlewood-Paley estimate. 

As a corollary, we have the following improvement of Theorem~\ref{thm:wellposed}.

\begin{corollary}
\label{cor:ntmax2}
For all $u_0\in L^2({\mathbb{R}}^n)$, the problem
$$
\partial_tu={\rm div}\,A\nabla u,
\quad
u\in X^2,
\quad 
{\rm Tr}(u)=u_0
$$
is well-posed. Moreover, the solution $u$ is the energy solution, i.e. 
$u(t,\cdot)=\Gamma(t,0)u_0$ for all $t>0$, and
$$
\|u_0\|_{L^2} = \|u\|_{L^\infty(L^2)}\lesssim \|u\|_{X^2}
\lesssim \|\nabla u\|_{L^2(L^2)}\le 
\textstyle{\sqrt{\frac{1}{2\lambda}}}\,\|u_0\|_{L^2}.
$$
\end{corollary}

\begin{proof}
{\tt Existence.} 
Let $u(t,\cdot)=\Gamma(t,0)u_0$ for all $t>0$.
By Proposition~\ref{prop:NbyS} and Theorem~\ref{thm:wellposed}, we have that
$$
\|u\|_{X^2} \lesssim \|\nabla u\|_{T^{2,2}} = \|\nabla u\|_{L^2(L^2
)} \le \textstyle{\sqrt{\frac{1}{2\lambda}}}\,\|u_0\|_{L^2}.
$$ 
{\tt Uniqueness.} 
We have by Proposition~\ref{prop:nrjloc} that $\|u\|_{L^\infty(L^2)} 
\lesssim \|u\|_{X^2}$ for all $u \in X^2$, and $L^\infty(L^2)$ is a class of 
uniqueness as shown in Theorem~\ref{thm:linfl2}.
\end{proof}

\subsection{Controlling the square function by the maximal function for $p \in [1,2)$}

\begin{theorem}
\label{thm:SbyN<2}
Let $u$ be a global weak solution of \eqref{eq1}. Let $p\in [1,2)$, and assume 
that $u \in X^p$. Then $\nabla u \in T^{p,2}$ and 
$$
\|\nabla u\|_{T^{p,2}} \lesssim \|u\|_{X^p},
$$
with constant depending only on the ellipticity parameters in \eqref{ell}.
\end{theorem}

Note that the proof works for $0<p<1$ as well with the definitions of $T^{p,2}$ 
and $X^p$ extended to these values. 

\begin{proof}
The proof is highly similar to its autonomous counterpart in 
\cite[Theorem 6.1]{HM09}, itself based on Fefferman-Stein's original 
argument for $L=-\Delta$ in \cite{FeSt72}.
The only difference is that we need to use cut-off functions rather than integration 
by parts to localise near the boundary of truncated cones, 
Proposition~\ref{prop:nrjloc} instead of Caccioppoli's inequalities, and 
Kenig-Pipher's maximal function instead of the maximal function used in 
\cite{HM09}. We include the full proof for the convenience of the reader. 
Let $\varepsilon,R,\sigma>0$ with $\varepsilon<\frac{R}{6}$. Pick $\beta>0$ 
to be determined later. Recall that
$$
{\mathcal{N}}_\beta(u)(x) = \sup_{\delta>0}
\Bigl(\fint_{\delta^2}^{\beta^2\delta^2} \fint_{B(x,\beta\delta)}|u(t,y)|^2 
\,{\rm d}y\,{\rm d}t\Bigr)^{\frac{1}{2}} \quad \forall x \in{\mathbb{R}}^n.
$$
Define $E = \bigl\{x \in {\mathbb{R}}^n \;;\; {}\mathcal{N}_\beta u(x)\le \sigma\bigr\}$, 
and $E^{*} = \bigl\{x \in E \;;\; \forall r>0 
\quad |B(x,r)\cap E| \ge \frac{1}{2}|B(x,r)|\bigr\}$. 
Let $B = E^c$, $B^* = (E^*)^c$, and, for $x\in {\mathbb{R}}^n$, $\alpha >0$,
$$
\Gamma^{\varepsilon, R, \alpha}(x) =
\bigl\{(t,y) \in [0,\infty)\times {\mathbb{R}}^n \;;\; |y-x| <\alpha t \mbox{ and }  
\varepsilon < t< R\bigr\},
$$
as well as $\mathcal{R}^{\varepsilon, R, \alpha}(E^*) 
= \bigcup _{x \in E^*} \Gamma^{\varepsilon, R, \alpha}(x)$.

\noindent
We also define $\tilde{\mathcal{B}}^{\varepsilon, R}(E^*)=
\tilde{\mathcal{B}}_{\varepsilon}(E^*)\cup \tilde{\mathcal{B}}^{R}(E^*)
\cup\tilde{\mathcal{B}}'(E^*)$, where
\begin{align*}
\tilde{\mathcal{B}}_{\varepsilon}(E^*)
&= \bigl\{(t,y) \in [0,\infty)\times {\mathbb{R}}^n\;;\; t\in (\varepsilon,2\varepsilon)
\mbox{ and }d(y,E^*)<t\bigr\},
\\
\tilde{\mathcal{B}}^{R}(E^*)
&=\bigl\{(t,y) \in [0,\infty)\times {\mathbb{R}}^n \;;\; t\in (R,2R) \mbox{ and }
d(y,E^*)<t\bigr\},
\\
\tilde{\mathcal{B}}'(E^*)
&= \bigl\{(t,y) \in [0,\infty)\times{\mathbb{R}}^n \;;\; t\in (\varepsilon,2R) 
\mbox{ and }t/2 \le  d(y,E^*)< t \bigr\}.
\end{align*}
Note that 
$$
{\mathcal{R}}^{\varepsilon, 2R, 1 }(E^*) ={\mathcal{R}}^{2\varepsilon, R, 1/2}(E^*)
\cup \tilde{\mathcal{B}}^{\varepsilon, R}(E^*).
$$
We remark that
\begin{align*}
\int_{E^*} \Bigl(\int_0^\infty \fint _{B(x,\frac{\sqrt{t}}{2})}|\nabla u(t,y)|^2
\,{\rm d}y\,{\rm d}t\Bigr)\,{\rm d}x
=2\int_{E^*} \Bigl(\int_0^\infty \fint_{B(x,\frac{s}{2})}s|\nabla u(s^2,y)|^2
\,{\rm d}y\,{\rm d}s\Bigr)\,{\rm d}x.
\end{align*}
and  
$$
\int_{E^*} \Bigl(\int_{2\varepsilon}^R \fint_{B(x,s)}s|\nabla u(s^2,y)|^2 
\,{\rm d}y\,{\rm d}s\Bigr)\,{\rm d}x
\le c_n \int_{{\mathcal{R}}^{2\varepsilon,R,1/2}(E^*)} s|\nabla u(s^2,y)|^2
\,{\rm d}y\,{\rm d}s.
$$
with $c_n$ the reciprocal of the volume of the unit ball. We estimate the last 
integral. To do so, set
$$
\chi(t,y) = \Bigl(1-\eta\bigl(\textstyle{\frac{2\,d(y,E^*)}{t}}\bigr)\Bigr)
\eta\bigl(\textstyle{\frac{t}{\varepsilon}}\bigr)
\Bigl(1-\eta\bigl(\textstyle{\frac{t}{R}}\bigr)\Bigr)
\quad \forall (t,y) \in{\mathbb{R}}^{n+1}_+,
$$
where $\eta \in {\mathscr{C}}_c^\infty([0,\infty), [0,1])$ is equal to 
$1$ on $[2,\infty)$, and equal to $0$ on $[0,1]$. 
Notice that $\chi$ is almost everywhere differentiable, supported in
${\mathcal{R}}^{\varepsilon, 2R,1}(E^*)$, and constantly equal to $1$ on 
${\mathcal{R}}^{2\varepsilon, R,\frac{1}{2}}(E^*)$.
Moreover, for all $(t,y) \in {\mathcal{R}}^{\varepsilon, 2R,1}(E^*)$, we have that
$|\nabla \chi(t,y)| \leq \frac{2\|\eta'\|_{\infty}\sqrt n}{t}$.
For $(t,y) \in \tilde{\mathcal{B}}_{\varepsilon}(E^*)$, we have,
$$
|\partial_t \chi(t,y)|\leq 2\|\eta'\|_\infty
\bigl(\textstyle{\frac{1}{t}+\frac{d(y,E^{*})}{t^{2}}}\bigr) \le \frac{4 \|\eta'\|_{\infty}}{t}.
$$
The same reasoning gives us that 
$|\partial_t \chi(t,y)|\ \leq \frac{4 \|\eta'\|_{\infty}}{t}$, for all 
$(t,y) \in \tilde{\mathcal{B}}^{R}(E^*)$. For $(t,y) \in \tilde{\mathcal{B}}'(E^*)$, 
we also have that 
$|\partial_t \chi(t,y)|\le \frac{2\|\eta'\|_{\infty}d(y,E^*)}{t^{2}} 
\le \frac{2\|\eta'\|_\infty}{t}$.
Putting all these estimates together, we have shown that there exists $C>0$ 
such that for all $(t,y) \in {\mathcal{R}}^{\varepsilon, 2R,1}(E^*)$, we have that
$$
|\nabla \chi(t,y)|+|\partial_t \chi(t,y)| \le \frac{C}{t}.
$$
According to Remark~\ref{rem:local}, provided we show that
$(s,x) \mapsto u(s^2,x)\chi^2(s,x)\in L^2(\varepsilon, 2R; H^1({\mathbb{R}}^n))$, 
we can use this function as a test function. We assume this for now, and 
estimate as follows:
\begin{align*}
\int_{{\mathcal{R}}^{2 \varepsilon,R,1/2}(E^*)} s|\nabla u(s^2,y)|^2
\,{\rm d}y\,{\rm d}s 
&\le 2\lambda\,\Re e \int_{{\mathbb{R}}^{n+1}_+} s\chi^2(s,y) A(s^2,y)
\nabla u(s^2,y)\cdot \overline{\nabla u(s^2,y)}\,{\rm d}y\,{\rm d}s
\\
&\lesssim J_1+J_2+J_3,
\end{align*} 
where 
\begin{align*}
J_1 &= \Bigl|\int_{{\mathbb{R}}^{n+1}_+}  
sA(s^2,y) \nabla u(s^2,y)\cdot\overline{u(s^2,y)\nabla (\chi^2(s,y))} 
\,{\rm d}y\,{\rm d}s\Bigr|,
\\
J_2 &=  \Bigl|\int_{{\mathbb{R}}_+} \langle \partial_s(\chi(s,y)u(s^2,y)), 
u(s^2,y) \chi(s,y)\rangle\,{\rm d}s\Bigr|, 
\\
J_3 &= \Bigl|\int_{{\mathbb{R}}^{n+1}_+} u(s^2,y) \partial_s(\chi(s,y))
\overline{u(s^2,y) \chi(s,y)}\,{\rm d}y\,{\rm d}s\Bigr|.
\end{align*}
Here, we notice that for each $s$, the bracket is the duality  
$H^{-1}({\mathbb{R}}^n), H^1({\mathbb{R}}^n)$. 
Hence, $J_2 = \frac{1}{2}\int_{{\mathbb{R}}_+} 
\partial_s\|\chi(s,\cdot)u(s^2,\cdot )\|_{L^2}^2\,{\rm d}s = 0$.
Moreover,
\begin{align*}
J_1 &\lesssim 
\int_{\tilde{\mathcal{B}}^{ \varepsilon, R}(E^*)} |u(s^2,y)||\nabla u(s^2,y)| 
\,{\rm d}y\,{\rm d}s,
\\
J_3 & \lesssim \int_{\tilde{\mathcal{B}}^{ \varepsilon, R}(E^*)} 
|u(s^2,y)|^2 \,\frac{{\rm d}y\,{\rm d}s}{s}.
\end{align*}
This yields
\begin{align*}
&\int_{E^*} \Bigl(\int_{2\varepsilon}^R \fint_{B(x,\frac{s}{2})}
s|\nabla u(s^2,y)|^2\,{\rm d}y\,{\rm d}s\Bigr)\,{\rm d}x
\\
\lesssim&
\int_{\tilde{\mathcal{B}}^{\varepsilon, R}(E^*)} |u(s^2,y)|^2
\,\frac{{\rm d}y\,{\rm d}s}{s} 
+\Bigl(\int_{\tilde{\mathcal{B}}^{\varepsilon, R}(E^*)} |u(s^2,y)|^2
\,\frac{{\rm d}y\,{\rm d}s}{s}\Bigr)^{\frac{1}{2}}
\Bigl(\int_{\tilde{\mathcal{B}}^{\varepsilon, R}(E^*)} 
s|\nabla u(s^2,y)|^2\,{\rm d}y\,{\rm d}s\Bigr)^{\frac{1}{2}}.
\end{align*}
We now consider the following six integrals, recalling that
$\tilde{\mathcal{B}}^{\varepsilon, R}(E^*)=
\tilde{\mathcal{B}}_\varepsilon(E^*)
\cup \tilde{\mathcal{B}}^R(E^*)
\cup\tilde{\mathcal{B}}'(E^*)$.
\begin{align*}
I_1 &= \int_{\tilde{\mathcal{B}}_{\varepsilon}(E^*)}
|u(s^2,y)|^2\,\frac{{\rm d}y\,{\rm d}s}{s},
\qquad \qquad
I_2 = \int_{\tilde{\mathcal{B}}_{\varepsilon}(E^*)}
s|\nabla u(s^2,y)|^2\,{\rm d}y\,{\rm d}s,
\\
I_3 &= \int_{\tilde{\mathcal{B}}^R(E^*)}
|u(s^2,y)|^2\,\frac{{\rm d}y\,{\rm d}s}{s},
\qquad\qquad
I_4 = \int_{\tilde{\mathcal{B}}^R(E^*)}
s|\nabla u(s^2,y)|^2\,{\rm d}y\,{\rm d}s,
\\
I_5 &= \int_{\tilde{\mathcal{B}}'(E^*)}
|u(s^2,y)|^2\,\frac{{\rm d}y\,{\rm d}s}{s},
\qquad \qquad
I_6 = \int_{\tilde{\mathcal{B}}'(E^*)}
s|\nabla u(s^2,y)|^2\,{\rm d}y\,{\rm d}s.
\end{align*}
For $I_1$, we have the following:
\begin{align*}
I_1 &\lesssim 
\int_{\tilde{\mathcal{B}}_{\varepsilon}(E^*)}
\Bigl(\int_{E \cap B(y,s)} s^{-n}\,{\rm d}x\Bigr)
|u(s^2,y)|^2\,\frac{{\rm d}y\,{\rm d}s}{s} \\
& \lesssim 
\int_{\varepsilon}^{2\varepsilon} 
\int_E \fint_{B(x,s)}|u(s^2,y)|^2\,\frac{{\rm d}y\,{\rm d}s\,{\rm d}x}{s} 
\lesssim 
\int_E \int_{\varepsilon^2}^{4\varepsilon^2} 
\fint_{B(x,2\varepsilon)}|u(s,y)|^2\,\frac{{\rm d}y\,{\rm d}s\,{\rm d}x}{s} 
\\
& \lesssim
\int_E \sup_{\delta>0}
\fint_{\delta^2} ^{4\delta^2} \fint_{B(x,2\delta)} |u(s,y)|^2\,{\rm d}y\,{\rm d}s\,{\rm d}x
=\int_E |{\mathcal{N}}_2u(x)|^2\,{\rm d}x.
\end{align*}
To handle $I_2$, we use Proposition~\ref{prop:nrjloc}, and a covering argument, 
to obtain the following:
\begin{align*}
I_2 
& \lesssim 
\int_{\varepsilon} ^{2\varepsilon} 
\int_E \fint_{B(x,2\varepsilon)} s|\nabla u(s^2,y)|^2\,{\rm d}y\,{\rm d}s\,{\rm d}x
\lesssim 
\int_E \int_{\varepsilon^2}^{4\varepsilon^2} 
\fint_{B(x,2\varepsilon)}|\nabla u(s,y)|^2\,{\rm d}y\,{\rm d}s\,{\rm d}x 
\\
& \lesssim
\int_E \fint_{\varepsilon^2 /2}^{4\varepsilon^2} 
\fint_{B(x,4\varepsilon)} |u(s,y)|^2\,{\rm d}y\,{\rm d}s\,{\rm d}x
= \int_E |{\mathcal{N}}_4u(x)|^2\,{\rm d}x.
\end{align*}
In the same way $I_3+I_4 \lesssim \int_E |{\mathcal{N}}_4u(x)|^2\,{\rm d}x$. 
We now turn to $I_5, I_6$, using a Whitney decom\-po\-sition of $B^*$: 
there exist $c_1, c_2 \in(0,1)$ with $c_2>c_1$, and $c_3 \in {\mathbb{N}}$, 
such that $\displaystyle{B^* = \bigcup_{k=0}^{\infty} B(x_k,r_k)}$ 
for some $x_k \in {\mathbb{R}}^n$ and $r_k >0$ such that
\begin{align*}
\forall k \in {\mathbb{N}} \quad c_1\,d(x_k,E^*)\le r_k \le c_2\,d(x_k,E^*),
\\ 
\forall x \in B^* \quad \bigl|\bigl\{k\in {\mathbb{N}} \;;\; 
x \in B(x_k,r_k)\bigr\}\bigr| \le c_3.
\end{align*}
We have the following:
\begin{align*}
I_5 &\lesssim \sum_{k=0}^\infty
\int_{r_k(\frac{1}{c_2}-1)}^{2r_k(\frac{1}{c_1}+1)}
\int_{B(x_k,r_k)} |u(s^2,y)|^2\,\frac{{\rm d}y\,{\rm d}s}{s}
\\
& \lesssim
\sum_{k=0}^\infty r_k^n
\fint_{r_k^2(\frac{1}{c_2}-1)^2}^{4r_k^2(\frac{1}{c_1}+1)^2}
\fint_{B(x_k,\frac{c_2}{1-c_2}\sqrt{s})} |u(s,y)|^2\,{\rm d}y\,{\rm d}s.
\end{align*}
Now remark that, for $k \in {\mathbb{N}}$ and 
$s \in [r_k^2(\frac{1}{c_2}-1)^2,4r_k^2(\frac{1}{c_1}+1)^2]$:
$$
d(x_k,E) \le d(x_k,E^*) \le \frac{r_k}{c_1} \le
\frac{c_2}{c_1(1-c_2)}\sqrt{s}.
$$
Therefore, there exists $x_k' \in E$ such that
$$
B\bigl(x_k,\textstyle{\frac{c_2}{1-c_2}}\sqrt{s}\bigr) \subset
B\bigl(x_k',\textstyle{\frac{c_2}{1-c_2}(\frac{1}{c_1}+1)}\sqrt{s}\bigr).
$$
This yields
$$
I_5 \lesssim 
\sum_{k=0}^\infty r_k^n \,\sup_{k \in {\mathbb{N}}}\Bigl(
\fint_{r_k^2(\frac{1}{c_2}-1)^2}^{4r_k^2(\frac{1}{c_1}+1)^2}
\fint _{B(x_k',\frac{c_2}{1-c_2}(\frac{1}{c_1}+1)\sqrt{s})} 
|u(s,y)|^2\,{\rm d}y\,{\rm d}s\Bigr)
\lesssim |B^*|\, \sup_{z \in E} |{\mathcal{N}}_{\gamma}u(z)|^{2},
$$
for some $\gamma \ge 4$ depending only on $c_1,c_2$.
In the same way, using Proposition~\ref{prop:nrjloc}, we also have that
\begin{align*}
I_6 & \lesssim
\sum_{k=0}^\infty
\int_{r_k^2(\frac{1}{c_2}-1)^2}^{4r_k^2(\frac{1}{c_1}+1)^2}
\int_{B(x_k,\frac{c_2}{1-c_2}\sqrt{s})} |\nabla u(s,y)|^2\,{\rm d}y\,{\rm d}s,
\\
& \lesssim
\sum_{k=0}^{\infty}
\int_{r_k^2(\frac{1}{c_2}-1)^2/2 } ^{4r_k^2(\frac{1}{c_1}+1)^2}   
\int_{B(x_k,\frac{2c_2}{1-c_2}\sqrt{s})} |u(s,y)|^2 
\,{\rm d}y\,{\rm d}s 
\lesssim |B^*|\,  \sup_{z \in E}  |{\mathcal{N}}_{\gamma'}u(z)|^2,
\end{align*}
for some $\gamma'\ge \gamma$ depending only on $c_1,c_2$.
Now fix $\beta = \gamma'$. Summing all the estimates, and taking limits as 
$\varepsilon \to 0$ and $R \to \infty$, we have the following:
$$
\int_{E^*} \int_0^\infty 
\fint_{B(x,\frac{\sqrt{t}}{2})} |\nabla u(t,y)|^2\,{\rm d}y\,{\rm d}t\,{\rm d}x
\lesssim |B^*|\sigma^2 + \int_E
|{\mathcal{N}}_\beta u(z)|^2\,{\rm d}z.
$$
Now we consider the distribution functions defined by
\begin{align*}
g_S(\sigma) &= \Bigl|\Bigl\{x \in {\mathbb{R}}^n \;;\; \Bigl(\int_0^\infty
\fint_{B(x,\frac{\sqrt{t}}{2})} |\nabla u(t,y)|^2\,{\rm d}y\,{\rm d}t\Bigr)^{\frac{1}{2}} 
> \sigma\Bigr\}\Bigr|,
\\
g_N(\sigma) &= \bigl|\bigl\{x \in {\mathbb{R}}^n \;;\; 
{\mathcal{N}}_{\beta}u(x) > \sigma
\bigr\}\bigr|.
\end{align*}
We have that $|B^*| \lesssim |B|=g_N(\sigma)$, and that
$$
\int_{E} |{\mathcal{N}}_\beta u(x)|^2\,{\rm d}x
\le 2 \int_0^\sigma tg_N(t)\,{\rm d}t.
$$
This implies
\begin{align*}
g_S(\sigma) &\lesssim |B^*| + \frac{1}{\sigma^2}
\int_E \int_0^\infty 
\fint_{B(x,\frac{\sqrt{t}}{2})}|\nabla u(t,y)|^2\,{\rm d}y\,{\rm d}t\,{\rm d}x 
\\
& \lesssim |B^*| +\frac{1}{\sigma^2} \int_{E^*}
|{\mathcal{N}}_\beta u(x)|^2\,{\rm d}x 
\lesssim g_N(\sigma) + \int_0^\sigma tg_{N}(t)\,{\rm d}t.
\end{align*}
Therefore, as $p<2$, 
\begin{align*}
\int_0^\infty \sigma^{p-1}g_S(\sigma)\,{\rm d}\sigma
& \lesssim \int_0^\infty \sigma^{p-1}g_N(\sigma)\,{\rm d}\sigma 
+ \int_0^\infty \sigma^{p-3}
\int_0^\sigma tg_{N}(t)\,{\rm d}t
\,{\rm d}\sigma
\\
& \lesssim
\int_0^\infty \sigma^{p-1}g_N(\sigma)\,{\rm d}\sigma,
\end{align*}
and thus $\nabla u \in T^{p,2}$ with
$\|\nabla u\|_{T^{p,2}} \lesssim 
\|{\mathcal{N}}_\beta u\|_{L^p} \lesssim \|u\|_{X^p}$.

To finish the proof, we check  that 
$(s,x) \mapsto u(s^2,x)\chi^2(s,x) \in L^2(\varepsilon, 2R; H^1({\mathbb{R}}^n))$.
We begin by checking $\int_{\varepsilon}^{2R} \int_{{\mathbb{R}}^n} 
|u(s^2, x)\chi^2(s,x)|^2 \,{\rm d}s\,{\rm d}x<\infty$, with constants that depend on 
$\varepsilon, R$. Indeed, we may split $[\varepsilon, 2R]$ into a finite number 
of intervals $[\delta^2, \beta^2 \delta^2]$. For each of them, arguing as for 
$I_1$, we obtain a bound $\int_E |{\mathcal{N}}_\beta u(x)|^2\,{\rm d}x$.
By definition of $E$ and $p<2$, this does not exceed 
$\sigma^{2-p} \int_{{\mathbb{R}}^n} |{\mathcal{N}}_\beta u(x)|^p\,{\rm d}x<\infty$.  
Next, $\int_\varepsilon^{2R} \int_{{\mathbb{R}}^n} |\nabla (u(s^2, x)\chi^2(s,x))|^2 
\,{\rm d}s\,{\rm d}x$ is estimated similarly, using Proposition~\ref{prop:nrjloc} 
as for $I_2$.   
\end{proof}

\subsection{Controlling the square function by the maximal function for 
$p \in [2,\infty]$}

\begin{theorem}\label{thm:SbyN>2}
Let $u$ be a global weak solution of \eqref{eq1}. Let $p \in [2,\infty]$.
If $u \in X^p$, then $\nabla u \in T^{p,2}$ with
$
\|\nabla u\|_{T^{p,2}} \lesssim \|u\|_{X^p}$ and implicit constant independent of $u$. 
\end{theorem}

\begin{proof}
{\tt Step 1:} 
In this step we prove the result for $p=2$. Let $R,\varepsilon>0$ with 
$ R \ge \frac{1}{\sqrt{\varepsilon}}$. Let 
$\chi \in {\mathscr{C}}^{\infty}((0,\infty), [0,1])$ be supported in 
$[\varepsilon, \frac{1}{\varepsilon}]$, and such that, for some constant 
$C>0$,
$$
\textstyle{
\chi(t)=1 \quad \forall t \in \bigl[2\varepsilon, \frac{1}{2\varepsilon}\bigr],
\qquad
|\chi'(t)| \le \frac{C}{\varepsilon} \quad \forall t \in [\varepsilon, 2\varepsilon],
\qquad
|\chi'(t)| \le C\varepsilon \quad \forall 
t \in \bigl[\frac{1}{2\varepsilon}, \frac{1}{\varepsilon}\bigr].
}
$$
Let $\theta \in {\mathscr{C}}^\infty({\mathbb{R}}^n, [0,1])$ be supported in 
$B(0,2R)$ and such that
$$
\theta(x)=1 \quad \forall x \in B(0,R), \qquad
|\nabla \theta(x)| \leq \frac{C}{R} \quad \forall x \in B(0,2R)\setminus B(0,R).
$$
For $M>0$, define $\phi_M(t,x) = e^{t\Delta}(1\!{\rm l}_{B(0,M)})(x)$ 
for all $(t,x) \in (0,\infty)\times {\mathbb{R}}^n$. Remark that $\phi_M$ 
is smooth, that $\phi_M(t,x)$ increases as $M$ increases with 
$\phi_M(t,x) \xrightarrow[M \to \infty]{} 1$
for  all $(t,x) \in (0,\infty)\times {\mathbb{R}}^n$, and that
$$
\|\nabla \phi_M\|_{T^{\infty,2}} \sim \|1\!{\rm l}_{B(0,M)}\|_{BMO} \le 2, 
\qquad \|\phi_M\|_{L^{\infty}(L^\infty)} \le 1.
$$
Let $R_0\ge 4R$ and set $\phi=\phi_{R_0}$. We want to show 
$$
I : = \int_0^\infty \int_{{\mathbb{R}}^n}
|\nabla u(t,x)|^2 \phi(t,x)^2\theta(x)^2\chi(t)^2\,{\rm d}x\,{\rm d}t 
\lesssim  \|u\|_{X^2}^2= : J,
$$ 
independently of $u$, $R_0$, $R$ and $\varepsilon$. Indeed, if this is the case, 
then we can let $R_0\to \infty$ first by monotone convergence, which implies a 
control of 
$\int_{2\varepsilon}^{1/2\varepsilon} \int_{B(0,R)}
|\nabla u(t,x)|^2 \,{\rm d}x\,{\rm d}t$ and it suffices to let $R\to \infty$ and 
$\varepsilon\to 0$. 

Let $a=\varepsilon$, $b=1/\varepsilon$ and $\Omega=B(0,2R)$. Let 
$\psi=\phi \theta \chi$ (we forget the variables to keep the notation reasonable) 
and remark that $u\psi^2 \in L^2(a,b; H^1_0(\Omega))$. According to 
Remark~\ref{rem:local}, we can use this as a test function to obtain 
$$
\int_0^\infty \int_{{\mathbb{R}}^n} A(t,\cdot)\nabla u(t,x)\cdot 
\overline{\nabla (u\psi^2)(t,x)} \rangle \,{\rm d}t\,{\rm d}x =
- \int_0^\infty \langle \partial_tu(t,\cdot),u\psi^2(t,\cdot) \rangle \,{\rm d}t.  
$$
From ellipticity, 
\begin{align*}
I& \lesssim 
\Bigl|\int_0^\infty \int_{{\mathbb{R}}^n}
A(t,x)\nabla u(t,x)\cdot\overline{\nabla(u \psi^2)(t,x)}
\,{\rm d}x\,{\rm d}t\Bigr|
\\
&\quad +\Bigl|\int_0^\infty \int_{{\mathbb{R}}^n}
A(t,x)\nabla u(t,x)\cdot\nabla(\psi^2)(t,x)
\overline{u(t,x)}\,{\rm d}x\,{\rm d}t\Bigr|
\\
& \lesssim I_1+I_2+I_3,
\end{align*}
where
\begin{align*}
I_1 &= \Bigl| \int_0^\infty
\langle \partial_t u(t,\cdot), (u\psi^2)(t,\cdot) \rangle\,{\rm d}t\Bigr|,
\\
I_2 &=\Bigl|\int_0^\infty \int_{{\mathbb{R}}^n}
A(t,x)\nabla u(t,x)\cdot\phi(t,x)\nabla \phi(t,x)
\theta^2(x)\chi^2(t)\overline{u(t,x)}\,{\rm d}x\,{\rm d}t\Bigr|,
\\
I_3 &=  \Bigl|\int_0^\infty \int_{{\mathbb{R}}^n}
A(t,x)\nabla u(t,x)\cdot\phi^2(t,x)\theta(x)\nabla \theta(x)\chi^2(t)
\overline{u(t,x)}\,{\rm d}x\,{\rm d}t\Bigr|.
\end{align*}
To estimate $I_1$, we decompose further and obtain 
$I_1 \lesssim I_{1,1}+I_{1,2}+I_{1,3}$, where (forgetting to write the $t$ variable)
\begin{align*}
I_{1,1} &= \Bigl|\int_0^\infty
\langle \partial_t(u\theta \phi \chi),u\theta \phi \chi \rangle\,{\rm d}t\Bigr|,
\\
\quad I_{1,2} &= \Bigl| \int_0^\infty
\langle u \theta \phi \chi\cdot\theta \chi \partial_t\phi,u \rangle\,{\rm d}t\Bigr|,
\\
\quad I_{1,3} &= \Bigl| \int_0^\infty
\langle u\theta \phi \chi\cdot\theta \phi \chi',u\rangle dt\Bigr|.
\end{align*}
In the first integral, the inner product is the $H^{-1}(\Omega), H^1_0(\Omega)$ 
duality so that we can use Lions' theorem again and 
we have that $I_{1,1} =\frac{1}{2} \int_0^\infty
\partial_t \|u(t,.)\theta\phi(t,.)\chi(t)\|_2^2\, {\rm d}t = 0$.
In the other two integrals, the inner product can be rexpressed with the 
$L^2$ duality. To estimate $I_{1,2}$, remark first that, if  
$g_t$ denotes the standard heat kernel defined by 
$g_t(z) = (\pi t)^{-\frac{n}{2}}e^{-\frac{|z|^2}{4t}}$, we have for all $x \in B(0,2R)$ 
and for all $t>0$,
\begin{align*}
|\partial_t \phi(t,x)| &\lesssim \Bigl|\int_{{\mathbb{R}}^n}
\partial_t g_t(x-y)1\!{\rm l}_{B(0,R_0)}(y)\,{\rm d}y\Bigr|
\\
&
= \Bigl|\int_{{\mathbb{R}}^n}
\partial_t g_t(x-y)1\!{\rm l}_{B(0,R_{0})^c}(y)\,{\rm d}y\Bigr|
\\ 
& \le C
\int_{{\mathbb{R}}^n} t^{-n/2-1}e^{-\frac{R^2}{4t}} e^{-\frac{c|x-y|^2}{4t}}\,{\rm d}y
\lesssim R^{-2}.
\end{align*}
In this calculation, we used $|x-y|\ge R_0-2R\ge 2R$, so that 
$|x-y|^2\ge \frac{|x-y|^2}{4}+ R^2$. 
Now define $k_\varepsilon \in {\mathbb{N}}$ such that 
$2^{k_\varepsilon} \le \frac{1}{\varepsilon^2} < 2^{k_\varepsilon+1}$. 
We have the following
$$
I_{1,2} \lesssim R^{-2} \sum_{k=0} ^{k_\varepsilon}
\int_{2^k\varepsilon}^{2^{k+1}\varepsilon}
\int_{{\mathbb{R}}^n} |u(t,x)|^2\,{\rm d}x\,{\rm d}t
\lesssim R^{-2} 
\sum_{k=0}^{k_\varepsilon} 2^k\varepsilon \|u\|_{X^2}^2 
\lesssim  \|u\|_{X^2}^2,
$$
using $R^{-2}\varepsilon^{-1}\le 1$. Moreover 
$$
I_{1,3} \lesssim \fint_{\varepsilon}^{2\varepsilon}
\int_{{\mathbb{R}}^n} |u(t,x)|^2\,{\rm d}x\,{\rm d}t
+
\fint_{\frac{1}{2\varepsilon}}^{\frac{1}{\varepsilon}}
\int_{{\mathbb{R}}^n} |u(t,x)|^2\,{\rm d}x\,{\rm d}t
\lesssim \|u\|_{X^2}^2,
$$
since 
$|\chi'(t)| \le \frac{C}{\varepsilon}$ for all $t \in [\varepsilon, 2\varepsilon]$, and 
$|\chi'(t)| \le C\varepsilon$ for all 
$t \in \bigl[\frac{1}{2\varepsilon}, \frac{1}{\varepsilon}\bigr]$.

To estimate $I_2$ we use Cauchy-Schwarz inequality,  
Harnack inequalities for each 
$\partial_{x_j}\phi$, $j=1,\dots,n$, and Carleson inequality 
(see \cite[Proposition~3]{CMS85}) to obtain 
\begin{align*}
I_2 &\lesssim
\Bigl(\int_0^\infty \int_{{\mathbb{R}}^n}
|\nabla u(t,x)|^2\theta(x)^2\phi(t,x)^2\chi(t)^2\,{\rm d}x\,{\rm d}t\Bigr)^{\frac{1}{2}}
\Bigl(\int_0^\infty \int_{{\mathbb{R}}^n}
|u(t,x)|^2|\nabla \phi(t,x)|^2\,{\rm d}x\,{\rm d}t\Bigr)^{\frac{1}{2}}
\\
& = I^{\frac{1}{2}}
\Bigl(\int_0^\infty \int_{{\mathbb{R}}^n}
\Bigl(\fint_t^{2t} \fint_{B(x,\sqrt{t})} |u(s,y)|^2|\nabla \phi(s,y)|^2
\,{\rm d}y\,{\rm d}s\Bigr)
\,{\rm d}x\,{\rm d}t\Bigr)^{\frac{1}{2}}
\\
&\lesssim
I^{\frac{1}{2}}
\Bigl(\int_0^\infty \int_{{\mathbb{R}}^n}
\Bigl(\fint_t^{2t} \fint_{B(x,\sqrt{t})} |u(s,y)|^2\,{\rm d}y\,{\rm d}s\Bigr)
|\nabla \phi(t,x)|^2\,{\rm d}x\,{\rm d}t\Bigr)^{\frac{1}{2}}
\\
&\lesssim
I^{\frac{1}{2}}
\Bigl(\int_{{\mathbb{R}}^n} \sup_{(t,x) \in \Gamma_z}
\Bigl(\fint_t^{2t} \fint_{B(x,\sqrt{t})}|u(s,y)|^2\,{\rm d}y\,{\rm d}s\Bigr)
\,{\rm d}z\, \|\nabla \phi\|_{T^{\infty,2}}^2\Bigr)^{\frac{1}{2}}
\lesssim I^{\frac{1}{2}}J^{\frac{1}{2}}.
\end{align*}
For $I_3$, we have by the  Cauchy-Schwarz inequality and  
$R^{-2}\varepsilon^{-1}\le 1$, 
\begin{align*}
I_3 & \lesssim I^{\frac{1}{2}}
\Bigl(\int_0^\infty \int_{{\mathbb{R}}^n}
|u(t,x)|^2|\nabla \theta(x)|^2\phi(t,x)^2\chi(t)^2
\,{\rm d}x\,{\rm d}t\Bigr)^{\frac{1}{2}}
\lesssim \frac{I^{\frac{1}{2}}}{R}
\Bigl(\int_0^\infty \int_{{\mathbb{R}}^n}
|u(t,x)|^2\chi(t)^2\,{\rm d}x\,{\rm d}t\Bigr)^{\frac{1}{2}}
\\
&\lesssim  
\frac{I^{\frac{1}{2}}}{R}\Bigl(
\sum_{k=0}^{k_\varepsilon} 2^{k_\varepsilon}
\fint_{2^k\varepsilon}^{2^{k+1}\varepsilon}
\int_{{\mathbb{R}}^n} \fint_{B(x,\sqrt{2^k\varepsilon})}
|u(t,y)|^2\,{\rm d}y\,{\rm d}t\,{\rm d}x\Bigr)^{\frac{1}{2}}
\lesssim
\frac{I^{\frac{1}{2}}}{R}\Bigl(2^{k_\varepsilon}\varepsilon\Bigr)^{\frac{1}{2}}
J^{\frac{1}{2}} \le I^{\frac{1}{2}}J^{\frac{1}{2}}.
\end{align*}
Combining all the estimates, we have that 
$I\lesssim I^{\frac{1}{2}}J^{\frac{1}{2}}+J$. As $I<\infty$ by definition, we 
thus conclude that $I \lesssim J$ as desired.

\medskip

\noindent
{\tt Step 2:}  In this step we prove the result for $2< p\le \infty$ essentially 
by establishing a local version of Step~1. 
Fix $x_0 \in {\mathbb{R}}^n$, $R_0,R>0$ with $R_0\geq 4R$, 
$\varepsilon>0$ and consider 
$\chi\in {\mathscr{C}}^{\infty}({\mathbb{R}}, [0,1])$ supported in 
$[\varepsilon, 4R^{2}]$ such that, for some $C>0$,
$$
\textstyle{
\chi(t)=1 \quad \forall t \in [2\varepsilon,R^2],
\qquad
|\chi'(t)|\leq \frac{C}{\varepsilon} \quad \forall t \in [\varepsilon,2\varepsilon],\\
\qquad
|\chi'(t)|\leq \frac{C}{R^2} \quad \forall t \in [R^2,4R^2].}
$$
We define $\phi(t,x)=e^{t\Delta}(1_{B(x_0,R_0)})(x)$ for all
$(t,x) \in [0,\infty)\times {\mathbb{R}}^n$. We also let 
$\theta \in {\mathscr{C}}^{\infty}({\mathbb{R}}^n, [0,1])$ 
be supported in $B(x_0,2R)$ and such that
$$
\textstyle{
\theta(x)=1 \quad \forall x \in B(x_0,R),\qquad
|\nabla \theta(x)|\le \frac{C}{R} \quad \forall x \notin B(x_0,R).}
$$
Defining 
\begin{align*}
I &= \int_0^\infty \int_{{\mathbb{R}}^n}
|\nabla u(t,x)|^2\phi(t,x)^2\theta(x)^2\chi(t)^2\,{\rm d}t\,{\rm d}x,
\\
J &= \bigl\|(t,x) \mapsto 1\!{\rm l}_{B(x_0,2R)}(x)1\!{\rm l}_{(0,4R^2)}(t)
u(t,x)\bigr\|_{X^2}^2,
\end{align*}
we only have to show that $I\lesssim J$ with implicit constants independent of 
$u$, $x_0$, $\varepsilon$, $R$, $R_0$. Indeed, if this holds, then, taking the 
limit as $R_0\to \infty$ and then as $\varepsilon \to 0$,
\begin{align*}
\int_0^{R^2} \fint_{B(x_0,R)} &
|\nabla u(t,x)|^2\,{\rm d}x\,{\rm d}t
\\
& \lesssim R^{-n} \int_{{\mathbb{R}}^n}
\Bigl(\sup_{\delta>0} \fint_\delta^{2\delta}
\fint_{B(x,\sqrt{\delta})} |u(t,y)|^2 1\!{\rm l}_{(0,4R^2)}(t)
1\!{\rm l}_{B(x_0,2R)}(y)\,{\rm d}y\,{\rm d}t\Bigr)^2\,{\rm d}x
\\
& \lesssim R^{-n} \int_{B(x_0,4R)}|\tilde Nu(x)|^2\,{\rm d}x
\  \lesssim \inf_{y \in B(x_0, 4R)}M((\tilde Nu)^2)(y)
\end{align*}
where $M$ is the Hardy-Littlewood operator on ${\mathbb{R}}^n$. 

Define $C(F)(y)=\sup_{B\ni y}(\int_B\int_0^{r^2} |F(t,x)|^2 
\,{\rm d}t\,{\rm d}x)^{\frac 1 2}$. Thus if $y\in {\mathbb{R}}^n$ by taking the 
supremum over all  $B=B(x_0, R)\ni y$, we have shown 
$$
C(|\nabla u|)(y) \lesssim \bigl[M((\tilde Nu)^2)\bigr]^{\frac{1}{2}}(y).
$$
As $p>2$, using the parabolic version of  \cite[Theorem~3, (2)]{CMS85} 
and the maximal theorem, we obtain the conclusion. 

We proceed as in Step~1, estimating
$I\lesssim I_{1,1}+I_{1,2}+I_{1,3}+I_2+I_3$, where
\begin{align*}
I_{1,1} &= \Bigl| \int_0^\infty
\langle u, \partial_t(u\theta \phi \chi)\theta \phi \chi \rangle
\,{\rm d}t\Bigr|=0,
\\ 
I_{1,2} &= \Bigl| \int_0^\infty
\langle u, u\theta \phi \chi\cdot \theta \chi \partial_t\phi \rangle \,{\rm d}t\Bigr|,
\\
I_{1,3} &= \Bigl| \int_0^\infty
\langle u, u\theta \phi \chi\cdot\theta \phi \chi' \rangle\,{\rm d}t\Bigr|,\\
I_2 &= \Bigl|\int_0^\infty \int_{{\mathbb{R}}^n}
A(t,x)\nabla u(t,x)\cdot\phi(t,x)\nabla \phi(t,x)\theta^2(x)\chi^2(t)
\overline{u(t,x)}\,{\rm d}x\,{\rm d}t\Bigr|,
\\
I_3 &= \Bigl|\int_0^\infty \int_{{\mathbb{R}}^n}
A(t,x)\nabla u(t,x)\cdot\phi^2(t,x)\theta(x)\nabla \theta(x)\chi^2(t)
\overline{u(t,x)}\,{\rm d}x\,{\rm d}t\Bigr|.
\end{align*}
Using $|\partial_t\phi(t,x)| \lesssim R_0^{-1}$, we have that
$$
I_{1,2} \lesssim \int_0^{R^2}
\int_{B(x_0,2R)} |u(t,x)|^2|\partial_t\phi(t,x)|\,{\rm d}x\,{\rm d}t
\lesssim \frac{R^2}{R_0^2} \int_{{\mathbb{R}}^n}
\tilde N(1\!{\rm l}_{B(x_0,2R)}u)(x)^2\,{\rm d}x \lesssim J,
$$
and using the properties of $\chi$ and $\theta$, we also obtain
$$
I_{1,3} \lesssim \fint_\varepsilon^{2\varepsilon}
\int_{{\mathbb{R}}^n} |1\!{\rm l}_{B(x_0,2R)}(x) u(t,x)|^2\,{\rm d}x\,{\rm d}t
+ \fint_{R^2}^{4R^2} \int_{{\mathbb{R}}^n} 
|1\!{\rm l}_{B(x_0,2R)}(x)u(t,x)|^2\,{\rm d}x\,{\rm d}t\lesssim J.
$$
Next, as in Step~1, we also have that $I_2 \lesssim I^{\frac{1}{2}}J^{\frac{1}{2}}$.
Finally, by Cauchy-Schwarz inequality 
\begin{align*}
I_3 \lesssim& \frac{I^{\frac{1}{2}}}{R}
\Bigl(\int_{{\mathbb{R}}^n} \int_\varepsilon^{4R^{2}} 
|1\!{\rm l}_{B(x_0,2R)}(x)u(t,x)|^2\,{\rm d}x\,{\rm d}t\Bigr)^{\frac{1}{2}}
\\
\lesssim&
\frac{I^{\frac{1}{2}}}{R}\Bigl(\int_{{\mathbb{R}}^n} 
\sum_{k=-2}^\infty 2^{-k}R^2
\fint_{2^{-k-1}R^2}^{2^{-k}R^2}
\fint_{B(x,R2^{-\frac{k}{2}})} |1\!{\rm l}_{B(x_0,2R)}(x)u(t,x)|^2
\,{\rm d}x\,{\rm d}t\,{\rm d}y \Bigr)^{\frac{1}{2}}
\lesssim I^{\frac{1}{2}}J^{\frac{1}{2}}.
\end{align*}
This concludes Step~2.
\end{proof}

\subsection{Consequences}

We have obtained comparison results only for solutions with $L^2$ data. 
We can remove this constraint as follows.

\begin{corollary}
Let $1<p<\infty$. Assume that, for all $f \in L^p({\mathbb{R}}^n)$, the problem
$$
\partial_tu={\rm div}\,A\nabla u,
\quad
u\in X^p,
\quad 
\lim_{t \to 0} u(t,\cdot)=f \ in \  L^2_{\rm loc},
$$
is well-posed. Then the solution satisfies 
$$
\|\tilde N (u)\|_{L^p}\sim \|\nabla u \|_{T^{p,2}} \sim \|f\|_{L^p}.   
$$
\end{corollary}  

\begin{remark}
We have well-posedness in $X^p$ for $p \geq 2$ by 
Corollaries~\ref{thm:uniqueXp} and~\ref{cor:ntmax2}. For $p<2$, we have 
well-posedness in $X^p$ under the assumptions of either 
Corollary~\ref{cor:kioloa}, or Theorem~\ref{thm:BV}, or Theorem~\ref{thm:pert}.
\end{remark}

\begin{proof} 
Let $u \in X^p$ be a global weak solution of \eqref{eq1}.
By well-posedness in $X^p$, we know that $u(t,\cdot)=\Gamma(t,0)f$ for a 
unique $f\in L^p({\mathbb{R}}^n)$, for all $t>0$. Moreover 
$\|\tilde N (u)\|_p \sim \|f\|_p$. Consider an approximation 
$f_k\in L^p({\mathbb{R}}^n) \cap L^2({\mathbb{R}}^n)$ converging to $f$, and 
let $u_k$ be the corresponding solution. We have that 
$\|\tilde N (u_k)\|_p\sim \|\nabla u_k\|_{T^{p,2}}$ combining 
Proposition~\ref{prop:NbyS} and Theorems~\ref{thm:SbyN<2} 
and~\ref{thm:SbyN>2}. The conclusion follows from taking the limit as $k\to \infty$. 
\end{proof}
 
We wish to make a connection with old ideas in the topic such as those 
of Nash, as explained in \cite{FaSt}. Assume $A$ has real coefficients. 
Then for $h\in {\mathscr{C}}^\infty_c({\mathbb{R}}^n)$, $h\ge 0$, 
$u(t)=\Gamma(t,0)h$, and $ 2\le p <\infty$, one has (if the coefficients are 
smooth to ease justifications)
$$
\frac{{\rm d}}{{\rm d}t} \|u(t)\|_{L^p}^p=  
p(p-1)\int_{{\mathbb{R}}^n} u(t,x)^{p-2} A(t,x)\nabla u(t,x) \cdot \nabla u(t,x) 
\,{\rm d}x
$$
and in particular 
\begin{equation}
\label{eq:lp}
\|h\|_{L^p}^p- \|u(t)\|_{L^p}^p = 
p(p-1) \int_0^t\int_{{\mathbb{R}}^n} u(s,x)^{p-2} A(s,x)\nabla u(s,x)\cdot\nabla u(s,x) 
\,{\rm d}x\,{\rm d}s.
\end{equation}
This yields the finiteness of the integral when $t\to \infty$, which is equal to 
$\|h\|_{L^p}^p$ as $\|u(t)\|_{L^p} \to 0$. Note that for $p=2$, we recover the 
energy equality
$$
\|h\|_{L^2}^2 = 2 \int_0^\infty\int_{{\mathbb{R}}^n} A(s,x)\nabla u(s,x) 
\cdot \nabla u(s,x) \,{\rm d}x\,{\rm d}s
$$
which motivated our approach to the energy space. 
Our observation is that for $p>2$ and arbitrary global weak solutions with 
$\|u^*\|_{L^p}<\infty$ (which, as we have shown, are of the form 
$u(t)=\Gamma(t,0)h$ for a unique $h\in L^p({\mathbb{R}}^n)$) we have
$$
I:=\int_0^\infty\int_{{\mathbb{R}}^n} |u(s,x)|^{p-2} A(s,x)\nabla u(s,x) 
\cdot \nabla u(s,x)\,{\rm d}x \,{\rm d}s \lesssim \|u^*\|_{L^p}^p.
$$
Indeed,  using the boundedness of $A$, 
\begin{align*}
\int_0^\infty\int_{{\mathbb{R}}^n} |u(s,x)|^{p-2} A(s,x)\nabla u(s,x) 
\cdot \nabla u(s,x) \,{\rm d}x\,{\rm d}s 
&\lesssim \int_{{\mathbb{R}}^n} C(|\nabla u |)(x)^2 u^*(x)^{p-2}\,{\rm d}x 
\\
&\le \|C(|\nabla u |)\|_{L^p}^2\|u^*\|_{L^p}^{p-2}
\end{align*}
by H\"older's inequality in the last line and, in the first, \cite[Proposition~3]{CMS85} 
adapted to parabolic scaling, with $C(f)(x)$ is the supremum of 
$\big(\int_0^{r^2}\fint_B f^2(t,y)\,{\rm d}y\,{\rm d}t\big)^{\frac{1}{2}}$ taken 
over all balls $B$, $r$ being the radius, that contain $x$. Next, by 
\cite[Theorem~3, (b)]{CMS85}, we have 
$\|C(|\nabla u |)\|_{L^p}\lesssim \|\nabla u\|_{T^{p,2}}$. Finally we conclude 
using our Theorem~\ref{thm:SbyN>2}.

It is not clear at all that integrals of the form $I$ can play a role when considering 
complex equations. In particular, \eqref{eq:lp} does not hold in this case.   
However, (modified) maximal functions and Lusin area functionals remain valid 
tools, as we have just demonstrated.  

\section{Fatou type results}
\label{sec8}
Our well-posedness results imply convergence in  $L^2_{loc}(\mathbb{R}^n)$ or in  $L^p(\mathbb{R}^n)$ sense for some $p$ to the initial value.  Here, we address almost everywhere convergence issues. As weak solutions 
may not be locally bounded, we replace the pointwise (parabolic) non-tangential  
convergence by the convergence of (parabolic) Whitney averages. 

\begin{theorem} 
\label{thm:fatou81}
Let $A$ be as in \eqref{ell}. Let $f\in L^2({\mathbb{R}}^n)$ and 
$u(t,\cdot)=\Gamma(t,0)f$. Then for almost every $x\in{\mathbb{R}}^n$, 
we have convergence of the Whitney averages
$$
\lim_{\delta \to 0}  \fint_{\frac{\delta}{2}}^\delta
\fint_{B(x,\sqrt{\delta})}|u(t,y)-f(x)|^2\,{\rm d}y\,{\rm d}t =0,
$$
as well as of the slice averages
$$
\lim_{\delta \to 0} 
\fint_{B(x,\sqrt{\delta})}|u(\delta ,y)-f(x)|^2\,{\rm d}y =0.
$$
In particular, for almost every $x\in {\mathbb{R}}^n$, 
$$
\lim_{\delta \to 0}  \fint_{\frac{\delta}{2}}^\delta
\fint_{B(x,\sqrt{\delta})}u(t,y)\,{\rm d}y\,{\rm d}t = \lim_{\delta \to 0} 
\fint_{B(x,\sqrt{\delta})}u(\delta ,y)\,{\rm d}y = f(x)
$$
\end{theorem}

\begin{proof}
Considering $f(x)$ as a constant, the convergence of the slice averages 
follows from the local estimates in Proposition~\ref{prop:nrjloc} and the 
convergence of the Whitney averages. 

To show the convergence of the Whitney averages, we use again the 
Duhamel formula \eqref{Duhamel} with $L=-\Delta$, which reads, for almost 
every $(t,y) \in {\mathbb{R}}^{n+1}_+$,  
$$
u(t,y) = e^{t\Delta}f(y) + {\mathcal{R}}_{-\Delta}((A-I)\nabla u)(t,y)
= e^{t\Delta}f(y)+ v(t,y).
$$
Recall that $v(t,y)$ only involves $(A-I)\nabla u$ at times between $0$ and $t$. 
Thus, the proof of the boundedness from $T^{2,2}=L^2(L^2)$ to $X^2$ from 
Proposition~\ref{prop:subMR} shows that 
$$
\int_{{\mathbb{R}}^n}\sup_{\delta \le \varepsilon} \fint_{\frac{\delta}{2}}^\delta
\fint_{B(x,\sqrt{\delta})}|v(t,y)|^2\,{\rm d}y\,{\rm d}t \,{\rm d}x  
\lesssim \|(A-I)\nabla u\|^2_{L^2(0,\varepsilon;L^2)}. 
$$
The right-hand side converges to 0 with $\varepsilon\to 0$ as 
$\|\nabla u \|_{L^2(L^2)} \sim \|f\|_{L^2}$. Thus the left-hand side converges to 0. 
As the integrand is non negative and non decreasing  as a function 
of $\varepsilon$, it follows that it converges to 0 almost everywhere, that is
$$
\fint_{\frac{\delta}{2}}^\delta
\fint_{B(x,\sqrt{\delta})}|v(t,y)|^2\,{\rm d}y\,{\rm d}t \xrightarrow[\delta\to0]{} 0
\qquad \mbox{almost everywhere.}
$$
Finally, for almost every $x \in {\mathbb{R}}^n$,   
$$
\fint_{\frac{\delta}{2}}^\delta
\fint_{B(x,\sqrt{\delta})}|e^{t\Delta}f(y)- f(x)|^2\,{\rm d}y\,{\rm d}t 
\xrightarrow[\delta \to 0]{} 0
$$
by the known results concerning pointwise non-tangential almost everywhere 
convergence of the solutions of the heat equation. This completes the proof.
\end{proof}

\begin{corollary} 
Let $1< p<\infty$. Consider any $A$ as in \eqref{ell} for which \eqref{eq1} is 
well-posed on $X^p$ with initial space $L^p$ (this holds for $p\ge 2$ and 
under conditions when $p<2$). Let $u$ be a global weak solution in $X^p$. Then 
$$
\fint_{\frac{\delta}{2}}^\delta
\fint_{B(x,\sqrt{\delta})}u(t,y)\,{\rm d}y\,{\rm d}t \quad \mbox{and} \quad 
\fint_{B(x,\sqrt{\delta})}u(\delta ,y)\,{\rm d}y 
$$
converge for almost every $x\in {\mathbb{R}}^n$ to the initial value as $\delta \to 0$. 
\end{corollary}

\begin{proof}  
Under our assumption, we know that $u(t,\cdot)=\Gamma(t,0)f$ for a unique 
$f\in L^p$ for all $t>0$. We invoke  the usual argument involving 1) control of the 
maximal function $\tilde N (u)$ and the one for the slice averages in $L^p$, 
and 2) existence of limits almost everywhere for a dense class of $f$, here 
$L^p\cap L^2$, from the previous theorem. We skip details.
\end{proof}

\begin{remark} 
Consider $A$ with real coefficients. This results holds for any $1<p<\infty$ and 
can be proven by usual arguments, using the pointwise Gaussian upper estimate 
and the conservation property. For $p=\infty$, one can get convergence under 
uniform continuity of $f$. For $p=1$, we have not attempted to describe the space
$\{u \in L^2_{\rm loc}({\mathbb{R}}^{n+1}_+) \;;\; \|u\|_{X^1}<\infty\}$.  
This would define a Hardy space associated with a non-autonomous equation.  
\end{remark}

\section{$L^1$ theory for propagators with kernel bounds}
\label{sec9}

In this  section, we assume that the propagators $\Gamma(t,s)$, for $t> s\geq 0$, 
have  kernel bounds, and that their adjoints have some time regularity. More 
precisely, we assume the following:   
\begin{align}
\label{p1}
&|k(t,s,x,y)| \le C(t-s)^{-\frac{n}{2}}e^{-c\frac{|x-y|^{2}}{t-s}},\\
\label{p3}
&\forall\, h \in {\mathscr{C}}_c({\mathbb{R}}^n) \quad
\|\Gamma(t_0,s)^*h-\Gamma(t_0,s_0)^*h\|_{L^\infty} 
\xrightarrow[s \to s_0]{} 0,
\end{align}
for some $C,c>0$, all $t_0>s_0\ge 0$, and almost all $x,y \in{\mathbb{R}}^n$.

See Section~\ref{sec:kernel} for a discussion of the first condition. The second 
condition can be checked for real equations as a consequence of Nash's 
regularity theorem \cite{Na}. 

\begin{lemma}
\label{lem:repro}
Assume \eqref{p1}. Let $t>s>r\ge 0$.
Then
$$
\int_{{\mathbb{R}}^n} k(t,s,x,z)k(s,r,z,y)\,{\rm d}z = k(t,r,x,y),
$$
for almost every $x,y \in {\mathbb{R}}^n$. Moreover, 
\begin{equation}
\label{eq:integral1}
\int_{{\mathbb{R}}^n} k(t,r, x,y)\, dx =1
\end{equation}
for almost every $y\in{\mathbb{R}}^n$. Finally, for any $f\in L^1({\mathbb{R}}^n)$, 
$\tau \mapsto \Gamma(\tau,s)f \in {\mathscr{C}}_0([s,\infty), L^1)$.
\end{lemma}

\begin{proof}
Let $x,y \in {\mathbb{R}}^n$ and define 
$I(x,y) = \int_{{\mathbb{R}}^n} k(t,s,x,z)k(s,r,z,y)\,{\rm d}z$. Note that the 
integral converges thanks to \eqref{p1}. Next, for 
$f \in{\mathscr{D}}({\mathbb{R}}^n)$, we have the following for almost all 
$x\in {\mathbb{R}}^n$
\begin{align*}
\int_{{\mathbb{R}}^n} I(x,y)f(y)\,{\rm d}y &= 
\int_{{\mathbb{R}}^n}k(t,s,x,z)\Big(\int_{{\mathbb{R}}^n} k(s,r,z,y)f(y)
\,{\rm d}y\Big)\,{\rm d}z\\
&= 
\int_{{\mathbb{R}}^n}k(t,s,x,z)\Gamma(s,r)f(z)\,{\rm d}z
= 
\Gamma(t,s)\Gamma(s,r)f(x)
\\&=\Gamma(t,s)f(x) = \int_{{\mathbb{R}}^n}k(t,s,x,y)f(y)\,{\rm d}y.
\end{align*}
Using \eqref{p1}, we have that the other equality is equivalent to 
$\Gamma(t,r)^*1\!{\rm l}=1\!{\rm l}$ almost everywhere. This is proved in the 
$L^2_{\rm loc}$ sense in Corollary~\ref{cor:cons}. 

Finally, the strong continuity of $\tau \mapsto \Gamma(\tau,s)$ on 
$L^1({\mathbb{R}}^n)$ is proven as follows. As the operators are uniformly 
bounded on $L^1({\mathbb{R}}^n)$, it suffices to check continuity for functions 
in a dense class.
Let $f \in L^\infty({\mathbb{R}}^n)$ be compactly supported, and $M>0$ be 
such that the support of $f$ is contained $B(0,M)$. 
Then, using \eqref{p1}, $\|\Gamma(t,s)f- \Gamma(t',s)f\|_{L^1(B(0,M)^c)}\to 0$ 
when $M\to \infty$ uniformly for $t,t'$ in any bounded set of $[s,\infty)$. So fix 
$t\ge s$ and take $t'\ge s$ with $|t-t'|\le 1$. Let $\varepsilon>0$ and fix 
$M>0$ so that  $\|\Gamma(t,s)f- \Gamma(t',s)f\|_{L^1(B(0,M)^c)}\le \varepsilon$ 
for all such $t,t'$ . 
Since   $\|\Gamma(t,s)f- \Gamma(t',s)f\|_{L^1(B(0,M))} \le 
M^{n/2} \|\Gamma(t,s)f- \Gamma(t',s)f\|_{L^2(B(0,M))}$, we can use the strong 
continuity of $\tau \mapsto \Gamma(\tau,s)$ on $L^2({\mathbb{R}}^n)$ to conclude.  
The proof for the limit at $\infty$ is similar.
\end{proof}

\begin{theorem}
Assume \eqref{p1} and \eqref{p3}.
For $u \in {\mathscr{D}}'$, the following assertions are equivalent.
\begin{align}
\label{eq81}
&\exists !\,u_0\in L^1({\mathbb{R}}^n)\mbox{ such that }u(t,\cdot)=\Gamma(t,0)u_0
\mbox{ in }L^1({\mathbb{R}}^n)\mbox{ for all }t>0;
\\
\label{eq82}&u \mbox{ is a global weak solution of \eqref{eq1} in } 
L^\infty_{wcs}(L^1),
\end{align}
where  $u\in L^\infty_{wcs}(L^1)$ means $u\in L^\infty(L^1)$ and there exists a 
weakly convergent sequence $(u(t_j,.))_{j \in {\mathbb{N}}}$ in 
$L^1({\mathbb{R}}^n)$, where $(t_j)_{j \in {\mathbb{N}}}$ is a sequence of 
positive reals decreasing to $0$. In this case, $u\in{\mathscr{C}}_0([0,\infty),L^1)$.   
\end{theorem}

\begin{proof}
To prove that \eqref{eq81} implies \eqref{eq82}, let $f \in L^1({\mathbb{R}}^n)$ 
and define $u:(t,x)\mapsto \Gamma(t,0)f(x)$ by the intergal formula.
By \eqref{p1} and Lemma~\ref{lem:repro}, we have that 
$u\in{\mathscr{C}}_0([0,\infty),L^1)$. It remains to see it is a global weak solution 
of \eqref{eq1}. For all $a>0$ and $t\in(a,\infty)$, we have, thanks to \eqref{p1}, 
that there exists $c>0$ such that
$$
\|\Gamma(\frac{a}{2},0)f\|_{L^2} 
\lesssim \|e^{ca\Delta}|f|\|_{L^2} 
\lesssim \|f\|_{L^1}.
$$
By Lemma~\ref{lem:repro}, and density of $L^1\cap L^2$ in $L^1$, 
$\Gamma(t,0)f=\Gamma(t,\frac{a}{2})\Gamma(\frac{a}{2},0)f$ for all 
$t\in (a,\infty)$. This shows that $u$ is a global weak solution of \eqref{eq1}.

We now turn to the other direction, and assume \eqref{eq82}.
Let $(t_j)_{j\in{\mathbb{N}}}$ be the sequence of positive reals decreasing to 
$0$ such that $(u(t_j,\cdot ))_{j\in{\mathbb{N}}}$ converges weakly in 
$L^1({\mathbb{R}}^n)$. Call $f$ its limit. By Proposition~\ref{prop:L2poids}
and Theorem~\ref{thm:THE_THEOREM}, 
for all $j \in{\mathbb{N}}$ with $t>t_j$, and $h \in{\mathscr{C}}_c({\mathbb{R}}^n)$, 
we have that
$$
\int_{{\mathbb{R}}^n} u(t_j,x)\overline{\Gamma(t,t_j)^*h(x)}\,{\rm d}x
=\int_{{\mathbb{R}}^n}u(t,x)\overline{h(x)}\,{\rm d}x.
$$
Using \eqref{p3} and also that $(\|u(t_j,\cdot)\|_1)_{j\in{\mathbb{N}}}$ is a 
bounded sequence, the left hand side converges to 
$\int_{{\mathbb{R}}^n} f(x)\overline{\Gamma(t,0)^*h(x)}\,{\rm d}x$. 
This implies $u(t,\cdot) = \Gamma(t,0)f$. The uniqueness of $u_0=f$ follows 
by continuity at $t=0$ in $L^1({\mathbb{R}}^n)$. 
\end{proof}

We then consider solutions only in $L^\infty(L^1)$ and show that such solutions 
arise from considering Radon measures as initial data. To do so, we need to 
impose a further condition on the kernel, which is satisfied in the case of real 
coefficients (see \cite[Theorem~9]{Ar68}), again as a consequence of Nash's 
regularity theorem \cite{Na}.  
\begin{equation}
\label{p4}
\mbox{For all } t > s\ge 0 \mbox { and almost all } x \in{\mathbb{R}}^n, 
\quad y \mapsto k(t,s,x,y) \mbox{ is continuous on } {\mathbb{R}}^n.
\end{equation}
Under \eqref{p1} and \eqref{p4}, then for all $ t > 0$ the integral 
$\int_{{\mathbb{R}}^n} k(t,0,x,y)\,{\rm d}\mu(y)$ makes sense in the duality of 
${\mathscr{C}}_0({\mathbb{R}}^n)$ with the space ${\mathscr{M}}({\mathbb{R}}^n)$
of Radon measures for almost all $x \in {\mathbb{R}}^n$, and belongs to 
$L^1({\mathbb{R}}^n)$ as a function of $x$. We call this function 
$\Gamma(t,0)\mu$. 

\begin{theorem}
Assume \eqref{p1}, \eqref{p3}  and \eqref{p4}. For $u \in {\mathscr{D}}'$,
the following assertions are equivalent.
\begin{align}
\label{eq83}
&\exists !\,\mu \in {\mathscr{M}}({\mathbb{R}}^n) \mbox{ such that }  
u(t,\cdot) = \Gamma(t,0)\mu \mbox{ for all } t>0;
\\
\label{eq84}&u \mbox{ is a global weak solution of \eqref{eq1} in } L^\infty(L^{1}).
\end{align}
In this case, $u\in {\mathscr{C}}_0((0,\infty), L^1)$ and is weakly-star 
convergent to $\mu$ as $t\to 0$.
\end{theorem}

\begin{proof}
We first show that \eqref{eq83} implies \eqref{eq84}.
Let $\mu$ be a Radon measure and  $u(t,\cdot) = \Gamma(t,0)\mu$ for all $ t>0$. 
By \eqref{p1} and Fubini's theorem, we have that $u \in L^\infty(L^1)$.
Moreover, for all $a>0$ and $t\in(a,\infty)$, we have that 
$u\bigl(\frac{a}{2},\cdot\bigr)\in L^1({\mathbb{R}}^n)\cap L^\infty({\mathbb{R}}^n)$, 
and, in particular $u\bigl(\frac{a}{2},\cdot\bigr) \in L^2({\mathbb{R}}^n)$.
By the reproducing formula for the kernels in Lemma~\ref{lem:repro} and 
Fubini's theorem, 
$u(t,\cdot)=\Gamma\bigl(t,\frac{a}{2}\bigr)u\bigl(\frac{a}{2},\cdot\bigr)$. 
This shows that $u$ is a weak solution of \eqref{eq1} on $\bigl(\frac{a}{2},\infty\bigr)$ 
for any $a>0$. Hence it is a global weak solution.

Turning to the other direction, we assume \eqref{eq84}.
Let $t_0<t$ and $(t_j)_{j\in{\mathbb{N}}}$ be a sequence of positive reals 
decreasing to $0$ and such that 
$(\int_{{\mathbb{R}}^n }u(t_j,x) \overline {g(x)}\,{\rm d}x)_{j\in {\mathbb{N}}}$ 
converges for all $g \in {\mathscr{C}}_0({\mathbb{R}}^n)$ by 
Banach-Alaoglu's theorem.
Let $\mu \in {\mathscr{M}}({\mathbb{R}}^n)$ be the weak$^*$ limit of
$(u(t_j,\cdot))_{j\in{\mathbb{N}}}$. Let $g \in {\mathscr{C}}_c({\mathbb{R}}^n)$. 
Using \eqref{p1} and \eqref{p4}, we see that 
$\Gamma(t,t_j)^*g \in{\mathscr{C}}_0({\mathbb{R}}^n)$, and by \eqref{p3}, 
that
$$
\|\Gamma(t,t_j)^*g-\Gamma(t,0)^*g\|_{L^\infty}\xrightarrow[j \to \infty]{} 0.
$$
By Proposition~\ref{prop:L2poids}
and Theorem~\ref{thm:THE_THEOREM}, we have that
$$
\int_{{\mathbb{R}}^n} u(t_j,x)\overline{\Gamma(t,t_j)^*g(x)}\,{\rm d}x
=\int_{{\mathbb{R}}^n}u(t,x)\overline{g(x)}\,{\rm d}x
$$
so that  taking the limit as $j\to \infty$, we obtain
$$
\int_{{\mathbb{R}}^n}u(t,x)\overline{g(x)}\,{\rm d}x
=\int_{{\mathbb{R}}^n}\overline{\Gamma(t,0)^{*}g(y)}\,{\rm d}\mu(y)
=\int_{{\mathbb{R}}^n} \int_{{\mathbb{R}}^n} k(t,0,x,y)\overline{g(x)}
\,{\rm d}\mu(y)\,{\rm d}x,
$$
and thus
$$
u(t,x) = \int_{{\mathbb{R}}^n} k(t,0,x,y)\,{\rm d}\mu(y),
$$
for all $t>0$ and almost all $x\in {\mathbb{R}}^n$.

\noindent
To show that $\mu$ is unique, let $g \in {\mathscr{C}}_c({\mathbb{R}}^n)$, 
and $\varepsilon>0$. By uniform continuity, pick 
$\delta>0$ such that $|g(x)-g(y)|\le \varepsilon$ for all $x,y$ such that 
$|x-y|\le \delta$. Using \eqref{p1}, we have that
$$
\Big|\int_{{\mathbb{R}}^n} \int_{B(y,2\delta)} k(t,0,x,y)(g(x)-g(y))
{\rm d}x\,{\rm d}\mu(y)\Big| \lesssim \varepsilon \|\mu\|_{\mathscr{M}}.
$$
Moreover,  we have that
$$
\Big|\int_{{\mathbb{R}}^n} \int _{B(y,2\delta)^c} k(t,0,x,y)(g(x)-g(y))
\,{\rm d}x\,{\rm d}\mu(y)\Big| \lesssim  \|\mu\|_{\mathscr{M}} \|g\|_{L^\infty}
e^{-c\frac{\delta^{2}}{t}},
$$
for some constant $c>0$. Therefore, 
$$
\lim_{t \to 0}\int_{{\mathbb{R}}^n} \int_{{\mathbb{R}}^n} 
k(t,0,x,y)(g(x)-g(y))\,{\rm d}x\,{\rm d}\mu(y) = 0.
$$ 
By Lemma~\ref{lem:repro}, we have that
$$
\int_{{\mathbb{R}}^n} \int_{{\mathbb{R}}^n} k(t,0,x,y)g(x)
\,{\rm d}x\,{\rm d}\mu(y) - \int_{{\mathbb{R}}^n} g(y)\,{\rm d}\mu(y) =
\int_{{\mathbb{R}}^n} \int_{{\mathbb{R}}^n} k(t,0,x,y)(g(x)-g(y))
\,{\rm d}x\,{\rm d}\mu(y),
$$
and thus $\displaystyle{\int_{{\mathbb{R}}^n}g(x)\,{\rm d}\mu(x) = 
\lim_{t\to 0}\int_{{\mathbb{R}}^n }u(t,x) g(x)\,{\rm d}x}$, which proves uniqueness 
of $\mu$.

Finally, the continuity on $(0,\infty)$ and the limit at $\infty$ follow directly from 
Lemma~\ref{lem:repro} since $u(t,\cdot)= \Gamma(t,s)u(s,\cdot)$ for all 
$t\ge s>0$ and $u(s,\cdot) \in L^1({\mathbb{R}}^n)$. 
\end{proof}

\section{Local results}
\label{sec10}

We have been interested in global results with scale invariant norms in the 
interior. As our estimates depend separately on 1) interior representation and 
2) taking limits as $t$ tends to $0$, we can formulate 
well-posedness for larger classes as follows. Let $X$ be one of the spaces 
where we can prove well-posedness for $Y$ data. For $T>0$, let $X_T$ 
be a local version of $X$, obtained by truncating functions by $0$, i.e.
$\|u\|_{X_T}= \|(t,x) \mapsto u(t,x) 1\!{\rm l}_{(0,T)}(t)\|_X$. 
Let ${\mathcal{X}}=\cap_{T>0} X_T$. Note that functions in ${\mathcal{X}}$ 
can have their $X_T$ norms growing arbitrarily fast as $T\to \infty$. Functions 
with bounded $X_T$ norms belong to $X$.  

Consider the Cauchy problem for \eqref{eq1} on ${\mathbb{R}}^{n+1}_+$ 
with $u\in {\mathcal{X}}$ and $u_0\in Y$. In each of the following cases, one 
obtains uniqueness. Since we can construct solutions in $X$, we get a posteriori 
control at $T\to \infty$ and the solution belongs to $X$. In other words, 
arbitrary a priori control on the norms for large times implies a posteriori bounded 
control. The interested reader can write the precise statements from our 
explanations. 

For $Y=L^2({\mathbb{R}}^n)$, one can take $X=L^\infty(L^2)$ or $X=X^2$. 
In the case $X=L^2(\dot H^1)$, one  takes $X_T= L^2(0,T; H^1)$ instead of 
$L^2(0,T;\dot H^1)$.   

For $Y=L^p({\mathbb{R}}^n)$, $2<p\le \infty$, one can take $X=X^p$ or the 
space in Proposition~\ref{prop:rh} when $p<\tilde q$.  

For $Y=L^p({\mathbb{R}}^n)$, $1\le p<\infty$ and uniform $L^q$ bounds on 
the propagators or estimates on its kernel (see above for the relation between 
$q$ and $p$), one can take $X=L^\infty(L^p)$.  

For $Y=L^p({\mathbb{R}}^n)$, $1\le p<\infty$ and non uniform $L^q$ bounds 
on the propagators or estimates on its kernel (i.e. allowing the 
${\mathscr{L}}(L^q)$ norm of $\Gamma(t,s)$ to grow as $|t-s|\to \infty$), one 
also obtains well-posedness for $X=L^\infty(L^p)$, but the solution has $X_T$ 
norms that may grow at infinity. 

For $Y=L^p({\mathbb{R}}^n)$, $p_c<p<2$ and $BV(0,T; L^\infty)$ coefficients 
for all $T>0$, we are in the previous situation with  $X=L^\infty(L^p)$. 

We also note that the results $\|\nabla u\|_{T^{p,2}} \lesssim \|u\|_{X^p}$ 
in Section~\ref{sec7} are easily localisable in time by looking at the proofs. Thus 
the further estimates using truncated maximal functions apply and one can 
deduce bounds on truncated Lusin area integrals. However, note that the 
converse $ \|u\|_{X^p}\lesssim \|\nabla u\|_{T^{p,2}} $ is not localisable in 
time because the right hand side vanishes when the time becomes small. 

{\small
\bibliographystyle{amsplain}

}
\end{document}